\documentclass[a4paper,fleqn]{cas-sc}

\usepackage[numbers]{natbib}
\usepackage{subcaption}
\usepackage{indentfirst}
\usepackage{amsmath}
\usepackage{mathtools}
\usepackage{amsthm}

\usepackage{graphicx}
\usepackage{enumerate}
\usepackage{listings}
\usepackage{multicol}
\usepackage{booktabs}
\usepackage{float}

\usepackage{siunitx}
\usepackage{xcolor} 

\usepackage[ruled,vlined]{algorithm2e}
\usepackage{algpseudocode}

\usepackage{tikz} % Loaded once here
\usetikzlibrary{shapes.geometric,arrows.meta,positioning,fit,backgrounds,calc}
\usepackage{pgfplots}
\pgfplotsset{compat=1.18} 

\usepackage[labelsep=period]{caption}
\usepackage{subcaption} % Correctly loaded alongside caption

\newcommand{\vu}{\mathbf{u}}

\newcommand{\vx}{\mathbf{x}}

\newcommand{\vn}{\mathbf{n}}

\newcommand{\vX}{\mathbf{X}}

\newcommand{\vU}{\mathbf{U}}
\newcommand{\vV}{\mathbf{V}}

\newcommand{\vA}{\mathbf{A}}

\newcommand{\vD}{\mathbf{D}}

\newcommand{\vK}{\mathbf{K}}

\newcommand{\vI}{\mathbf{I}}
\newcommand{\vS}{\mathbf{S}}

\newcommand{\cA}{\mathcal{A}}
\newcommand{\cB}{\mathcal{B}}

\newcommand{\cE}{\mathcal{E}}
\newcommand{\cF}{\mathcal{F}}
\newcommand{\cG}{\mathcal{G}}

\newcommand{\cT}{\mathcal{T}}

\newcommand{\by}{\bar{y}}
\newcommand\doubleBrackets[1]{\lbrace\hspace*{-0.3em}\lbrace{#1}\rbrace\hspace*{-0.3em}\rbrace}
\newcommand\jump[1]{\left[\hspace*{-0.18em}\left[{#1}\right]\hspace*{-0.18em}\right]}

\newcommand{\bdf}{\mathsf{BDF}}
\newcommand{\ph}{P_w}
\newcommand{\sh}{S_w}

\newcommand{\Q}{\mathbb{Q}}

\newcommand{\R}{\mathbb{R}}

\newcommand{\Z}{\mathbb{Z}}   
   
\newcommand{\T}{\mathbb{T}}   
\newcommand{\I}{\mathbb{I}}

\newcommand{\SBP}{S^{\mathbf{BP}}}
\newcommand{\SOF}{S^{\mathbf{BP+OF}}}
\newcommand{\SOFD}{S^{\mathbf{BP+OF+D}}}

\newcommand{\dt}{\Delta t}

\newcommand{\grad}[1]{\nabla #1}

\definecolor{domainbg}{RGB}{240,245,255}       
\definecolor{blockgreen}{RGB}{60,160,60}        
\definecolor{blockgreenlight}{RGB}{180,230,180} 
\definecolor{hatchcol}{RGB}{30,100,180}         
\definecolor{labelblue}{RGB}{20,80,200}         
\definecolor{labelred}{RGB}{200,40,40}          
\definecolor{arrowcol}{RGB}{30,80,180}          

\tikzset{
  startstop/.style={
    rectangle, rounded corners=6pt,
    minimum width=6.5cm, minimum height=0.72cm,
    fill=gray!18, draw=black, line width=0.6pt, font=\small},
  process/.style={
    rectangle,
    minimum width=6.5cm, minimum height=0.72cm,
    fill=white, draw=black, line width=0.5pt,
    font=\small, text width=6.2cm, align=center},
  arrow/.style={-Stealth, line width=0.6pt, black},
  groupbox/.style={
    rectangle, rounded corners=5pt,
    draw=black, dashed, line width=0.55pt, fill=black!4}
}

\usepackage{tabularx}                                            % Specify total length of tabular.
  \newcolumntype{R}{>{\raggedleft\arraybackslash}X}
  \newcolumntype{L}{>{\raggedright\arraybackslash}X}
  \newcolumntype{C}{>{\centering\arraybackslash}X}
\usepackage{booktabs}                                            % Improve tabulars, enable \top-, \mid-, \bottomrule.
 
\def\tsc#1{\csdef{#1}{\textsc{\lowercase{#1}}\xspace}}
\tsc{WGM}
\tsc{QE}
\newtheorem{thm}{Theorem}[section]

\newtheorem{rmk}{Remark}

\begin{document}
\let\WriteBookmarks\relax
\def\floatpagepagefraction{1}
\def\textpagefraction{.001}

\shorttitle{An optimization-based property-preserving limiter}    

\shortauthors{S. Lee et al.}  

\title [mode = title]{An optimization-based property-preserving limiter for two-phase flow in porous media on non-matching meshes}  




\author[1]{Sanghyun Lee}%[]
\cormark[1]
\ead{slee17@fsu.edu}
\ead[url]{https://www.math.fsu.edu/~lee/}

\affiliation[1]{organization={Department of Mathematics, Florida State University},
            addressline={1017 Academic Way}, 
            city={Tallahassee},
%          citysep={}, % Uncomment if no comma needed between city and postcode
            postcode={32306}, 
            state={FL},
            country={USA}}

\author[1]{Seungmin Lee}%[<options>]
\ead{sl22b@fsu.edu}

\author[2]{Chen Liu}%[]

\affiliation[2]{organization={Department of Mathematical Sciences, University of Arkansas},
            addressline={850 West Dickson Street}, 
            city={Fayetteville},
            postcode={72701}, 
            state={AR},
            country={USA}}

\ead{chenl@uark.edu}

%\cortext[2]{Corresponding author}



\begin{abstract}
Two-phase flow in porous media arises in a wide range of applications in engineering and science. A fundamental challenge in its numerical simulation is that the approximation of the saturation transport equation can violate the maximum principle, and spurious oscillations can appear even when the solution remains within physical bounds. In this work, we propose a bound-preserving anti-oscillation postprocessing limiter for discontinuous Galerkin discretizations of the two-phase flow system on non-matching meshes. An optimization-based cell average limiter enforces the prescribed bounds on the cell averages while preserving global mass by minimizing the $L^2$ distance to the original cell averages. To suppress the remaining oscillations, the cell average limiter is combined with a minmod limiter, which is activated only in the cells identified by a troubled-cell detector to avoid unnecessary loss of accuracy, and a Zhang-Shu scaling limiter subsequently enforces the bounds on point values. In addition, we prove that the cell average limiter preserves the accuracy of the numerical solution at a fixed time on non-matching meshes. Numerical experiments on homogeneous, non-homogeneous, and fractured porous media demonstrate the accuracy and robustness of the proposed method.
\end{abstract}


%\item 
%\item 
%\item 

% Keywords
% Each keyword is separated by \sep
\begin{keywords}
  two-phase flow \sep discontinuous Galerkin \sep bound-preserving \sep optimization-based limiter \sep Davis-Yin splitting 
\end{keywords}

\maketitle

\section{Introduction}

Two-phase flow in porous media has been widely used to model flow in petroleum reservoirs~\citep{abou2013petroleum, chavent1986mathematical, morel1973two, peaceman2000fundamentals, slattery1970two, whitaker1986flow, kueper1991two,lee2018phase}, environmental engineering~\citep{bear2013dynamics, helmig1997multiphase}, biotechnology applications~\citep{teixeira2018emerging}, and chemical engineering~\citep{azbel1981two}. In a two-phase flow system, we consider the flow of a wetting phase and a non-wetting phase in a porous medium, with possibly heterogeneous characteristics.
The conventional two-phase flow problem consists of a pressure equation coupled with a saturation transport equation~\citep{chen2006computational}. Various numerical methods, including finite difference methods~\citep{abou2013petroleum,peaceman2000fundamentals}, finite element and mixed finite element methods~\citep{helmig1997multiphase, arbogast1992existence, arbogast2002implementation}, finite volume methods~\citep{radu2015robust}, and discontinuous Galerkin (DG) methods~\citep{epshteyn2007fully,hoteit2008efficient,riviere2008discontinuous}, have been employed to solve the two-phase flow problem. 

\par
A fundamental challenge in solving the two-phase flow problem is to ensure that the numerical solution of the saturation equation satisfies the physical bounds throughout the time evolution. 
Violation of this constraint not only yields unphysical solutions but also leads to numerical instabilities, particularly in advection-dominated regimes where sharp saturation fronts develop. Ensuring that the numerical solutions respect this bound constraint is therefore essential for the robustness and physical consistency of the simulation.

\par
Various approaches have been developed to construct bound-preserving schemes. 
For the explicit Euler method, under a suitable CFL condition, the cell average of a discrete solution can be decomposed into a convex combination of monotone first-order schemes at certain quadrature points~\citep{zhang2010maximum, zhang2010positivity}. This approach has been extended to convection-diffusion equations~\citep{sun2018discontinuous, srinivasan2018positivity} and compressible flow~\citep{zhang2017positivity}.
For the implicit Euler method, establishing monotonicity of the system matrix has been used to guarantee a discrete maximum principle~\citep{liu2023positivity, shen2021discrete, hu2023positivity, liu2024structure}. However, the monotonicity does not hold for high-order methods on unstructured meshes~\citep{hohn1981some}.
Another widely used strategy is flux limiting, which constructs a convex combination of a high-order scheme and a first-order bound-preserving scheme~\citep{kuzmin2009constrained, kuzmin2012flux, guermond2021second, guermond2019invariant}. For a comprehensive discussion of the advantages and disadvantages of these approaches, we refer to~\citep{liu2025optimization}.

\par
For two-phase flow in porous media specifically, we mention a non-exhaustive list of publications on bound-preserving DG methods~\citep{sarraf2024bound, joshaghani2022maximum, arbogast2013discontinuous,wang2025physics,chen2006computational,epshteyn2007fully}. 
In particular, the methods developed in \citep{sarraf2024bound, joshaghani2022maximum} enforce the saturation bounds through convex combinations involving cell averages from previous time steps and high-order DG fluxes. While effective for specific formulations, these approaches require problem-specific construction tailored to the structure of the two-phase flow problems.

\par
Optimization-based approaches offer a more flexible framework for enforcing bounds. A bound-preserving limiter for the spectral element method, implemented via standard quadratic solvers, is proposed in~\citep{guba2014optimization}. In~\citep{van2019positivity}, positivity constraints are incorporated within a KKT system for DG methods with implicit time-stepping, solved by an active set semismooth Newton method. Variational inequality constraints for continuous finite element discretizations of convection-diffusion equations are formulated in~\citep{kirby2024high}. However, the resulting optimization problems are often expensive to solve and require problem-specific solver design, which limits their applicability to large-scale computations. 

\par
Since the pressure and saturation equations are strongly coupled in the two-phase flow problem, solving the full system simultaneously can be computationally expensive. To decouple the system, we employ the Implicit Pressure Explicit Saturation (IMPES) method~\citep{ewing1983simulation}, which solves the pressure equation implicitly and the saturation equation explicitly. This sequential approach has been widely used in multiphase flow simulations~\citep{kou2010new, lee2018enriched, chen2004improved, lu2009iterative, arbogast2013discontinuous}.

\par
We incorporate the optimization-based bound-preserving limiter into our DG discretization of the saturation equation. Originally introduced in~\citep{liu2024simple} for phase-field models and subsequently extended to compressible Navier-Stokes equations~\citep{liu2024optimization}, Fokker-Planck equations~\citep{liu2025optimization}, and gas dynamics~\citep{liu2025efficient}, this limiter efficiently enforces physical bounds while preserving conservation and high-order accuracy. 
The optimization-based bound-preserving limiter computes modified cell averages by minimizing the $L^2$ distance to the original cell averages, ensuring that the limited solutions satisfy the prescribed bounds while preserving global mass.
The resulting minimization problem is solved efficiently via splitting methods~\citep{lions1979splitting, ryu2022large}, which enjoy asymptotic linear convergence~\citep{demanet2016eventual, liu2024simple}. The overall algorithm is straightforward to implement and requires only matrix-vector operations and a cut-off operation.  
Although various numerical tests have been presented to validate this approach on different equations, to our best knowledge, its application to porous media flows has not yet been studied in literature. In addition, a theoretical proof of the accuracy is currently available only for uniform meshes. In this paper, we extend it to non-matching meshes.

\par
In simulations of two-phase flow in porous media, spurious oscillations may persist in the vicinity of sharp saturation fronts, even when the discrete solution remains within physical bounds. 
While optimization-based bound-preserving limiters have been successfully constructed in other contexts, their use for two-phase porous media flow has not previously been combined with a dedicated mechanism to actively suppress spurious oscillations, a critical gap we address here with the integration of the minmod oscillation-free limiter.
We apply the minmod limiter~\citep{krivodonova2007limiters} after the optimization-based cell average limiter, which acts on the zeroth-order moment. By contrast, the minmod limiter operates on the higher-order moments and does not change the cell averages. 
To minimize excessive limiting and avoid degradation of the numerical accuracy, a troubled-cell detector~\citep{persson2006sub} is employed so that the minmod limiter is activated only in cells where oscillations are detected. Finally, a Zhang-Shu scaling limiter is employed to enforce bound preservation at the point-value level.
The resulting bound-preserving anti-oscillation framework unites rigorous bound constraints and robust oscillation suppression for DG discretizations of the two-phase flow problem.

\par
Another important issue is the treatment of non-matching meshes. Non-matching meshes frequently arise in porous-media simulations due to local mesh refinement, non-homogeneous geological structures, and domain decomposition techniques. To broaden the applicability of the method, we extend both the optimization-based cell average limiter and the minmod limiter to non-matching meshes. The proposed extension preserves global conservation, bounds, and oscillation-control properties across non-matching meshes. Furthermore, while previous studies established the proof of the accuracy of the optimization-based cell average limiter on uniform meshes~\citep{liu2024optimization,liu2025efficient,liu2024simple}, we extend the theoretical analysis to non-matching meshes, proving that the optimization-based cell average limiter preserves the accuracy of the numerical solution at a fixed time. The resulting method provides a robust and practical approach for constructing high-order DG approximations of two-phase flow problems on both uniform and non-matching meshes.

\par
The rest of this paper is organized as follows. In Section~\ref{governing equation}, we introduce the mathematical model of two-phase flow. In Section~\ref{numerical method}, we describe the DG scheme to solve the two-phase flow. In Section~\ref{limiter}, we explain the details of our bound-preserving anti-oscillation postprocessing procedure on uniform meshes. In Section~\ref{nonmatching-limiter}, we extend our algorithm to non-matching meshes. In Section~\ref{numerical experiments}, numerical tests are conducted to validate our algorithm, and conclusions are drawn in Section~\ref{conclusion}.

\section{Governing equations} \label{governing equation}
Let $\Omega$ be a bounded polygonal domain in $\R^2$, and $\I=(0,\T]$ be the computational time interval with $\T>0$. The two-phase flow is modeled by a system of mass balance equations for each phase $i$,
\begin{equation*}
    \frac{\partial}{\partial t}(\phi \rho_i s_i)+\nabla\cdot\vu_i=\rho_if_i,\quad i\in \{w, n\}, %\label{massbalance}
\end{equation*}
where $\phi$ is the porosity of the porous media, $\rho_i$ is the constant density, $s_i(\vx,t):\Omega\times\I\rightarrow\R$ is the saturation, and $i\in\{w,n\}$ indicates wetting (w) or non-wetting (n) phases, respectively. Here, $\vu_{i}(\vx,t):\Omega\times\I\rightarrow\R^2$ is the Darcy velocity for each phase $i$, given by
\begin{equation}
    \vu_i\coloneqq-\vK\frac{\rho_i k_i}{\mu_i}\nabla p_i. \label{Darcylaw}
\end{equation}
The matrix $\vK\coloneqq K_D(\vx)\vI$ appearing in \eqref{Darcylaw} is the isotropic absolute permeability matrix, where $K_D:\Omega\rightarrow\R$ satisfies $0<K_{\min}\leq K_D(\vx)\leq K_{\max}<+\infty$ with some constants $K_{\min}$ and $K_{\max}$, $\vI\in\R^{2\times2}$ is the identity matrix, $\mu_i$ is the viscosity, and $p_i(\vx,t):\Omega\times\I\rightarrow\R$ is the pressure for each phase. 
For simplicity, we neglect the effect of gravity. The relative permeability $k_i$ is a given function of saturation defined as $k_i\coloneqq k_i(s_w)$, whose explicit form is specified in Section~\ref{numerical experiments}.
We consider that all pores are filled with fluid, the following closure relations must be satisfied:
\begin{equation*}
s_w + s_n = 1 ~\text{and}~ 
p_c(s_w) = p_n-p_w,
\end{equation*}
where $p_c$ is the capillary pressure.
Throughout the paper, we choose the wetting phase pressure $p_w$ and the wetting phase saturation $s_w$ as the primary variables. Our slightly compressible two-phase flow model is given as follows:
\begin{align}
    \phi\frac{\partial}{\partial t}\left(c^F_w\rho_ws_wp_w\right)-\nabla\cdot(\vK\lambda_t\nabla p_w+\vK\lambda_n\nabla p_c)&=(\rho f)_t,\label{compressiblepressure}\\
    \phi\frac{\partial}{\partial t}(\rho_ws_w)+\nabla\cdot\vu_w&=\rho_w f_w,\label{saturation}
\end{align}
where 
\begin{align*}
    \lambda_i&\coloneqq\lambda_i(s_w)=\rho_i\frac{k_i(s_w)}{\mu_i},\ \ \text{phase mobility},\\ 
    \lambda_t&\coloneqq\lambda_t(s_w)=\lambda_w(s_w)+\lambda_n(s_w),\ \ \text{total mobility},\\
    (\rho f)_t&\coloneqq\rho_w f_w+\rho_n f_n,\\
    c_w^F&\coloneqq \text{compressibility coefficient.} 
\end{align*}

We assume that the boundary of domain $\Omega$ is partitioned into disjoint subsets $\partial\Omega=\Gamma_D\cup\Gamma_N =\Gamma_{-}\cup\Gamma_+\cup\Gamma_N$. $\Gamma_D$ is a Dirichlet boundary for the pressure equation, $\Gamma_{-}$ is an inflow boundary for the saturation equation, and $\Gamma_{+}$ is an outflow boundary. 
Let $\vn$ denote an outward normal vector to $\partial\Omega$. We define the inflow and outflow boundaries as
\begin{equation*}
\Gamma_{-} \coloneqq \{\vx\in\partial\Omega:\vu_w\cdot\vn\leq0\} ~\text{and}~ \Gamma_{+} \coloneqq \{\vx\in\partial\Omega:\vu_w\cdot\vn > 0\}.
\end{equation*}
We denote $\Gamma_N$ as the Neumann boundary of the domain. Then, for the flow problem, we impose
\begin{align*}
    p_w&=p_{D}\ \text{on}\ \Gamma_{D}\times\I,\\
    (\vu_w+\vu_n)\cdot\vn&=\vu_N\ \text{on}\ \Gamma_{N}\times\I,
\end{align*}
For the saturation system, we impose
\begin{equation*}
    s_w=s_{-}\ \text{on}\ \Gamma_{-}\times\I,
\end{equation*}
where $s_{-}$ is a given boundary value for saturation. Finally, we supplement the above system with initial conditions
\begin{equation*}
    s_w(\vx,0)=s_w(\vx) ~\text{and}~ p_w(\vx,0)=p_w(\vx),\quad\forall\vx\in\Omega.
\end{equation*}

\section{Numerical Methods}\label{numerical method}
In this section, we formulate the discrete system using a decoupled IMPES algorithm, namely treating pressure implicitly and saturation explicitly. For spatial discretization, we utilize the DG method for both pressure and saturation equations.

\subsection{Temporal discretization}
Given an integer $N_\mathrm{st}\geq 2$, we define a uniform partition of the time interval $0 \eqqcolon t^0 < t^1 < \cdots < t^{N_\mathrm{st}}\coloneqq \T$. Denote $\dt \coloneqq t^k-t^{k-1}$ for the time step size.
Let $\varphi^k=\varphi(t^k)$ denote a time-dependent function at the discrete time level $k$. We define the operator
\begin{equation*}
\bdf_m(\varphi^{k+1})\coloneqq
    \begin{cases}
        \frac{1}{\dt}(\varphi^{k+1}-\varphi^k), & m=1, \\[0.2cm]
        \frac{1}{2\dt}(3\varphi^{k+1}-4\varphi^k+\varphi^{k-1}), & m=2,
    \end{cases}
\end{equation*}
which corresponds to the backward differentiation formulas (BDF) of orders one and two. In this paper, we employ the second-order time marching scheme, namely $m=2$, to discretize the time derivatives for $k\geq 1$.
We have the following semi-discretized formula for the pressure equation
\begin{equation}
    \phi c^F_w\rho_w\bdf_m(s_w^{k+1}p_w^{k+1}) 
    -\nabla\cdot\left(\vK\lambda_t(s_w^{k+1})\nabla p^{k+1}_w\right)-\nabla\cdot\left(\vK\lambda_n(s_w^{k+1})\nabla p_c(s_w^{k+1})\right) = (\rho f)^{k+1}_t.\label{timeflow}
\end{equation}
And the saturation equation is discretized by
\begin{equation}
\phi\rho_w\bdf_m(s_w^{k+1}) + \nabla\cdot\vu^{k+1}_w = \rho_wf^{k+1}_w.\label{timesaturation}
\end{equation}

\paragraph{\textbf{Decoupling system}.}
Notice that the system \eqref{timeflow}-\eqref{timesaturation} is fully coupled and nonlinear. To facilitate more tractable computational expediency and algorithmic simplicity, we adapt the IMPES algorithm~\citep{ewing1983simulation} to decouple the equations and solve them sequentially. 
For uniform time steps, we define the extrapolation $\varphi^{k+1,\ast}$ to approximate the time-dependent variables by
\begin{equation*}
    \varphi^{k+1,\ast}\coloneqq\varphi^{k}+(\varphi^{k}-\varphi^{k-1}).
\end{equation*}
At time step $k\geq1$, the IMPES algorithm solves the two-phase flow system \eqref{timeflow}-\eqref{timesaturation} as follows: given $p_w^{k-1},p_w^{k},s^{k-1}_w$ and $s^{k}_w$,

\medskip
\textbf{Step 1.} Solve $p_w^{k+1}$ at time $t^{k+1}$ using $\lambda_i(s^{k+1,\ast})$ and $p_c(s^{k+1,\ast})$:
\begin{align*}
    \phi c^F_w\rho_w\bdf_m(s_w^{k+1,\ast} p_w^{k+1})
    - \nabla &\cdot \left(\mathbf{K}\lambda_t(s_w^{k+1,\ast})\nabla p^{k+1}_w\right) 
    = (\rho f)^{k+1}_t
    + \nabla\cdot\left(\mathbf{K}\lambda_n(s_w^{k+1,\ast})\nabla p_c(s_w^{k+1,\ast})\right).
\end{align*}

\medskip
\textbf{Step 2.} Solve $s_w^{k+1}$ using an explicit time stepping:
\begin{equation*}
    \phi\rho_w\bdf_m(s_w^{k+1})
    = \rho_w f^{k+1}_w
    - \nabla\cdot\vu^{k+1,\ast}_w.
\end{equation*}
To initialize the algorithm above, we use the initial condition $s_w^0$ and set $s_w^{1} = s_w^0$ at the first time step.

\subsection{Spatial discretization}
Let $\cT_h=\{E_i\}_{i=1}^N$ be a shape-regular quadrilateral partition of $\Omega$ consisting of square elements. 
We denote by $h_E$ the diameter of $E\in \cT_h$ and set $h=\max_{E\in\cT_h}h_E$. Note that $h_E = h$ for a uniform mesh. We denote by $\cE_h$ the set of all edges and by $\cE_h^I$ and $\cE_h^{\partial}$ the collection of all interior and boundary edges, respectively. 
For the pressure equation, the boundary edges $\cE_h^{\partial}$ are further decomposed into $\cE_h^{\partial} = \cE_h^{D,\partial}\cup\cE_h^{N,\partial}$. Here, $\cE_h^{D,\partial}$ is the collection of edges where the Dirichlet boundary condition is imposed; and $\cE_h^{N,\partial}$ is the collection of edges where the Neumann boundary condition is imposed. 
In addition, let $\cE_h^{1}\coloneqq\cE_h^{I}\cup\cE_h^{D,\partial}$ and $\cE_h^{2}\coloneqq\cE_h^{I}\cup\cE_h^{N,\partial}$. 
For the transport problem, the boundary edges $\cE_h^{\partial}$ are decomposed into $\cE_h^{\partial}=\cE_h^{-,\partial}\cup \cE_h^{+,\partial}$. Here, $\cE_h^{-,\partial}$ is the collection of edges where the inflow boundary condition is imposed; and $\cE_h^{+,\partial}$ is the collection of edges where the outflow boundary condition is imposed.

The space $H^s(\cT_h)$, $s\in\Z_+$, is the set of element-wise $H^s$ functions on $\cT_h$, and $L^2(\cE_h)$ refers to the set of functions whose traces on the elements of $\cE_h$ are square integrable. Let $\Q_{1}(E_i)$ denote the space of polynomials of partial degree at most $1$ defined on $E_i$.
We introduce the space of piecewise bilinear discontinuous polynomials as 
\begin{equation}
M^{1}(\cT_h) \coloneqq \{\psi \in L^2(\Omega):~ \psi|_{E}\in\Q_{1}(E),~ \forall E\in\cT_h \}.
\end{equation}
For any space $\vX$, $\vX^2$ denotes the space of size $2$ vectors, whose components belong to $\vX$; and $\vX^{2\times 2}$ denotes the space of $2\times 2$ matrices, whose components belong to $\vX$.
We employ the standard notation for Sobolev norms. 
On a subdomain $O\subset\Omega$, the $H^1$ norm and seminorm are denoted by $\|\cdot\|_{H^1(O)}$ and $|\cdot|_{H^1(O)}$, respectively. We denote the $L^2$ norm on $O$ by $\|\cdot\|_{L^2(O)}$. When $O = \Omega$, we suppress the set subscript and write $\|\cdot\|_{H^1}$ and $\|\cdot\|_{L^2}$, and analogously for the seminorm. 
In addition, we use $\|\cdot\|_2$ to denote the Euclidean norm.

\par
For any interior edge $e\in\cE^I_h$, let $E^-$ and $E^+$ be two neighboring elements such that $e = \partial E^-\cap \partial E^+$. We define a unit normal vector $\vn_e$ that points from $E^-$ to $ E^+$. For any boundary edge $e\in\cE_h^{\partial}$, the normal vector $\vn_e$ is taken to be the unit outward vector to $\partial\Omega$.
%For a given scalar function $x$ and vector function $\vx$, defined on the $\cT_h$, we denote $x^{\pm}$ and $\vx^{\pm}$ by the restrictions of $x$ and $\vx$ to $E^{\pm}$, respectively. 
%
We define the average $\doubleBrackets{\cdot}$ and jump $\jump{\cdot}$ of a scalar quantity $x$ across the interior edge $e\in\cE^I_h$ as follows:
\begin{equation*}
\doubleBrackets{x}\coloneqq\frac{1}{2}(x|_{E^-} + x|_{E^+}) 
\quad\text{and}\quad
\jump{x} \coloneqq x|_{E^-} - x|_{E^+}.
\end{equation*}
On a boundary edge $e \in \cE^{\partial}_h$, the average and jump coincide with its trace, i.e., we set $\doubleBrackets{x} \coloneqq x$ and $\jump{x} \coloneqq x$. 
The related definitions of any vector quantity are similar \citep{riviere2008discontinuous}.
For scalar inner products, we use the notations:
\begin{equation*}
    (v,w)_{\cT_h}\coloneqq \sum_{E\in\cT_h}{\int_E{vw\,dx}},\quad\forall v,w\in L^2(\cT_h),\quad
    \left<v,w\right>_{\cE_h}\coloneqq\sum_{e\in\cE_h}{\int_e vw\,d\gamma},\quad\forall v,w\in L^2(\cE_h).
\end{equation*}
The related inner product notations for vectors are similar.

\par
Throughout this paper, both the pressure and saturation equations are discretized using the incomplete interior penalty DG method~\citep{dawson2004compatible,sun2005symmetric}. In the following, we first introduce the DG formulation of the pressure equation. The wetting-phase mass flux is subsequently evaluated from the numerical approximation of $p_w^{k+1}$ and extrapolated saturation $s_w^{k+1,\ast}$. Finally, we introduce the DG formulation of the saturation equation.

\paragraph{\textbf{Discretization of the pressure equation}.}
Let us denote the spatial approximation of the wetting phase pressure $p_w(\vx,t)$ as $\ph(\vx,t)\in M^1(\cT_h)$ and let $\ph^k\coloneqq\ph(\vx,t^k)$ for time discretization, $0\leq k\leq N_\mathrm{st}$. 
At time $t^{k+1}$, let $p_D^{k+1}$ be approximations of $p_D(\cdot,t^{k+1})$ on $\Gamma_D$; let $\vu_{N}^{k+1}$ be an approximation of $\vu_{N}(\cdot,t^{k+1})$ on $\Gamma_N$; and let $f^{k+1}$ be an approximation of $f(\cdot, t^{k+1})$ in $\Omega$.
Then, the implicit DG formulation of the pressure equation can be written as follows: Given $\ph^{k-1}, \ph^k$, find
\begin{equation}
    \ph^{k+1}\in M^1(\cT_h) ~\text{such\ that}~ \cA(\ph^{k+1},w)=\cF(w),\quad\forall w\in M^1(\cT_h),\label{flow_weak}
\end{equation}
where $\cA$ and $\cF$ are the bilinear form and the linear functional, respectively, defined as
\begin{align*}
    \cA(\ph^{k+1}, w)\coloneqq&
    \left((\phi\rho_w c^F_w s_w^{k+1,\ast})\frac{3}{2\dt}\ph^{k+1},w\right)_{\cT_h}+
    \left(\vK \lambda_t(s^{k+1,\ast}_w)\nabla\ph^{k+1},\nabla w\right)_{\cT_h}\\&-\left<\doubleBrackets{\vK \lambda_t(s^{k+1,\ast}_w) \nabla\ph^{k+1}\cdot\vn_e},\jump{w}\right>_{\cE_h^{1}}+\left<\frac{\sigma}{h_e}\jump{\ph^{k+1}},\jump{w}\right>_{\cE_h^{1}},
\end{align*}
and
\begin{align*}
\cF(w)\coloneqq&
\left((\phi\rho_w c^F_w s_w^{k})\frac{2}{\dt}\ph^{k}-(\phi\rho_w c^F_w s_w^{k-1})\frac{1}{2\dt}\ph^{k-1},w\right)_{\cT_h}+\left((\rho f)^{k+1}_t,w\right)_{\cT_h}\\
&-\left(\vK\lambda_n(s_w^{k+1,\ast})\nabla p_c(s_w^{k+1,\ast}),\nabla w\right)_{\cT_h}+\left<\doubleBrackets{\vK\lambda_n(s^{k+1,\ast}_w)\nabla p_c(s_w^{k+1,\ast})\cdot\vn_e},\jump{w}\right>_{\cE_h^1}\\
&-\left<\frac{\sigma}{h_e}\jump{p_c(s^{k+1,\ast}_w)},\jump{w}\right>_{\cE^1_h}+\left<\frac{\sigma}{h_e}p_{D}^{k+1},w\right>_{\cE_h^{D,\partial}}-\left<\vu_N^{k+1},w\right>_{\cE^{N,\partial}_h}.
\end{align*}
Here, $h_e=\min{(h_{E^-},h_{E^+})}$ and $\sigma>0$ is the penalty parameter.

\paragraph{\textbf{Discretization of the saturation equation}.}
Let $s_{-}^{k+1}$ be an approximation of $s_{-}(\cdot,t^{k+1})$ on $\Gamma_{-}$ at time $t^{k+1}$. The numerical approximation of the wetting phase saturation $s_w(\vx,t)$ is denoted by $\sh(\vx,t)\in M^{1}(\cT_h)$ and we let $\sh^k\coloneqq\sh(\vx,t^k)$ for time discretization. With $\ph^{k+1}$ computed by \eqref{flow_weak}, the explicit DG formulation of the saturation equation is written as follows: Given $\sh^{k-1}, \sh^{k}$, find
\begin{equation*}
\sh^{k+1}\in M^{1}(\cT_h) ~\text{such~that}~ \cB(\sh^{k+1},\psi)=\cG(\psi),\quad\forall\psi\in M^{1}(\cT_h), %\label{saturation_weak}
\end{equation*}
where,
\begin{align*}
\cB(\sh^{k+1},\psi)=\left(\phi\rho_w\frac{3}{2\dt}\sh^{k+1},\psi\right)_{\cT_h},
\end{align*}
and
\begin{align*}
\cG(\psi) = \left(\phi\rho_w\frac{2\sh^{k}}{\dt}-\phi\rho_w\frac{\sh^{k-1}}{2\dt},\psi\right)_{\cT_h} + (\rho_w f_w^{k+1},\psi)_{\cT_h} + (\vU_w^{k+1,\ast},\nabla\psi)_{\cT_h}-\left<\vU^{k+1,\ast}_w\cdot\vn_e,\jump{\psi}\right>_{\cE_h}.
\end{align*}

\paragraph{\textbf{Computation of $U^{k+1,\ast}_w$.}} With the computed spatial approximation of the wetting phase pressure $P_w^{k+1}$, we define the wetting phase velocity as $\vU^{k+1,\ast}_w$ which is a globally and locally conservative flux variable. The details for convergence analysis of conservative flux variables can be found in~\citep{lee2016locally,sun2009locally}. With the extrapolated saturation value $s_w^{k+1,\ast}$, we define the conservative flux variables $\vU_w^{k+1,\ast}$ at time step $t^{k+1}$ as follows: 
\begin{align*}
\vU^{k+1,\ast}_w|_E&\coloneqq-\vK\lambda_w(s_w^{k+1,\ast})\grad P_w^{k+1},\quad \forall E\in\cT_h,\\ %\label{flux1}
\vU^{k+1,\ast}_w|_e\cdot\vn_e&\coloneqq-\left(\lambda_w(s_w^{k+1,\ast})\right)^{\uparrow}\doubleBrackets{\vK\grad P_w^{k+1}}\cdot\vn_e+\frac{\sigma}{h_e}\left(\frac{\lambda_w(s_w^{k+1,\ast})}{\lambda_t(s_w^{k+1,\ast})}\right)^{\uparrow}\jump{P_w^{k+1}},\quad \forall e\in\cE^I_h, \\ %\label{flux2}
\vU^{k+1,\ast}_w|_e\cdot\vn_e&\coloneqq\frac{\lambda_w(s_w^{k+1,\ast})}
     {\lambda_t(s_w^{k+1,\ast})}
(\vu_N^{k+1}+\mathbf{K}\lambda_n(s_w^{k+1,\ast})\nabla p_c(s_w^{k+1,\ast})\cdot\vn_e),\quad \forall e\in\cE^{N,\partial}_h,\\ %\label{flux3}
\vU^{k+1,\ast}_w|_e\cdot\vn_e&\coloneqq-\left(\vK\lambda_w(s_w^{k+1,\ast})\grad P_w^{k+1}\right)\cdot\vn_e+\frac{\sigma}{h_e}\left(\frac{\lambda_w(s_w^{k+1,\ast})}{\lambda_t(s_w^{k+1,\ast})}\right)(P_w^{k+1}-p^{k+1}_{D}),\quad \forall e\in\cE^{D,\partial}_h. %\label{flux4}
\end{align*}
Here, we define the upwind operator $(\cdot)^{\uparrow}$ as:
\begin{equation}
    \forall e\in\partial E_i^{+}\cap\partial E_i^{-},\quad\xi^{\uparrow}|_e=\begin{cases}
        \xi|_{E_i^{-}},\ &\mathrm{if}\ \doubleBrackets{\vV^{k+1,\ast}_w} > 0,\\
        \xi|_{E_i^{+}},\ &\mathrm{if}\ \doubleBrackets{\vV^{k+1,\ast}_w} \leq 0, \nonumber
    \end{cases}
\end{equation}
where $\xi$ is any interface quantity and $\doubleBrackets{\vV_w^{k+1,\ast}} = -\doubleBrackets{\lambda_w(s_w^{k+1,\ast})\vK\grad P_w^{k+1}}\cdot\vn_e+\frac{\sigma}{h_e}\doubleBrackets{\frac{\lambda_w(s_w^{k+1,\ast})}{\lambda_t(s_w^{k+1,\ast})}}\jump{P_w^{k+1}}$. 

To start the BDF2 time marching for our fully discrete scheme, we set $S_w^0$ and $P_w^0$ by pointwise interpolation of the initial data, and simply replicate these values for the first time level, namely $S_w^1 = S_w^0$ and $P_w^1 = P_w^0$.

\section{Postprocessing procedure on uniform meshes}\label{limiter}

When solving the saturation equation \eqref{saturation}, it is essential to preserve several properties of the solution. In particular, the numerical approximation should satisfy the discrete maximum principle, both to yield a physically meaningful solution and to guarantee numerical stability. In this paper, we employ the optimization-based bound-preserving limiter introduced in~\citep{liu2024simple}. 
\par
Notice, although our bound-preserving limiter effectively enforces the maximum principle, it does not directly address spurious oscillations that may arise near sharp transitions. Such oscillations are frequently encountered in two-phase flow simulations and can affect the accuracy of the solution even when the computed saturation remains within prescribed bounds. To resolve this issue, we incorporate an additional anti-oscillation step based on the minmod limiter~\citep{krivodonova2007limiters}. 
\par
Our bound-preserving anti-oscillation postprocessing procedure consists of sequentially applying limiters. First, an optimization-based cell-average limiter enforces the zeroth-order moment of the discrete solution in the admissible set. Second, a minmod limiter is applied to eliminate oscillatory behavior that may remain. Finally, the Zhang-Shu scaling limiter adjusts the higher-order moments to ensure point values remain in bounds. The details are presented in the following sections.

\subsection{Optimization-based bound-preserving limiter on uniform meshes} \label{BPlimiter-uniform}
We start by presenting an efficient algorithm for the optimization-based cell average limiter. 
Recall $\sh^{k+1}$ is the DG approximation of $s_w(\vx,t^{k+1})$. We define a piecewise constant function $\overline{S}_w^{k+1}$ whose value in each cell $E_i$ represents the cell average of $S_w^{k+1}$, namely $\overline{S}_w^{k+1}|_{E_i} = \frac{1}{|E_i|}\int_{E_i}{S_w^{k+1}}dx$. With a desired range interval $[m,M]$, we find a piecewise constant function $\chi_w$ that minimizes the $L^2$ distance to $\overline{S}_w^{k+1}$ under the constraints of preserving conservation and bounds:
\begin{equation}
\min_{\chi_w}{\|\chi_w-\overline{S}_w^{k+1}\|_{L^2}^2}\ \ \text{subject~to}\ \ \int_{\Omega}{\chi_w\,dx}=\int_{\Omega}{S_w^{k+1}\,dx}\ \ \text{and}\ m \leq \chi_w\leq M. \label{constraints}
\end{equation}

% \subsubsection{DR splitting}
We denote the solution of the minimization problem \eqref{constraints} by $\overline{\chi}_w$. To find the minimizer $\overline{\chi}_w$ effectively, we apply the generalized Douglas-Rachford (DR) splitting method. 
We refer to~\citep{demanet2016eventual,zhang2026ubiquitous} and the references therein for details on the equivalences of generalized DR method and popular splitting methods, including PDHG~\citep{chambolle2016introduction}, ADMM~\citep{fortin2000augmented}, and the split Bregman method~\citep{goldstein2009split}.
\par
Let us rewrite \eqref{constraints} in matrix-vector form. Define a matrix $\mathbf{1}=[1,1,\cdots,1]^T\in\R^{N\times 1}$, where $N$ is the total number of active cells. A vector $\omega\in\R^N$ is the vector that stores the cell average of $S_w^{k+1}$, namely $\omega_i = \overline{S}_w^{k+1}|_{E_i}$. Then, with $\zeta\in\R^N$, we define a cut-off operator $S(\zeta):\R^N\rightarrow\R^N$ by
\[[S(\zeta)]_i = \min{(\max{(\zeta_i,m)},M)},\quad \forall i = 1,\cdots,N.\] 
The algorithm of DR splitting~\citep{lions1979splitting} to find the minimizer of \eqref{constraints} is as follows: set $y^0 = \omega$, $x^0 = S(\omega)$, $b = \mathbf{1}^T\omega$, and $l=0$. With appropriate constants $c$, $\lambda$, and convergence tolerance $\varepsilon$,

\begin{itemize}
\item[] \textbf{Step 1.} Compute $z^l = 2x^l - y^l$
\item[] \textbf{Step 2.} Compute $y^{l+1} = \lambda c \left(z^l + \frac{1}{N}\left(b - \sum z^l_i\right)\mathbf{1}\right) + \lambda(1-c)\omega + y^l - \lambda x^l$
\item[] \textbf{Step 3.} Compute $x^{l+1} = S(y^{l+1})$
\item[] \textbf{Step 4.} If $h\|y^{l+1} - y^l\|_2 < \varepsilon$, terminate and output $\overline{\chi}_w = x^{l+1}$. Otherwise, set $l \leftarrow l+1$ and go to \textbf{Step 1}.
\end{itemize}
The nearly optimal parameters $c$ and $\lambda$ for the algorithm above are specified as follows~\citep{liu2024simple}. Let $N_{\hat{r}}$ be the number of bad cells for which $\overline{S}_w \notin [m,M]$ and let $\hat{\theta}=\cos^{-1}{\sqrt{\frac{N_{\hat{r}}}{N}}}$, then we have
\begin{equation*}
\begin{cases}
c=\frac{1}{2},\ \lambda = \frac{4}{2-\cos{(2\hat{\theta})}},\qquad &\mathrm{if}\ \hat{\theta}\in(\frac{3}{8}\pi,\frac{1}{2}\pi],\\
c=\frac{1}{(\cos\hat{\theta}+\sin\hat{\theta})^2},\ \lambda = \frac{2}{1+\frac{1}{1+\cot(\hat{\theta})}-\frac{1}{(\cos\hat{\theta}+\sin\hat{\theta})^2}},\qquad&\mathrm{if}\ \hat{\theta}\in(\frac{1}{4}\pi,\frac{3}{8}\pi],\\
c=\frac{1}{(\cos\hat{\theta}+\sin\hat{\theta})^2},\ \lambda = 2,\qquad&\mathrm{if}\ \hat{\theta}\in(0,\frac{1}{4}\pi].
\end{cases}
\end{equation*}
After computing $\overline{\chi}_w$, we postprocess the solution $S_w^{k+1}$ to be 
\begin{equation}
\hat{S}^{k+1}_w = (S_w^{k+1}-\overline{S}_w^{k+1})+\overline{\chi}_w. \label{postprocess}
\end{equation}

Note that only the cell-wise average of $\hat{S}_w^{k+1}$ is in $[m,M]$. The point values of $\hat{S}_w^{k+1}$ may still contain overshoots or undershoots. To eliminate these out-of-bounds point values, we apply the following Zhang-Shu limiter after the minmod limiter described in Section~\ref{sec:minmod}. 
Let $\dot{S}_w^{k+1}$ denote the cell average of $\hat{S}_w^{k+1}$ on a cell $E\in\cT_h$. 
Define $M'\coloneqq\max_{E}{\hat{S}_w^{k+1}}$ and $m'\coloneqq\min_{E}{\hat{S}_w^{k+1}}$. Then, the bound-preserving solution $\tilde{S}_w^{k+1}$ is given by
\begin{equation}
\tilde{S}_w^{k+1} = \theta(\hat{S}_w^{k+1} - \dot{S}_w^{k+1}) + \dot{S}_w^{k+1},
~\text{where}~
\theta=\min{\left(\left|\frac{M-\dot{S}_w^{k+1}}{M'-\dot{S}_w^{k+1}}\right|,\left|\frac{m-\dot{S}_w^{k+1}}{m'-\dot{S}_w^{k+1}}\right|\right)}.\label{zslimiter}
\end{equation}

The accuracy of the optimization-based cell average limiter \eqref{constraints} at a fixed time step is demonstrated in~\citep{liu2025optimization}. We extend this algorithm to non-matching meshes, as presented in Section~\ref{nonmatching-limiter}.

\subsection{Minmod limiter on uniform meshes}\label{sec:minmod}
As discussed in Section~\ref{BPlimiter-uniform}, the optimization-based bound-preserving limiter ensures that the numerical solution satisfies the discrete maximum principle; however, it does not fully eliminate oscillations. Such oscillations can deteriorate the accuracy and stability of the numerical solution. To address this issue, we incorporate the minmod limiter, introduced in~\citep{krivodonova2007limiters}, which suppresses oscillatory behavior while retaining the high-order accuracy of the solution as much as possible.

Let $\{P_{00},P_{10},P_{01},P_{11}\}$ denote the orthonormal $\Q_1$ Legendre basis on the reference element $[-1,1]^2$, defined by
\begin{equation*}
P_{00}(x,y)=\frac{1}{2},\quad 
P_{10}(x,y)=\frac{\sqrt{3}}{2}x,\quad 
P_{01}(x,y)=\frac{\sqrt{3}}{2}y,\quad 
P_{11}(x,y)=\frac{3}{2}xy.
\end{equation*}
Then, for a cell $E_i\in\cT_h$, $\hat{S}_w^{k+1}$ restricted to $E_i$ is represented by 
\begin{equation}
    \hat{S}_{w}^{k+1}|_{E_i} = s_{00}^iP_{00}(x,y)+s_{10}^iP_{10}(x,y)+s_{01}^iP_{01}(x,y)+s_{11}^iP_{11}(x,y), \label{legendre}
\end{equation}
where $s_{00}^i$, $s_{10}^i$, $s_{01}^i$, and $s_{11}^i$ are the modal coefficients associated with the cell $E_i$. Let $E_{i,L}$, $E_{i,R}$, $E_{i,B}$, and $E_{i,T}$ denote the neighboring cells of $E_i$ in the left, right, bottom, and top directions, respectively. The stencil employed in applying the minmod limiter to the target cell $E_i$ is shown in Figure~\ref{uniform-minmod-stencil}.

\begin{figure}[ht!]
\centering
\begin{tikzpicture}[scale=1., line join=round, line cap=round]

\definecolor{cellgreen}{RGB}{145,235,145}
\definecolor{cellcyan}{RGB}{175,245,245}

\tikzset{
  mesh/.style={draw=black, line width=0.9pt},
  thickmesh/.style={draw=black, line width=1.4pt},
  arrow/.style={-latex, line width=1.1pt}
}

\draw[thickmesh] (0,0) rectangle (3,3);

\draw[mesh] (1,0) -- (1,3);
\draw[mesh] (2,0) -- (2,3);
\draw[mesh] (0,1) -- (3,1);
\draw[mesh] (0,2) -- (3,2);

\fill[cellgreen] (1,2) rectangle (2,3);
\fill[cellgreen] (1,0) rectangle (2,1); 
\fill[cellgreen] (0,1) rectangle (1,2); 
\fill[cellgreen] (2,1) rectangle (3,2);

\fill[cellcyan] (1,1) rectangle (2,2);

\draw[thickmesh] (0,0) rectangle (3,3);
\draw[mesh] (1,0) -- (1,3);
\draw[mesh] (2,0) -- (2,3);
\draw[mesh] (0,1) -- (3,1);
\draw[mesh] (0,2) -- (3,2);

\coordinate (C) at (1.5,1.5);
\node at (1.5,1.5) {$\mathbf{E_i}$};

\draw[arrow] (1.2,1.5) -- (0.2,1.5);
\node at (0.7,1.68) {$\mathbf{E_{i,L}}$};

\draw[arrow] (1.75,1.5) -- (2.75,1.5);
\node at (2.3,1.68) {$\mathbf{E_{i,R}}$};

\draw[arrow] (1.5,1.8) -- (1.5,2.8);
\node at (1.75,2.35) {$\mathbf{E_{i,T}}$};

\draw[arrow] (1.5,1.3) -- (1.5,0.3);
\node at (1.75,0.65) {$\mathbf{E_{i,B}}$};

% \draw[arrow] (1.0,1.5) -- (1.1,1.5);
% \draw[arrow] (2.1,1.5) -- (1.7,1.5);
% \draw[arrow] (1.5,2.1) -- (1.5,1.8);
% \draw[arrow] (1.5,0.9) -- (1.5,0.8);

% \filldraw[black] (C) circle (1.5pt);

\end{tikzpicture}
\caption{Illustration of a 2D stencil used for limiting on a uniform mesh. The cell \(E_i\) is shown in cyan, and its left, right, bottom, and top neighboring cells are used in the minmod limiter.}
\label{uniform-minmod-stencil}
\end{figure}
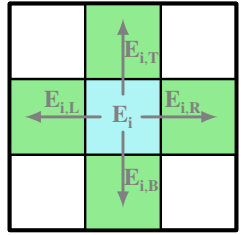

Here, we define a minmod function as
\begin{equation*}
\mathrm{minmod}(a_1,\cdots,a_n) = \begin{cases}
\mathrm{sgn}(a_1)\min_{i}{(|a_i|)},\ &\mathrm{if}\ \mathrm{all\ }a_i\ \mathrm{have\ the\ same\ sign},\\
0,\ &\mathrm{otherwise}.
\end{cases}
\end{equation*}

We first limit the highest-order coefficient $s^i_{11}$. We define
\begin{equation}
    \tilde{s}_{11}^i=\mathrm{minmod}(s_{11}^i,\ \alpha(s^{i,T}_{10}-s^i_{10}),\ \alpha(s^{i}_{10}-s^{i,B}_{10}),\ \alpha(s^{i,R}_{01}-s^i_{01}),\ \alpha(s^{i}_{01} -s^{i,L}_{01})),\label{minmod-higher}
\end{equation}
where $s^{i,\cdot}_{\cdot,\cdot}$ represents the modal coefficients in the cell $E_{i,\cdot}$. If $\tilde{s}_{11}^i=s_{11}^i$, then no further limiting is applied and the remaining coefficients are left unchanged. Otherwise, we replace $s_{11}^{i}$ with $\tilde{s}_{11}^i$, and the lower-order coefficients are limited according to the following:

\begin{equation}
\tilde{s}_{10}^i=\mathrm{minmod}(s_{10}^i,\ \alpha(s^{i,R}_{00}-s^i_{00}),\ \alpha(s^{i}_{00}-s^{i,L}_{00})),\quad \tilde{s}_{01}^i=\mathrm{minmod}(s_{01}^i,\ \alpha(s^{i,T}_{00}-s^i_{00}),\ \alpha(s^{i}_{00}-s^{i,B}_{00})),\label{minmod-lower}
\end{equation}
Note that the $\alpha$ is a constant depending on the order of the coefficient. The following choice of $\alpha$ is proposed:
\[\frac{1}{2\sqrt{3}}\leq\alpha  \leq \frac{1}{\sqrt{3}}.\]
Choosing $\alpha$ outside of this range resulted in either loss of accuracy or numerical instability~\citep{krivodonova2007limiters}. In the numerical experiments presented in Section \ref{numerical experiments}, we choose the smallest $\alpha$ value, i.e., the most diffusive choice.

\paragraph{\textbf{Troubled-cell detector}.}
When applying the minmod limiter, a uniform constant $\alpha$ is used across all cells to limit the coefficients. However, a more careful treatment is required in regions containing spurious oscillations or sharp saturation transitions. To tackle this problem, we employ the troubled-cell detector introduced in~\citep{persson2006sub}, which is designed to identify high-frequency components and sharp variations in the numerical solution. Furthermore, restricting the minmod limiter to only the detected troubled cells reduces the unnecessary loss of accuracy caused by limiting. 
    
We consider a truncated expansion of the same solution, only containing the lower-order terms, i.e.,
\begin{equation*}
{S}_w^{k+1,-}|_{E_i} = s_{00}^iP_{00}(x,y) + s_{10}^iP_{10}(x,y) + s_{01}^iP_{01}(x,y).
\end{equation*}
Then, in the cell $E_i$, the following smoothness indicator is defined
\begin{equation*}
\delta_{i} = \frac{\|\hat{S}_{w}^{k+1}|_{E_i} - {S}_{w}^{k+1,-}|_{E_i}\|_{L^2(E_i)}^2}{\|\hat{S}_{w}^{k+1}|_{E_i}\|_{L^2(E_i)}^2}.
\end{equation*}
Since the basis functions are orthonormal, the $L^2$ norms above are simply evaluated as the Euclidean norm of their expansion coefficients, namely
\begin{equation}\label{detector}
\delta_i = \frac{(s_{11}^{i})^2}{\sum_{l,m=0}^1{(s_{lm}^{i})^2}}.
\end{equation}

By looping all cells in the domain, if $\delta_{i}\geq\varepsilon_i$, where $\varepsilon_i$ is a threshold, we mark the cell as a troubled cell and apply the minmod limiter to the cell.
\par
Finally, we illustrate the full algorithm of our bound-preserving anti-oscillation postprocessing procedure in Figure~\ref{fig:algorithm-flowchart}.

\begin{figure}[ht!]
\centering
\resizebox{\textwidth}{!}{%
\begin{tikzpicture}[node distance=0.55cm]

% --- Init node (centered above both columns) ---
\node (init) [startstop, minimum width=8cm]
  {Initialize: solve \eqref{compressiblepressure}--\eqref{saturation} for $S_w^{k+1}$};

% --- left column (optimization-based cell average limiter) ---
\node (step1) [process, below=1.5cm of init, xshift=-4cm]
  {\textbf{Step 1.} Compute cell-average vector $\overline{S}_w^{k+1}$};
\node (step2) [process, below=of step1]
  {\textbf{Step 2.} Solve optimization problem \eqref{constraints}; obtain $\overline{\chi}_w$};
\node (step3) [process, below=of step2]
  {\textbf{Step 3.} Set $\hat{S}_w^{k+1} = \bigl(S_w^{k+1} - \overline{S}_w^{k+1}\bigr) + \overline{\chi}_w$};

% --- right column ---
\node (step4) [process, below=1.5cm of init, xshift=4cm]
  {\textbf{Step 4.} Convert $\hat{S}_w^{k+1}$ to Legendre basis};
\node (step5) [process, below=of step4]
  {\textbf{Step 5.} Mark troubled cells via smoothness indicator \eqref{detector}};
\node (step6) [process, below=of step5]
  {\textbf{Step 6.} Apply minmod limiter (Section~\ref{sec:minmod}) to troubled cells};
\node (step7) [process, below=of step6]
  {\textbf{Step 7.} Convert limited solution back to Lagrange basis};

% --- left column (Zhang-Shu limiter) ---
\node (step8) [process, below=1.35cm of step3]
  {\textbf{Step 8.} Apply Zhang-Shu limiter \eqref{zslimiter}; obtain $\tilde{S}_w^{k+1}$};

% --- Final node (centered below both columns) ---
\node (step9) [startstop, minimum width=8cm,
  below=1.8cm of step7, xshift=-3.2cm]
  {\textbf{Step 9.} Set $S_w^{k+1} = \tilde{S}_w^{k+1}$; advance to $t^{k+2}$};

% --- Group boxes ---
\begin{scope}[on background layer]
  \node[groupbox, fit=(step1)(step2)(step3), inner sep=7pt] (boxB) {};
\end{scope}
\begin{scope}[on background layer]
  \node[groupbox, fit=(step4)(step5)(step6)(step7), inner sep=7pt] (boxA) {};
\end{scope}
\begin{scope}[on background layer]
  \node[groupbox, fit=(step8), inner sep=7pt] (boxC) {};
\end{scope}

% --- Box labels ---
\node[font=\scriptsize, anchor=south] at ($(boxB.north)+(-0.025,0)$)
  {Optimization-based cell average limiter};
\node[font=\scriptsize, anchor=south] at ($(boxA.north)+(0.0,0)$)
  {minmod limiter};
\node[font=\scriptsize, anchor=south] at ($(boxC.north)+(0.0,0)$)
  {Zhang-Shu scaling limiter};

% --- Arrows ---
% Init -> Step 1: down then left into step1.north
\coordinate (initdrop) at ($(init.south) + (0,-0.6)$);
\coordinate (step1top)  at ($(step1.north) + (0, 0.9)$);
\draw[arrow] (init.south) -- (initdrop) -- (step1top) -- (step1.north);

\draw[arrow] (step1) -- (step2);
\draw[arrow] (step2) -- (step3);

% Step 3 -> Step 4: right out of step4.east, up, right, into step4.west
\draw[arrow] (step3.east)
  -- ++(0.7, 0)
  -- ++(0, 1.05)
  -- ($(step4.west) + (-0.8, 0)$)
  -- (step4.west);

\draw[arrow] (step4) -- (step5);
\draw[arrow] (step5) -- (step6);
\draw[arrow] (step6) -- (step7);

% Step 8 -> Step 9: down then left into step9.north
\coordinate (s7drop) at ($(step7.west) + (-0.6,0)$);
\coordinate (s8top)  at ($(step8.east) + (0.9, 0)$);
\draw[arrow] (step7.west) -- (s7drop) -- (s8top) -- (step8.east);

\coordinate (s8drop) at ($(step8.south)+(0,-0.6)$);
\coordinate (s9top) at ($(step9.north)+(0,0.5)$);
\draw[arrow] (step8.south) -- (s8drop) -- (s9top) -- (step9.north);

\end{tikzpicture}%
}
\caption{Flowchart illustrating the bound-preserving anti-oscillation postprocessing procedure.}
\label{fig:algorithm-flowchart}
\end{figure}
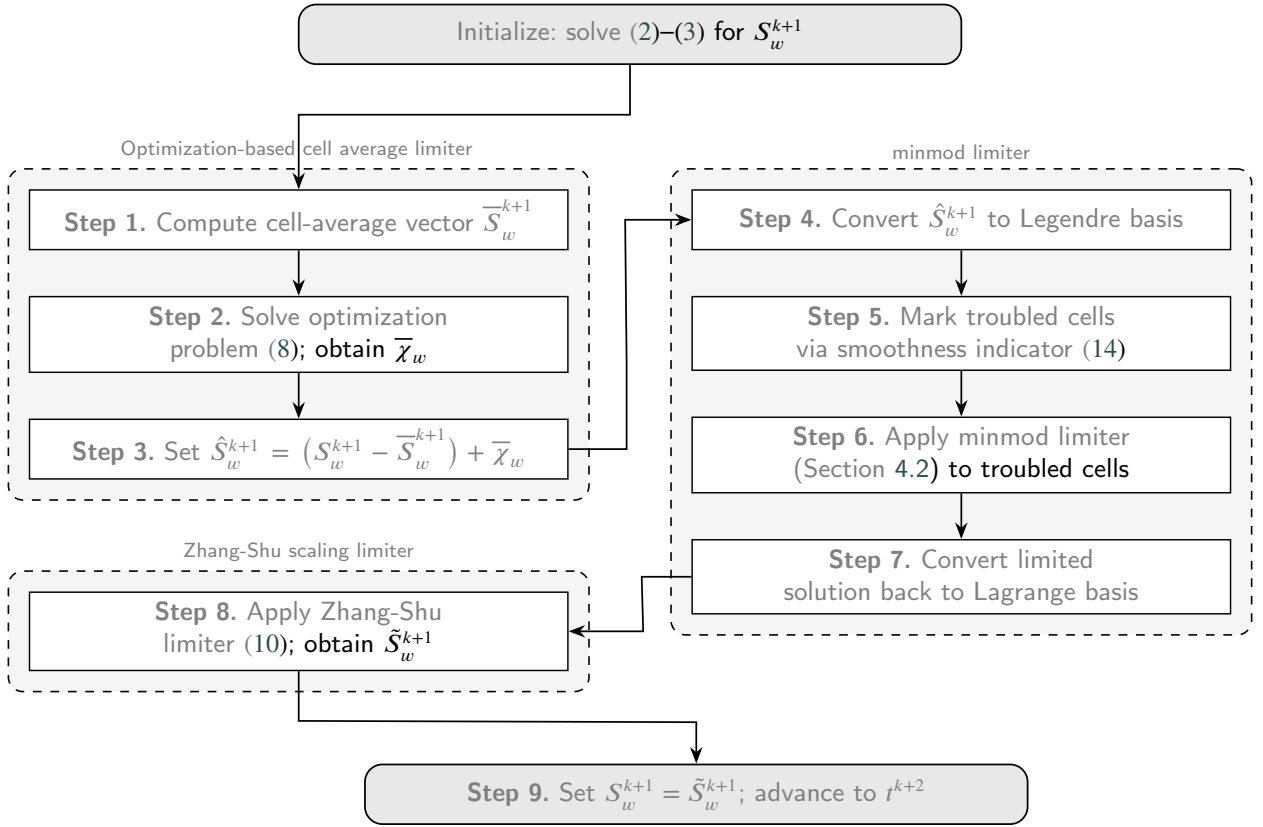

\section{Postprocessing procedure on non-matching meshes}\label{nonmatching-limiter}

In this section, we extend the discussion of Section~\ref{limiter} to non-matching meshes. Non-matching meshes arise in local mesh refinement, adaptive simulations, and multiscale problems.
Extending the limiter to non-matching meshes is essential for porous-media simulations. While the optimization-based bound-preserving limiter and its accuracy analysis have previously been studied on uniform meshes~\citep{liu2024optimization,liu2024simple,liu2025efficient,liu2025optimization}, their extension to non-matching meshes is nontrivial due to the irregular mesh connectivity and varying cell sizes. We therefore reformulate both the optimization-based bound-preserving limiter and the minmod limiter to operate consistently on non-matching meshes. 
In addition, we establish an accuracy result for the optimization-based cell average limiter on non-matching meshes. 
\par
In this paper, we do not assume the dynamic mesh adaptivity; instead we focus on the given non-matching mesh and the implementation of the limiters, which we present in the following subsections.

\subsection{Optimization-based bound-preserving limiter on non-matching meshes}
Suppose that the mesh partition $\cT_h = \{E_i\}_{i=1}^N$ is not uniform. Let us recall the numerical setups of the optimization problem in Section \ref{BPlimiter-uniform}. Denote $S_w^{k+1}$ as the numerical approximation of $s_w(\vx,t^{k+1})$ and $\overline{S}_w^{k+1}$ as a piecewise constant function that $\overline{S}_w^{k+1}|_{E_i} = \frac{1}{|E_i|}\int_{E_i}{S_w^{k+1}} dx$. Then, with prescribed bounds $[m,M],$ we find a minimizer $\overline{\chi}_w$ that satisfies 
\begin{equation}\label{constraints-ch5.1}
\min_{\chi_w}{\|\chi_w-\overline{S}_w^{k+1}\|_{L^2}^2}\ \ \mathrm{subject\ to}\ \ \int_{\Omega}{\chi_w\,dx}=\int_{\Omega}{S_w^{k+1}\,dx}\ \ \mathrm{and}\ m \leq \chi_w\leq M.
\end{equation}
Now, we define a matrix $\vA=[|E_1|,\cdots,|E_N|] \in \R^{1\times N}$ which stores the volume of each cell. Let vector $\omega\in\R^N$ store the cell average of $S^{k+1}_w$, i.e., the $i$-th entry of $\omega$ is
\[\omega_i = \overline{S}_w^{k+1}|_{E_i} = \frac{1}{|E_i|}\int_{E_i}{S_w^{k+1}}\,dx.\]
Note that the total mass $b = \vA\omega$ is 
\begin{equation*}
b = \vA\omega = \sum_{E_i}{|E_i|\omega_i}=\sum_{E_i}{|E_i|\left(\frac{1}{|E_i|}\int_{E_i}S_w^{k+1}\,dx\right)} = \sum_{E_i}{\int_{E_i}S_w^{k+1}\,dx} = \int_{\Omega}S_w^{k+1}\,dx.
\end{equation*}
Additionally, let us define a diagonal matrix $\vD = \mathrm{diag}(\sqrt{|E_1|},\cdots,\sqrt{|E_N|})\in\R^{N\times N}.$ Then, associated with global conservation and bound-preserving constraint, we define sets
\begin{equation*}
\Lambda_1 = \{\vx: \vA\vx=b\}\quad\mathrm{and}\quad 
\Lambda_2 = \{\vx: m\leq x_i\leq M,\ \forall i = 1,\cdots, N\}.
\end{equation*}
Define the indicator function of a set $\Lambda$ as: $\iota_{\Lambda}(\vx)=0$, if $\vx\in\Lambda$, otherwise $\iota_{\Lambda}(\vx)=+\infty$.
We have the following theorem:
\begin{thm}
The model \eqref{constraints-ch5.1} is equivalent to the following minimization problem in matrix-vector form: find a vector $\vx\in\R^N$ such that it solves
\begin{equation}\label{eq:opt_model2}
\min_{\vx\in\R^{N}}{\|\vD(\vx-\omega)\|_2^2} + \iota_{\Lambda_1}(\vx) + \iota_{\Lambda_2}(\vx).
\end{equation}
\end{thm}
\begin{proof}
First, notice that the $i$-th entry of the vector $\omega$ is equal to the cell average of $S_w^{k+1}$ on cell $E_i$, we have
\begin{align*}
\|\chi_w-\overline{S}_w^{k+1}\|_{L^2}^2 
%= \int_\Omega |x_h-\overline{f_h}|^2 \,dx
&= \sum_{E_i}\int_{E_i}|\chi_w-\overline{S}_w^{k+1}|^2 \,dx
= \sum_{E_i}\left|\chi_w|_{E_i}-\overline{S}_w^{k+1}|_{E_i}\right|^2 \int_{E_i} 1 \,dx\\
&= \sum_{i} |E_i|\,|x_i-\omega_i|^2
= \|\vD(\vx-\omega)\|_2^2.
\end{align*}
Thus, seeking a piecewise constant polynomial $\overline{\chi}_w$ to minimize the $\|\chi_w-\overline{S}_w^{k+1}\|_{L^2}^2$ is equivalent to seeking a vector $\vx\in\R^{N}$ to minimize $\|\vD(\vx-\omega)\|_2^2$.
\par
Next, regarding the conservation constraint in the minimization problem \eqref{constraints-ch5.1}, notice $\chi_w$ is piecewise constant, we have
\begin{align*}
\int_\Omega \chi_w\,dx = \int_\Omega S_w^{k+1} \,dx
\quad&\Leftrightarrow\quad
\sum_{E_i} \int_{E_i} \chi_w\,dx = \sum_{E_i} \int_{E_i} S_w^{k+1}\,dx \\
\quad&\Leftrightarrow\quad
\sum_{E_i} \chi_w|_{E_i} \int_{E_i} 1\,dx = \sum_{E_i} |E_i|\Big(\frac{1}{|E_i|}\int_{E_i} S^{k+1}_w\,dx\Big). 
\end{align*}
Again, by definition of the vector $\omega$, we have
\begin{align*}
\sum_{E_i} |E_i|(\chi_w|_{E_i}) = \sum_{E_i} |E_i|(\overline{S}_w^{k+1}|_{E_i}) 
\quad\Leftrightarrow\quad
\sum_{E_i} |E_i|x_i = \sum_{E_i} |E_i|\omega_i.
\end{align*}
Recall the matrix $\vA = [|E_1|, \cdots, |E_N|]$ and the scalar $b = \vA\omega$. We get $\vA\vx = \vA\omega = b$. Thus, finding a piecewise constant polynomial $\overline{\chi}_w$ satisfies the conservation constraint in \eqref{constraints-ch5.1} is equivalent to finding a vector $\vx\in \Lambda_1$.
\par
Finally, by definition of the set $\Lambda_2$, it is straightforward to verify that the bound-preserving constraint in \eqref{constraints-ch5.1} is equivalent to seeking a vector $\vx\in \Lambda_2$.
\end{proof}

The accuracy of the optimization-based limiter is easily understood. Let ${\overline{\vS_w^\mathrm{exact}}}\in\R^N$ denote the vector whose $i$-th entry is the cell average of the exact solution $S_w^\mathrm{exact}$ on cell $E_i$, namely the $i$-th entry of ${\overline{\vS_w^\mathrm{exact}}}$ equals
\begin{align*}%\label{eq:exact_sol_cell_average}
{\overline{S_i^\mathrm{exact}}} = \frac{1}{|E_i|}\int_{E_i} S_w^\mathrm{exact}\,dx.
\end{align*}
Notice, the sets $\Lambda_1$ and $\Lambda_2$ are convex, which implies \eqref{eq:opt_model2} is a strongly convex minimization problem. Thus, it has a unique solution.
Then, we have the nonexpansiveness of the cell average limiter:
\begin{thm}\label{thm:accuracy_1}
Let $\vx^\ast$ be the minimizer of \eqref{eq:opt_model2}. Assume the exact solution satisfies $m \leq S_w^\mathrm{exact} \leq M$ on $\Omega$ and the numerical algorithm is conservative, namely $\int_\Omega S_w^\mathrm{exact}\,dx = \int_\Omega S_w^{k+1}\,dx$. Then
\begin{align*}
\|\vD(\overline{\vS_w^\mathrm{exact}} - \vx^\ast)\|_2 \leq \|\vD(\overline{\vS_w^\mathrm{exact}} - \omega)\|_2.
\end{align*}
\end{thm}
\begin{proof}
We first verify $\overline{\vS_w^\mathrm{exact}} \in \Lambda_1 \cap \Lambda_2$.
Since $m \leq S_w^\mathrm{exact} \leq M$ on $\Omega$, the cell averages satisfy $m \leq \overline{S^\mathrm{exact}_i} \leq M$ for all $i$, giving $\overline{\vS_w^\mathrm{exact}} \in \Lambda_2$.
Since the algorithm is conservative, we have
\begin{align*}
\vA\overline{\vS_w^\mathrm{exact}} 
= \sum_{E_i} |E_i|\,\overline{S^\mathrm{exact}_i} 
= \sum_{E_i} \int_{E_i} S_w^\mathrm{exact}\,dx
= \int_\Omega S_w^\mathrm{exact}\,dx
= \int_\Omega S_w^{k+1}\,dx
= b.
\end{align*}
Thus, we get $\overline{\vS_w^\mathrm{exact}} \in \Lambda_1$. Therefore, $\overline{\vS_w^\mathrm{exact}} \in \Lambda_1 \cap \Lambda_2$.
\par
Notice, the sets $\Lambda_1$ and $\Lambda_2$ are convex and closed, giving $\Lambda_1 \cap \Lambda_2$ is a convex closed set. We have $\eta\overline{\vS_w^\mathrm{exact}} + (1-\eta)\vx^\ast \in \Lambda_1 \cap \Lambda_2$ for all $\eta \in [0,1]$. 
Recall $\vD$ is a diagonal matrix. Define
\begin{align*}
\phi(\eta)
&= \|\vD(\omega - (\eta\overline{\vS_w^\mathrm{exact}} + (1-\eta)\vx^\ast))\|_{2}^2 \\
&= \eta^2\|\vD(\overline{\vS_w^\mathrm{exact}} - \vx^\ast)\|_{2}^2
- 2\eta\,(\omega - \vx^\ast)^\mathrm{T} \vD^2(\overline{\vS_w^\mathrm{exact}} - \vx^\ast)
+ \|\vD(\omega - \vx^\ast)\|_{2}^2.
\end{align*}
If $\|\vD(\overline{\vS_w^\mathrm{exact}} - \vx^\ast)\|_{2}^2 = 0$, then $\|\vD(\overline{\vS_w^\mathrm{exact}} - \vx^\ast)\|_{2} \leq \|\vD(\overline{\vS_w^\mathrm{exact}} - \omega)\|_{2}$ automatically holds.
Otherwise, it is obvious that $\phi$ is a quadratic function with respect to $\eta$.
Since $\vx^\ast$ minimizes $\|\vD(\vx - \omega)\|_{2}^2$ over $\Lambda_1 \cap \Lambda_2$, we have $\phi(0) \leq \phi(\eta)$ for all $\eta \in [0,1]$.
Since $\phi$ is a convex parabola in $\eta$ achieving its minimum at $\eta = 0$, its axis of symmetry lies at $\eta^\ast \leq 0$, which gives the sign of the linear coefficient:
\begin{align*}
(\vx^\ast - \omega)^\mathrm{T} \vD^2(\overline{\vS_w^\mathrm{exact}} - \vx^\ast) \geq 0.
\end{align*}
Therefore,
\begin{align*}
\|\vD(\overline{\vS_w^\mathrm{exact}} - \omega)\|_{2}^2
&= \|\vD(\overline{\vS_w^\mathrm{exact}} - \vx^\ast + \vx^\ast - \omega)\|_{2}^2 \\
&= \|\vD(\overline{\vS_w^\mathrm{exact}} - \vx^\ast)\|_{2}^2
+ 2(\vx^\ast - \omega)^\mathrm{T} \vD^2(\overline{\vS_w^\mathrm{exact}} - \vx^\ast)
+ \|\vD(\vx^\ast - \omega)\|_{2}^2 \\
&\geq \|\vD(\overline{\vS_w^\mathrm{exact}} - \vx^\ast)\|_{2}^2.
\end{align*}
We conclude the proof.
\end{proof}
Theorem~\ref{thm:accuracy_1} implies the optimality of the cell average limiter in matrix-vector form and further inspires the optimality of model \eqref{constraints-ch5.1}, namely
\begin{align*}%\label{eq:cell_avg_L2_optimality}
\|\overline{\chi}_w - \overline{S_w^\mathrm{exact}}\|_{L^2} \leq \|\overline{S}_w^{k+1} - \overline{S_w^\mathrm{exact}}\|_{L^2}. 
\end{align*}
To see this, recall that since $\overline{\chi}_w$ and $\overline{S_w^\mathrm{exact}}$ are both piecewise constant polynomials, we have
\begin{align*}%\label{eq:vec_to_fun_2norm}
\|\vD(\overline{\vS_w^\mathrm{exact}} - \vx^\ast)\|_2^2
= \sum_{E_i} |E_i|\,|\overline{S^\mathrm{exact}_i} - x_i^\ast|^2
= \sum_{E_i} \int_{E_i} |\overline{S^\mathrm{exact}} - \overline{\chi}_w|^2\,dx
= \|\overline{\chi}_w - \overline{S_w^\mathrm{exact}}\|_{L^2}^2,
\end{align*}
and similarly $\|\vD(\overline{\vS_w^\mathrm{exact}} - \omega)\|_2^2 = \|\overline{S}_w^{k+1} - \overline{S_w^\mathrm{exact}}\|_{L^2}^2$.
\par
We consider the three-operator splitting method for solving \eqref{constraints-ch5.1}. This method was proven to be convergent by Davis and Yin~\citep{davis2017three}. The Davis-Yin (DY) splitting method can be written into matrix-vector form, which is described as follows: set $z^0=\omega$, $b=\vA\omega$, and $l=0$. With appropriate tolerance $\varepsilon$,

\medskip
\textbf{Step 1.} Compute $x^{l+1/2} = S(z^l)$

\medskip
\textbf{Step 2.} Compute $y^{l} = 2x^{l+1/2}-z^l-\frac{1}{\max_{i}{(|E_i|)}}A^T\circ\  (x^{k+1/2}-\omega)$, where $\circ$ denotes the Hadamard product. 

\medskip
\textbf{Step 3.} Compute $x^{l+1} = y^l + \frac{b-\sum_i{|E_i|y_i^l}}{\sum_i{|E_i|^2}}\vA^T$

\medskip
\textbf{Step 4.} Compute $z^{l+1} = z^{l} + x^{l+1} - x^{l+1/2}$

\medskip
\textbf{Step 5.} If $h\|z^{l+1} - z^l\|_2 < \varepsilon$, terminate and output $\overline{\chi}_w = x^{l+1/2}$. Otherwise, set $l \leftarrow l+1$ and go to \textbf{Step 1}.

Similar to the uniform mesh case, we apply Zhang-Shu limiter \eqref{zslimiter} to eliminate undershoot and/or overshoot of the solution after anti-oscillation.

\begin{rmk}
Unlike the uniform mesh case in Section~\ref{BPlimiter-uniform}, the weighted structure of the objective function in \eqref{eq:opt_model2} makes it difficult to determine optimal parameters in DR method.
We therefore employ the DY splitting, which does not require tuning of parameters and naturally handles the weighted norm $\|\vD(\vx-\omega)\|_2$ arising from non-uniform cell volumes.
\end{rmk}

\begin{rmk}
For a uniform mesh, the minimization problem in \eqref{eq:opt_model2} reduces to \eqref{constraints}, i.e., in the matrix-vector form the matrices $\vA$ and $\vD$ reduce to a scalar multiple of an identity.
Theorem~\ref{thm:accuracy_1} still holds for the uniform mesh as a special case of the non-matching mesh~\citep{liu2024structure}.
\end{rmk}

\subsection{Minmod limiter on non-matching meshes}

Now we extend the original minmod limiter to non-matching meshes. For convenience, we reuse the notation for neighboring cells and coefficients introduced in Section \ref{sec:minmod}. Figure~\ref{nonmatching-minmod-stencil} illustrates the stencil employed for limiting in cases where neighboring cells are ~i) coarser, ~ii) finer, or ~iii) at the same refinement level. 

\begin{figure}[ht!]
\centering
\begin{tikzpicture}[scale=1., line join=round, line cap=round]

\definecolor{cellgreen}{RGB}{145,235,145}
\definecolor{cellcyan}{RGB}{175,245,245}

\tikzset{
  mesh/.style={draw=black, line width=0.9pt},
  thickmesh/.style={draw=black, line width=1.6pt},
  ghost/.style={
    draw=black,
    dashed,
    dash pattern=on 4pt off 3pt,
    line width=1.0pt
  },
  ghost2/.style={
    draw=red,
    dashed,
    dash pattern=on 4pt off 3pt,
    line width=1.2pt
  },
  arrow/.style={-latex, line width=1.1pt}
}

\draw[thickmesh] (0,0) rectangle (2,4);

\draw[mesh] (0,1) -- (2,1);
\draw[mesh] (0,2) -- (2,2);
\draw[mesh] (0,3) -- (2,3);
\draw[mesh] (1,0) -- (1,4);

\draw[mesh] (1.5,2) -- (1.5,3);

\fill[cellgreen] (0,1) rectangle (1,2);

\fill[cellcyan] (1,1) rectangle (2,2);

\fill[cellgreen] (1,0) rectangle (2,1);

\fill[cellgreen] (1,2) rectangle (1.5,2.5);
\fill[cellgreen] (1.5,2) rectangle (2,2.5);

\fill[cellgreen] (2.05,0.05) rectangle (3.95,1.95);

\draw[mesh] (1,2.5) -- (2,2.5);
\draw[mesh] (1.5,2) -- (1.5,2.5);

\draw[ghost] (2.05,0.05) -- (3.95,0.05);
\draw[ghost] (2.05,0.05) -- (2.05,1);
\draw[ghost] (3.95,0.05) -- (3.95,1.95);
\draw[ghost] (2.05,1.95) -- (3.95,1.95);
\draw[ghost2] (2.05,1) -- (2.05,1.95);

\draw[thickmesh] (0,0) rectangle (2,4);

\draw[mesh] (0,1) -- (1,1);     
\draw[mesh] (1,1) -- (2,1);
\draw[mesh] (0,2) -- (2,2);
\draw[mesh] (0,3) -- (2,3);

\draw[mesh] (1,0) -- (1,4);
\draw[mesh] (1.5,2.5) -- (1.5,3);

\draw[thickmesh] (2,0) rectangle (4,2);

\coordinate (C) at (1.5,1.5);

\draw[arrow] (1.2,1.5) -- (0.45,1.5);
\node at (0.6, 1.68) {\textbf{~iii)}}; 
\node at (0.6, 1.25){$\mathbf{E_{i,L}}$};
\draw[arrow] (1.8,1.5) -- (2.6,1.5);
\node at (2.45, 1.68) {\textbf{~i)}};
\node at (2.55, 1.25){$\mathbf{E_{i,R}}$};
\draw[arrow] (1.5,1.2) -- (1.5,0.45);
\draw[arrow] (1.5,2.3) -- (1.5,2.45);
\node at (1.68, 2.7) {\textbf{~ii)}};
\draw[arrow,blue] (0.75,2.27)--(1.26,2.27);
\node at (0.7, 2.2) {$\mathbf{E_{i,T_1}}$};
\draw[arrow,blue] (2.37,2.27)--(1.77,2.27);
\node at (2.67,2.22) {$\mathbf{E_{i,T_2}}$};

\draw[arrow] (1.1,1.5) -- (1.25,1.5);
\draw[arrow] (1.9,1.5) -- (1.75,1.5);
\draw[arrow] (1.5,1.1) -- (1.5,1.25);
\draw[arrow] (1.5,1.9) -- (1.5,1.75);

\node at (C) {$\mathbf{E_i}$};

\end{tikzpicture}
\caption{Illustration of a 2D stencil used for limiting when coarser, finer, or equal-level neighbors are present.}
\label{nonmatching-minmod-stencil}
\end{figure}
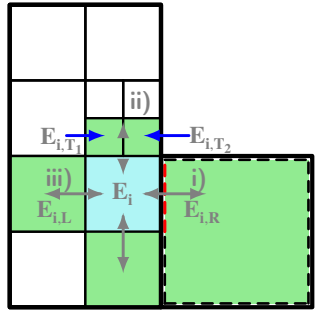

If the neighboring cells are at the same refinement level as shown in case~iii), we can use the original modal coefficients of the adjacent cell without any modification to apply the limiter, i.e., we can use \eqref{minmod-higher} and \eqref{minmod-lower}. 

When the neighboring cell is coarser than the current cell, as in case~i), the limiter uses the coefficients associated with the corresponding subface of the coarse element, where the subface is defined as the portion of the coarse cell face that intersects the face of the current fine cell. For example, suppose that cell $E_{i,R}$ is coarser than the current cell $E_{i}$. The red dashed line in Figure~\ref{nonmatching-minmod-stencil} represents the edge $e\in \partial E_{i,R}\cap \partial E_{i}$. Then, in \eqref{minmod-higher} and \eqref{minmod-lower}, the coefficients $s_{01}^{i,R}$ and $s_{00}^{i,R}$ are obtained by restricting $\tilde{S}_w^{k+1}$ to the edge $e$, i.e., $\tilde{S}_w^{k+1}|_{e}$. This treatment preserves the consistency of the limiter across hanging-node interfaces. 

Conversely, when the neighboring cells are finer than the current cell, as illustrated in case~ii), fine-cell coefficients are computed by taking the average of the coefficients of all neighboring child cells sharing the interface with the current cell. Suppose that $E_{i,T_1}$ and $E_{i,T_2}$ are two finer cells than the current cell $E_{i}$. Note that we consider the case where the common face of $E_{i}$ is subdivided into two cells. We can define the Legendre coefficients $s^{i,T_{\ast}}_{10}$ and $s^{i,T_{\ast}}_{00}$ on each cell $E_{i,T_1}$ and $E_{i,T_2}$ where $\ast=1,2$ to be the coefficients of $\tilde{S}_{w}^{k+1}|_{E_{i,T_{\ast}}}$. Then, we compute the limiter with \eqref{minmod-higher} and \eqref{minmod-lower} by using the value $s_{10}^{i,T} = \frac{1}{2}(s^{i,T_1}_{10}+s^{i,T_2}_{10})$ and $s_{00}^{i,T} = \frac{1}{2}(s^{i,T_1}_{00}+s^{i,T_2}_{00})$.

\section{Numerical experiments}\label{numerical experiments}
In this section, we present several representative numerical experiments to validate our bound-preserving anti-oscillation approach. 
We first verify, via a manufactured solution test, that the optimization-based limiter does not affect the accuracy of the DG solution. We then examine the behavior of the proposed algorithm under different settings, including inflow-driven two-phase flow in homogeneous and non-homogeneous domains, and flow in a fractured porous medium.
\par
Throughout this section, we refer to the optimization-based cell average limiter as the BP limiter and to the minmod limiter in Figure~\ref{fig:algorithm-flowchart} as the OF limiter.
In all cases, the Zhang-Shu scaling limiter is applied as the final step of the postprocessing procedure.
Each notation in Table~\ref{quantities} denotes the postprocessed saturation solution obtained after the indicated limiters (BP, OF, or BP+OF) and the subsequent Zhang-Shu scaling.
% \begin{table}[ht!]
\begin{center}
\begin{tabular}{l|l}
$S$ & Saturation without any limiter applied \\
$S^{\mathbf{BP}}$& Saturation with BP limiter (followed by Zhang-Shu scaling) applied \\
$S^{\mathbf{BP+OF}}$& Saturation with BP+OF limiters (followed by Zhang-Shu scaling) applied \\
$S^{\mathbf{BP+OF+D}}$ & Saturation with BP+OF limiters and detector (followed by Zhang-Shu scaling) applied       
\end{tabular}
\captionof{table}{Notations of the numerical solution of the saturation equation under different postprocessing strategies.}
\label{quantities}
\end{center}
% \end{table}

\subsection{Example~1: Convergence test}\label{accuracy}
To verify the accuracy of our algorithm, we consider the two-phase flow problem with a manufactured solution given by
\begin{equation*}
p_w = \cos(t+x-y),\ \ s_w=\sin(t+x-y+1)
\end{equation*}
in the domain $\Omega = (0,1)^2$ with final time $\T = 0.5$. Dirichlet conditions are imposed for pressure on the inflow and outflow boundaries with the above prescribed solution; and we only apply Dirichlet condition on the inflow boundary for the saturation equation. 
The capillary pressure is defined as
\begin{equation*}
p_c(s_w)\coloneqq \frac{B_c}{\sqrt{K_D}}\log(s_w+\varepsilon_s),
\end{equation*}
where $\vK = K_D \vI$ with $K_D = 1$, $B_c = -10^{-4}$, and $\varepsilon_s = 10^{-1}$ to avoid singularity. 
Relative permeabilities are given as a function of $s_w$, 
\begin{equation}\label{eq:relative_permeability}
k_w(s_w)\coloneqq s_w^2 \quad\text{and}\quad
k_n(s_w)\coloneqq (1-s_w)^2.
\end{equation}
In addition, we define the following parameters: $\mu_w = 1$, $\mu_n = 2$, $\rho_w = \rho_n = 1$, $c_w^F = \SI{1e-10}{}$, and $\phi=0.8$. In this test, exact values of $s_w(t^{k})$ and $s_w(t^{k-1})$ are provided to compute each $s_w^{k+1,\ast}$. We choose the penalty parameter $\sigma = 20$ and the tolerance for the splitting method $\varepsilon = \SI{1e-12}{}$. 
For each of the flow and transport equations, we performed five computations on a sequence of uniform meshes, with mesh size $h$ halved at each refinement. The time step count was doubled concomitantly. The five cycles use $50$, $100$, $200$, $400$, and $800$ steps, respectively. 
In each cycle, we measure the error in $L^2$ norm and $H^1$ seminorm of the saturation at final time $\T$ and compute the corresponding convergence rate. We compared the convergence rates for three cases: i) no limiter, ii) with only the {BP} limiter, and iii) with both {BP} and {OF} limiters. On the coarsest mesh, the OF limiter is triggered once, introducing numerical diffusion and consequently increasing the errors in case iii). From the second refinement cycle, however, the errors obtained with BP and OF limiters coincide with those of the no limiter case. As shown in Table~\ref{saturation_cr01}, the optimal convergence rates are recovered, confirming that neither BP limiter nor OF limiter affects the accuracy of the numerical solution.
% The optimal convergence rate is obtained, see Table~\ref{saturation_cr01}, confirming that neither {BP} limiter nor {OF} limiter affects the accuracy. 

%\begin{table}[ht!]
\begin{center}
\begin{tabularx}{0.85\linewidth}{@{~~}l@{\quad}c@{\quad}c@{}C@{\hspace{-0.65cm}}c@{\quad}C@{}c@{~~}} 
\toprule
~ & h & time step $\dt$ & $L^2$ error & rate & $H^1$~seminorm~error & rate\\ 
\midrule 
i) No limiter & $1/2^1$ & $1/2^0\times 10^{-2}$ & 1.003e-2 & --- & 1.207e-1 & ---\\ 
 & $1/2^2$ & $1/2^1\times 10^{-2}$ & 2.404e-3 & 2.061 & 5.655e-2 & 1.094\\ 
 & $1/2^3$ & $1/2^2\times 10^{-2}$ & 5.889e-4 & 2.029 & 2.696e-2 & 1.068\\ 
 & $1/2^4$ & $1/2^3\times 10^{-2}$ & 1.479e-4 & 1.993 & 1.326e-2 & 1.024\\ 
 & $1/2^5$ & $1/2^4\times 10^{-2}$ & 3.838e-5 & 1.946 & 6.737e-3 & 0.977\\ 
\midrule
ii) BP & $1/2^1$ &$1/2^0\times 10^{-2}$ & 1.003e-2 & --- & 1.207e-1 & ---\\ 
 & $1/2^2$ & $1/2^1\times 10^{-2}$ & 2.404e-3 & 2.061 & 5.654e-2 & 1.094\\ 
 & $1/2^3$ & $1/2^2\times 10^{-2}$ & 5.889e-4 & 2.030 & 2.696e-2 & 1.068\\ 
 & $1/2^4$ & $1/2^3\times 10^{-2}$ & 1.479e-4 & 1.993 & 1.326e-2 & 1.024\\ 
 & $1/2^5$ & $1/2^4\times 10^{-2}$ & 3.837e-5 & 1.947 & 6.734e-3 & 0.978\\
\midrule
iii) BP+OF & $1/2^1$ &$1/2^0\times 10^{-2}$ & 2.240e-2 & --- & 2.405e-1 & ---\\ 
 & $1/2^2$ & $1/2^1\times 10^{-2}$ & 2.404e-3 & 3.220 & 5.654e-2 & 2.089\\ 
 & $1/2^3$ & $1/2^2\times 10^{-2}$ & 5.889e-4 & 2.030 & 2.696e-2 & 1.068\\ 
 & $1/2^4$ & $1/2^3\times 10^{-2}$ & 1.479e-4 & 1.993 & 1.326e-2 & 1.024\\ 
 & $1/2^5$ & $1/2^4\times 10^{-2}$ & 3.837e-5 & 1.947 & 6.734e-3 & 0.978\\
\bottomrule
\end{tabularx}
\captionof{table}{(Example~1) Convergence rates of the saturation approximation without and with limiters on a uniform mesh.}
\label{saturation_cr01}
\end{center}
%\end{table}

Now, let us validate the accuracy of our limiters on non-matching mesh. Comparing the uniform mesh and the non-matching mesh of the first cycle in Figure~\ref{two-meshes}, we denote $h_{\max} = \max_{e\in\cE_h}{h_e}$ and $h_{\min}=\min_{e\in\cE_h}{h_e}$. 

\begin{figure}[ht!]
\centering
\begin{subfigure}{0.45\textwidth}
\centering
\begin{tikzpicture}[scale=3]

  \fill[blue!12] (0,0) rectangle (1,1);

  % interior grid
  \draw[line width=0.35pt] (0.25,0) -- (0.25,1);
  \draw[line width=0.35pt] (0.5,0) -- (0.5,1);
  \draw[line width=0.35pt] (0.75,0) -- (0.75,1);

  \draw[line width=0.35pt] (0,0.25) -- (1,0.25);
  \draw[line width=0.35pt] (0,0.5) -- (1,0.5);
  \draw[line width=0.35pt] (0,0.75) -- (1,0.75);

  % thick boundary
  \draw[line width=1.2pt] (0,0) rectangle (1,1);

\end{tikzpicture}
\caption{Uniform mesh}
\label{ex1-uniform-mesh}
\end{subfigure}
% \hfill
\begin{subfigure}{0.45\textwidth}
\centering
\begin{tikzpicture}[scale=3]

  \fill[blue!12] (0,0) rectangle (1,1);

  % coarse grid
  \foreach \x in {0.25,0.5,0.75}
    \draw[line width=0.35pt] (\x,0) -- (\x,1);

  \foreach \y in {0.25,0.5,0.75}
    \draw[line width=0.35pt] (0,\y) -- (1,\y);

  % refined band
  \draw[line width=0.35pt] (0,0.625) -- (1,0.625);

  % refined vertical lines
  \foreach \x in {0.125,0.375,0.625,0.875}
    \draw[line width=0.35pt] (\x,0.5) -- (\x,0.75);

  % top refinement
  \draw[line width=0.35pt] (0.25,0.875) -- (0.5,0.875);
  \draw[line width=0.35pt] (0.375,0.75) -- (0.375,1);

  \draw[line width=0.35pt] (0.25,0.375) -- (0.5,0.375);
  \draw[line width=0.35pt] (0.375,0.25) -- (0.375,0.5);

  % bottom refinement
  \draw[line width=0.35pt] (0.25,0.125) -- (0.5,0.125);
  \draw[line width=0.35pt] (0.375,0) -- (0.375,0.25);

  % thick boundary
  \draw[line width=1.2pt] (0,0) rectangle (1,1);

\end{tikzpicture}
\caption{Non-matching mesh with $h_{\max}=2h_{\min}$}
\label{ex1-nonmatching-mesh}
\end{subfigure}

\caption{Comparison between a uniform mesh and a locally refined non-matching mesh.}
\label{two-meshes}
\end{figure}
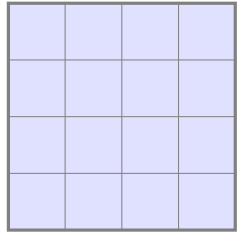
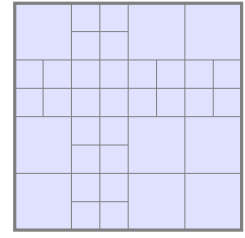

Using the same coefficients as in the previous uniform-mesh test, we evaluate the error in $L^2$ norm and $H^1$ seminorm at final time $\T=0.5$. 
The number of time steps is increased successively to $100$, $200$, $400$, $800$, and $1600$ across the cycles. We compare the convergence rates of the discrete solutions obtained from three cases: i) no limiter, ii) with only {BP} limiter, and iii) with both {BP} and {OF} limiters. Similar to the uniform mesh case, OF limiter is triggered at the first refinement cycle, resulting in a larger convergence rate. However, as we can observe in Table~\ref{saturation_cr02}, the optimal convergence rates are recovered from the second refinement cycle. 
% As shown in Table~\ref{saturation_cr02}, the optimal convergence rates for $\Q_{1}$ space are obtained. This confirms that neither {BP} limiter nor {OF} limiter harms the accuracy of the numerical solutions.

% \begin{table}[ht!]
\begin{center}
\begin{tabularx}{0.9\linewidth}{@{~~}l@{\quad}c@{\quad}c@{\quad}c@{}C@{\hspace{-0.65cm}}c@{\quad}C@{}c@{~~}} 
\toprule
~ & $h_{\mathrm{min}}$ & $h_{\mathrm{max}}$ & time step $\dt$ & $L^2$ error & rate & $H^1$ seminorm error & rate\\ 
\midrule
i) No limiter & $1/2^3$ & $1/2^2$  & $1/2^1\times10^{-2}$ & 9.942e-3 & --- & 1.523e-1 & ---\\ 
 & $1/2^4$ & $1/2^3$ & $1/2^2\times10^{-2}$ & 2.222e-3 & 2.162 & 6.698e-2 & 1.185\\ 
 & $1/2^5$ & $1/2^4$ & $1/2^3\times10^{-2}$ & 5.065e-4 & 2.133 & 2.796e-2 & 1.260\\ 
 & $1/2^6$ & $1/2^5$ & $1/2^4\times10^{-2}$ & 1.237e-4 & 2.034 & 1.261e-2 & 1.149\\ 
 & $1/2^7$ & $1/2^6$ & $1/2^5\times10^{-2}$ & 3.223e-5 & 1.940 & 6.137e-3 & 1.039\\ 
\midrule
ii) BP & $1/2^3$ & $1/2^2$ & $1/2^1\times10^{-2}$ & 9.991e-3 & --- & 1.521e-1 & ---\\ 
 & $1/2^4$ & $1/2^3$ & $1/2^2\times10^{-2}$ & 2.220e-3 & 2.170 & 6.693e-2 & 1.184\\ 
 & $1/2^5$ & $1/2^4$ & $1/2^3\times10^{-2}$ & 5.062e-4 & 2.133 & 2.795e-2 & 1.260\\ 
 & $1/2^6$ & $1/2^5$ & $1/2^4\times10^{-2}$ & 1.236e-4 & 2.034 & 1.260e-2 & 1.150\\ 
 & $1/2^7$ & $1/2^6$ & $1/2^5\times10^{-2}$ & 3.220e-5 & 1.941 & 6.125e-3 & 1.040\\ 
\midrule
iii) {BP+OF} & $1/2^3$ & $1/2^2$ & $1/2^1\times10^{-2}$ & 2.870e-2 & --- & 3.211e-1 & ---\\ 
 & $1/2^4$ & $1/2^3$ & $1/2^2\times10^{-2}$ & 2.220e-3 & 3.692 & 6.693e-2 & 2.263\\ 
 & $1/2^5$ & $1/2^4$ & $1/2^3\times10^{-2}$ & 5.062e-4 & 2.133 & 2.795e-2 & 1.260\\ 
 & $1/2^6$ & $1/2^5$ & $1/2^4\times10^{-2}$ & 1.236e-4 & 2.034 & 1.260e-2 & 1.150\\ 
 & $1/2^7$ & $1/2^6$ & $1/2^5\times10^{-2}$ & 3.220e-5 & 1.941 & 6.125e-3 & 1.040\\ 
\bottomrule
\end{tabularx}
\captionof{table}{(Example~1) Convergence rate of the saturation approximation without and with limiters on non-matching mesh.}
\label{saturation_cr02}
\end{center}
% \end{table}

\subsection{Example~2: Inflow-driven two-phase flow}\label{ex2}

We investigate the performance of our proposed algorithm on homogeneous and non-homogeneous domains. The computational domains, boundary conditions, and initial conditions are illustrated in Figure~\ref{ex2-setup}. In both cases, the wetting-phase pressure is prescribed on the left and right boundaries to drive the flow from left to right, while the wetting phase saturation is imposed at the inflow boundary. The homogeneous case considers a constant permeability throughout the domain, and non-homogeneous case contains a low-permeability box embedded in a highly permeable porous matrix.

\begin{figure}[ht!]
\centering
\begin{subfigure}{0.45\linewidth}
\scalebox{0.4}
{\begin{tikzpicture}[
    line cap=round,
    line join=round,
    font=\normalsize,
]
 
\def\L{7} 
\fill[domainbg] (0,0) rectangle (\L,\L);
 
\def\bx{2.1}  \def\by{2.1}   
\def\bw{2.8}  \def\bh{2.8}   
 
\draw[labelblue, line width=2pt] (0,0) -- (0,\L);
\draw[labelblue, line width=2pt] (\L,0) -- (\L,\L);
\draw[red!70, line width=2pt] (0,0) -- (\L,0);
\draw[red!70, line width=2pt] (0,\L) -- (\L,\L);

\node[below] at (0,    -0.05) {$0$};
\node[below] at (\L,   -0.05) {$1$};
\node[left]  at (-0.05, \L)   {$1$};
 
\node[font=\LARGE, align=center] at (\bx+\bw/2, \by+\bh/2)
      {$\mathbf{K} = 10^{-1}\mathbf{I}$};
 
\foreach \y in {2.5, 3.5, 4.5}{
  \draw[arrowcol, line width=1.4pt, -stealth]
        (-1.8,\y) -- (-0.08,\y);
}

\node[ font=\huge, anchor=east] at (-1.95, 4.2)
      {$s_w = 1$};
\node[ font=\huge, anchor=east] at (-1.95, 3.1)
      {$p_w = 1$};
 
% \node[labelblue, font=\large, anchor=north east, align=right]
%       at (-0.05, 0.4) {$\Gamma_{D}$};

% \node[red!70, font=\large, anchor=north, align=right] at (0.5, 7.6) {$\Gamma_N$};
 
\node[font=\huge, anchor=west] at (\L+0.35, \L/2)
      {$p_w = 0$};
% \node[labelblue, font=\large, anchor=north west, align=left]
%       at (\L+0.05, 0.4) {$\Gamma_{+}$};

% \node[red!70, font=\large, anchor=south, align=right] at (0.5, -0.6) {$\Gamma_N$};
 
\node[font=\LARGE] at (\L/2, -0.65)
      {$\mathbf{u}_N = 0 \text{ on } \Gamma_{N}$};
 
\node[font=\LARGE] at (\L/2, \L+0.65)
      {$\mathbf{u}_N = 0 \text{ on } \Gamma_{N}$};
 
\node[font=\LARGE, align=center] at (\L/2, 1.1) {%
  \begin{tabular}{l}
    $s_w(\mathbf{x},\,t=0) = 0$\\[4pt]
    $p_w(\mathbf{x},\,t=0) = 0$
  \end{tabular}
};

\end{tikzpicture}}
\caption{Computational domain and boundary values of the homogeneous domain}
\label{ex2-1-homogeneous}
\end{subfigure}
\begin{subfigure}{0.45\linewidth}
    \centering
\scalebox{0.4}
{\begin{tikzpicture}[
    line cap=round,
    line join=round,
    font=\normalsize,
]
 
\def\L{7} 
 
\fill[domainbg] (0,0) rectangle (\L,\L);
 
\def\bx{2.8}  \def\by{2.8}  
\def\bw{1.4}  \def\bh{1.4}   
\fill[blockgreenlight] (\bx,\by) rectangle (\bx+\bw,\by+\bh);
\draw[blockgreen, line width=1.6pt] (\bx,\by) rectangle (\bx+\bw,\by+\bh);
 
\draw[labelblue, line width=2pt] (0,0) -- (0,\L);
\draw[labelblue, line width=2pt] (\L,0) -- (\L,\L);
\draw[red!70, line width=2pt] (0,0) -- (\L,0);
\draw[red!70, line width=2pt] (0,\L) -- (\L,\L);

\node[below] at (0,    -0.05) {$0$};
\node[below] at (\L,   -0.05) {$1$};
\node[left]  at (-0.05, \L)   {$1$};
 
\node[font=\LARGE] at (1.4, \L-1.0)
      {$\mathbf{K} = 10^{-1}\mathbf{I}$};
 
\node[font=\LARGE, align=center] at (\bx+\bw/2, \by+\bh/2)
      {$\mathbf{K} = 10^{-3}\mathbf{I}$};
 
\foreach \y in {2.5, 3.5, 4.5}{
  \draw[arrowcol, line width=1.4pt, -stealth]
        (-1.8,\y) -- (-0.08,\y);
}

\node[ font=\huge, anchor=east] at (-1.95, 4.2)
      {$s_w = 1$};
\node[ font=\huge, anchor=east] at (-1.95, 3.1)
      {$p_w = 1$};

% \node[labelblue, font=\large, anchor=north east, align=right]
%       at (-0.05, 0.4) {$\Gamma_{D}$};

% \node[red!70, font=\large, anchor=north, align=right] at (0.5, 7.6) {$\Gamma_N$};
 
\node[font=\huge, anchor=west] at (\L+0.35, \L/2)
      {$p_w = 0$};
% \node[labelblue, font=\large, anchor=north west, align=left]
%       at (\L+0.05, 0.4) {$\Gamma_{D}$};

% \node[red!70, font=\large, anchor=south, align=right] at (0.5, -0.6) {$\Gamma_N$};
 
\node[font=\LARGE] at (\L/2, -0.65)
      {$\mathbf{u}_N = 0 \text{ on } \Gamma_{N}$};
 
\node[font=\LARGE] at (\L/2, \L+0.65)
      {$\mathbf{u}_N = 0 \text{ on } \Gamma_{N}$};
 
\node[font=\LARGE, align=center] at (\L/2, 1.1) {%
  \begin{tabular}{l}
    $s_w(\mathbf{x},\,t=0) = 0$\\[4pt]
    $p_w(\mathbf{x},\,t=0) = 0$
  \end{tabular}
};

\end{tikzpicture}}
\caption{Computational domain and boundary values of the non-homogeneous domain}
\label{ex2-2-heterogeneous}
\end{subfigure}
        \caption{(Example~2) Computational domains, permeability, boundary conditions, and initial conditions for the inflow-driven two-phase flow simulations.}
        \label{ex2-setup}
\end{figure}
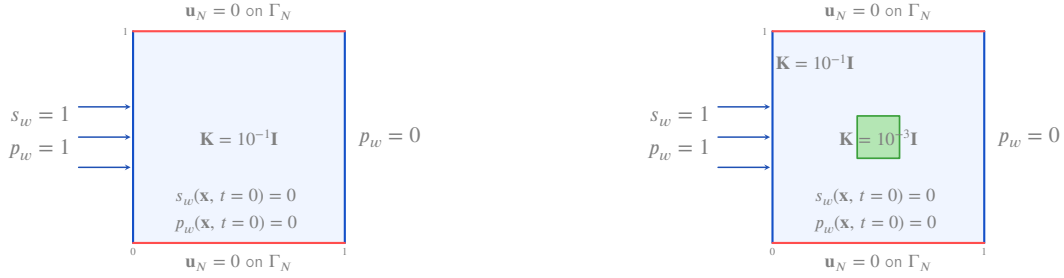

\subsubsection{Example~2-1: Homogeneous domain}\label{ex2-1-simulation}

This part presents a homogeneous test case with constant permeability $\vK = K_D\vI$, where $K_D = 10^{-1}$. The physical parameters are set to $\mu_w = 1$, $\mu_n = 2$, $\rho_w = \rho_n = 1$, $c_w^F = \SI{1e-10}{}$, and $\phi=0.2$. Relative permeabilities follow the same definition as in \eqref{eq:relative_permeability} from Section~\ref{accuracy} and the tolerance for the splitting method is $\varepsilon = \SI{1e-12}{}$.. 
We consider a uniform mesh with $h=1/128$ and the time step is chosen to be $\dt = 0.001$ where the final time is $\T=1$. Dirichlet boundary conditions for the pressure are imposed as $p_w(\vx,\cdot)=1$ on $\{x=0\}\times[0,1]$ and $p_w(\vx,\cdot)=0$ on $\{x=1\}\times[0,1]$. The inflow boundary condition for saturation is given as $s_w(\vx,\cdot)=1$ on $\{x=0\}\times[0,1]$, and the outflow boundary condition is chosen on the right boundary $\{x=1\}\times[0,1]$, see Figure~\ref{ex2-1-homogeneous}. The saturation front is expected to propagate from the inlet $\{x=0\}$ toward the outlet $\{x=1\}$, and the saturation field to stay within $[0,1]$.

We illustrate the results that compare the numerical solutions ~i) without any limiter applied, ~ii) with limiter BP applied only, and ~iii) limiter {BP} and {OF} applied. We plot the profiles of the three cases on $\{0 \leq x \leq 1\}\times \{y=0.5\}$ at time $t = 0.25, 0.5$, and $1$ in Figure~\ref{ex2-01}. First, numerical solutions with limiters are in good agreement with the solution without the limiter with respect to front location. Furthermore, we observe that applying only the BP limiter still exhibits oscillations near the propagating fronts comparable to those displayed in the unlimited solution. 
By applying the limiter {OF}, we are able to suppress spurious oscillations present in solutions ~i) and ~ii), thereby yielding satisfactory oscillation-free solution.

\begin{figure}[ht!]
\centering
    \begin{subfigure}{0.3\linewidth}
    \centering
    \includegraphics[trim = 0 10 0 20, clip, width=\textwidth]{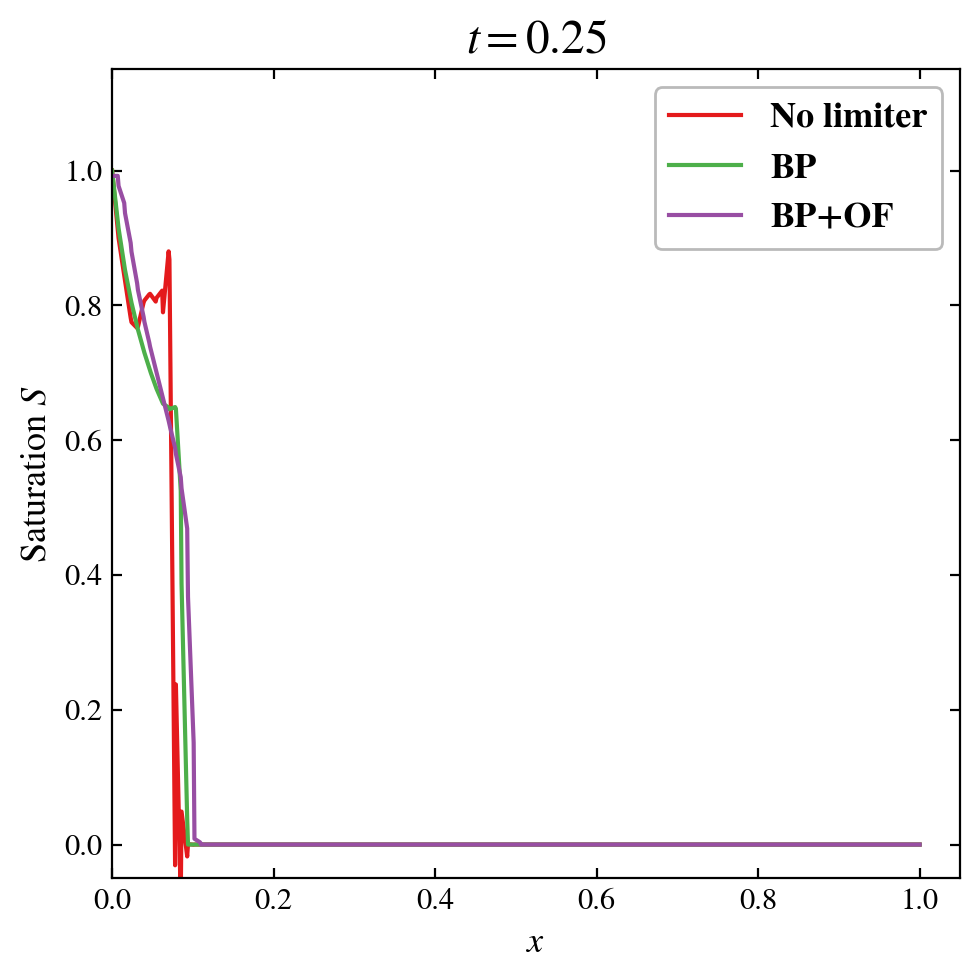}
    \caption{$t=0.25$}    
    \end{subfigure}
    % \hspace{0.2in}
    \begin{subfigure}{0.3\textwidth}
    \centering
    \includegraphics[trim = 0 10 0 20, clip,width=\textwidth]{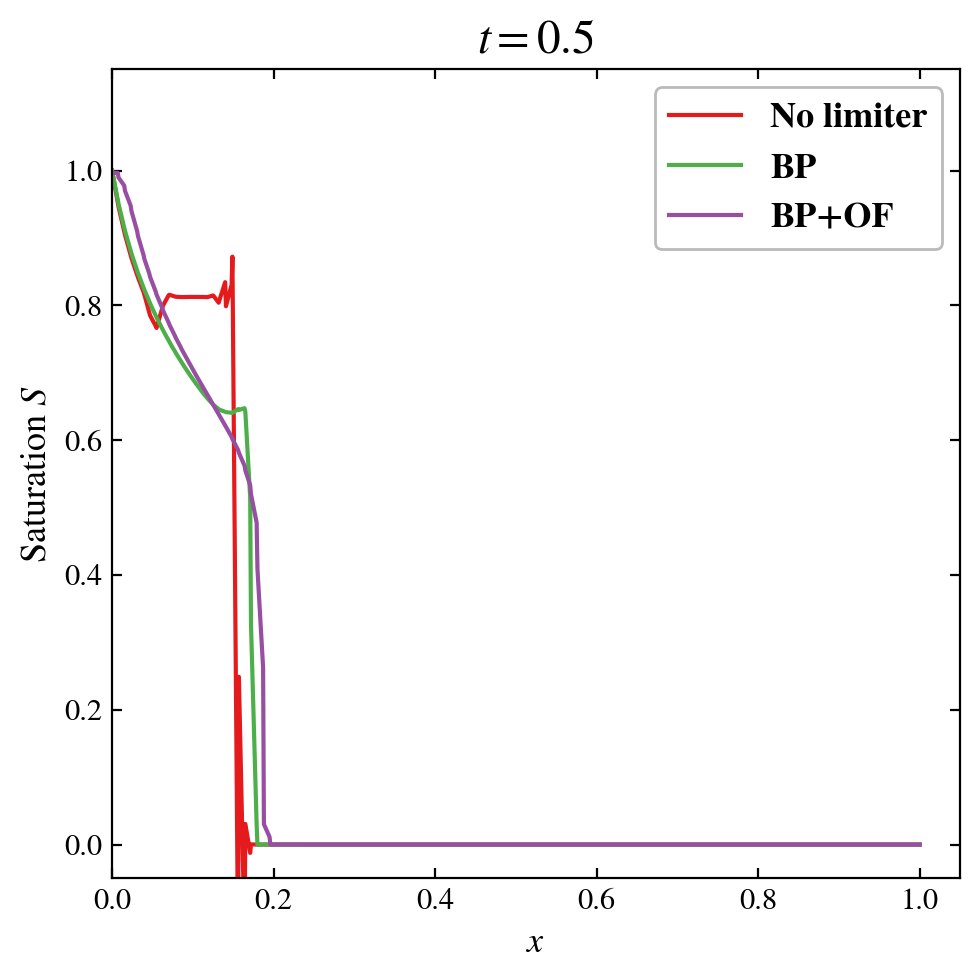}
    \caption{$t=0.5$}    
    \end{subfigure}
    % \vspace{0.2cm}
    % \begin{subfigure}{0.32\textwidth}
    % \centering
    % \includegraphics[trim = 0 10 0 20, clip,width=\textwidth]{01_saturation_ts299.png}
    % \caption{$t=0.75$}        
    % \end{subfigure}
    \begin{subfigure}{0.3\textwidth}
    \centering
    \includegraphics[trim = 0 10 0 20, clip,width=\textwidth]{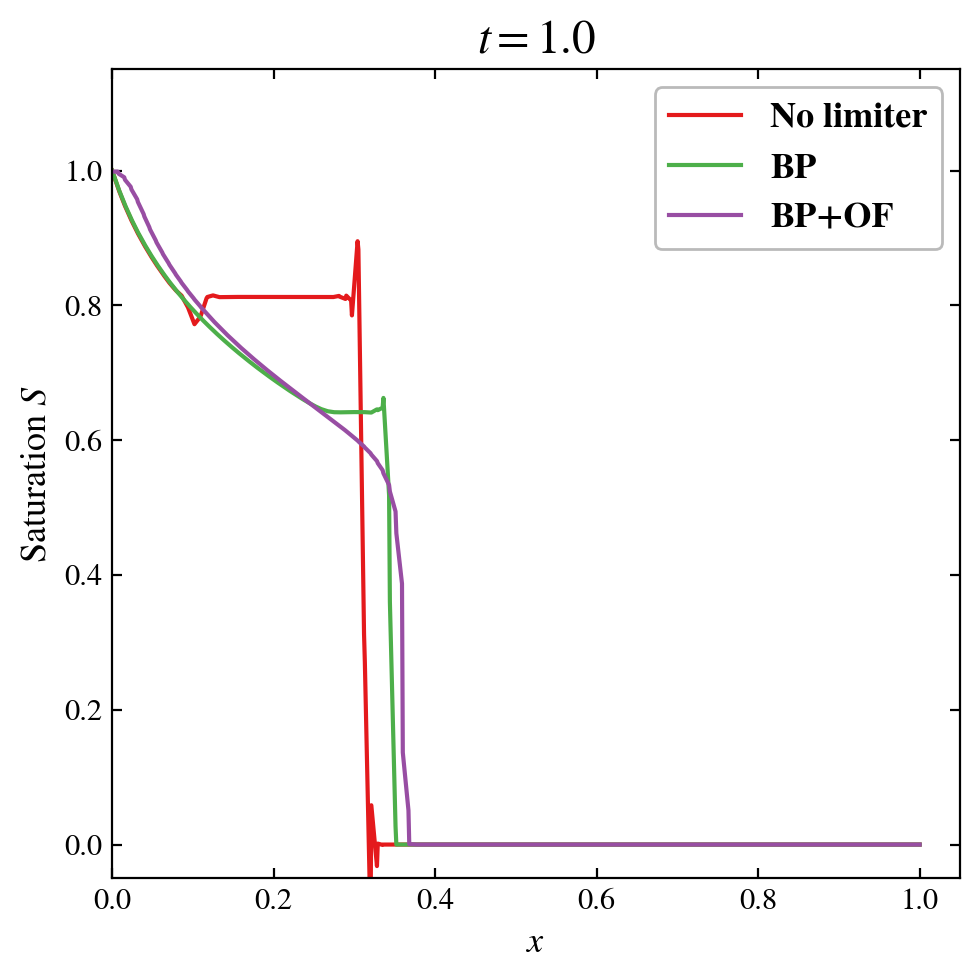}
    \caption{$t=1.0$}            
    \end{subfigure}
    \caption{(Example~2-1) $S$, $\SBP$, $\SOF$ on $\{0\leq x \leq 1\}\times \{y = 0.5\}$ at time $t=0.25$, $0.5$, and $1$.}
    \label{ex2-01}
\end{figure}

Now, we compare the solutions ~i) limiter {OF} and {BP} applied to the entire cells and ~ii) limiter {OF} and {BP} applied only to the troubled cells. In this experiment, since the solution is expected to vary more strongly in the $x$-direction, we modify the smoothness indicator in \eqref{detector} as:
\begin{equation}\label{eq:smooth_indicator_x_direction}
  \delta_{i}=\frac{(s_{10}^{i})^2}{\sum_{l,m=0}^1{(s_{lm}^{i})^2}},  
\end{equation}

where $s_{10}^{i}$ represents the first order coefficient in $x$-direction. We mark the cell as a troubled cell when 
$\delta_{i} \geq \SI{1e-4}{}$.

Before showing the solution profiles, we first illustrate the indicator distribution used to identify troubled cells in Figure~\ref{ex2-indicator}. The cells marked in red correspond to those classified as troubled cells. We now compare this figure with the solution profiles in Figure~\ref{ex2-02}.

\begin{figure}[ht!]
\centering
    \begin{subfigure}{0.3\textwidth}
    \centering
    \includegraphics[trim = 0 0 100 120,clip,width=\textwidth]{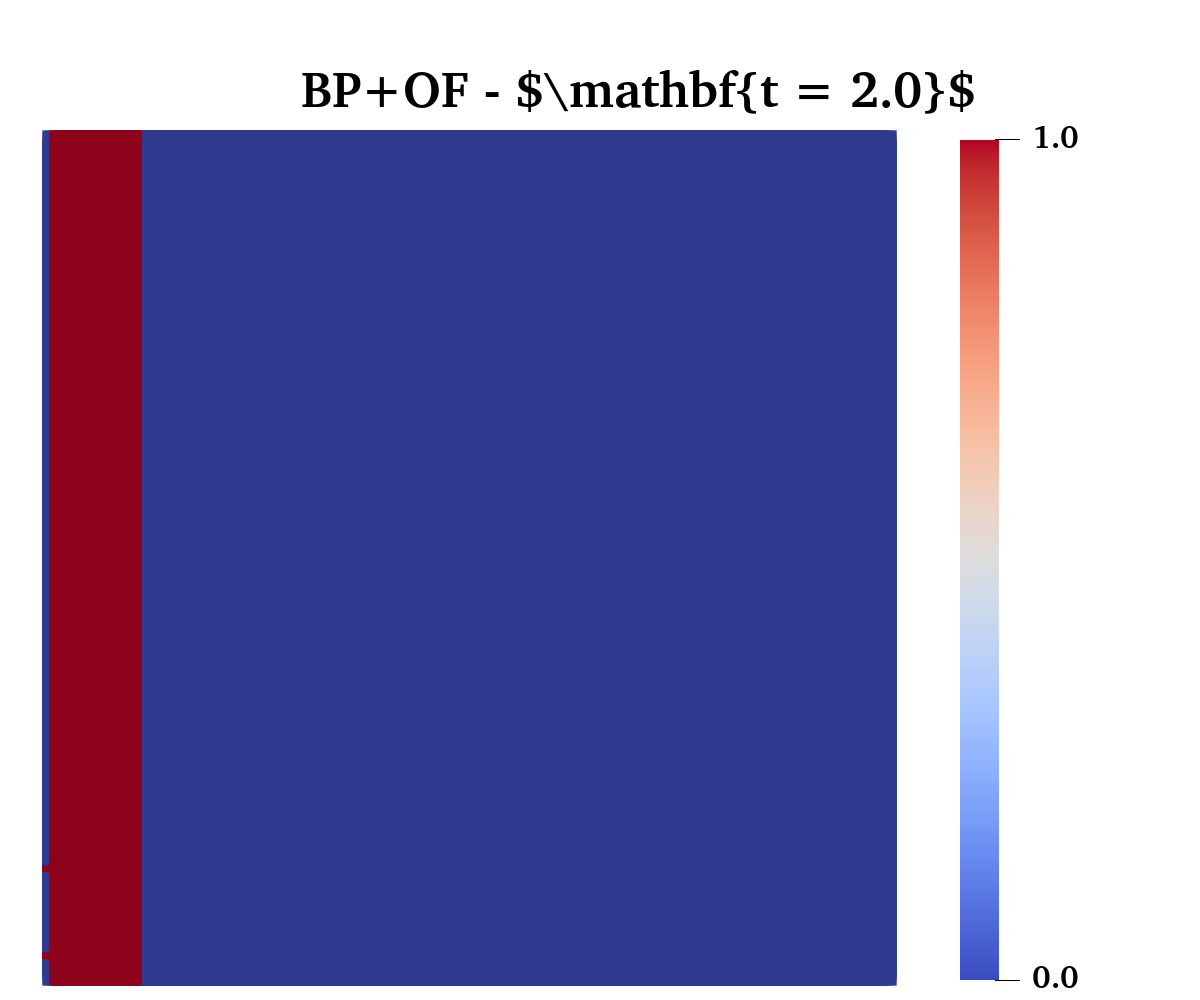}
    \caption{$t=0.25$}    
    \end{subfigure}
    \begin{subfigure}{0.3\textwidth}
    \centering
    \includegraphics[trim = 0 0 100 120,clip,width=\textwidth]{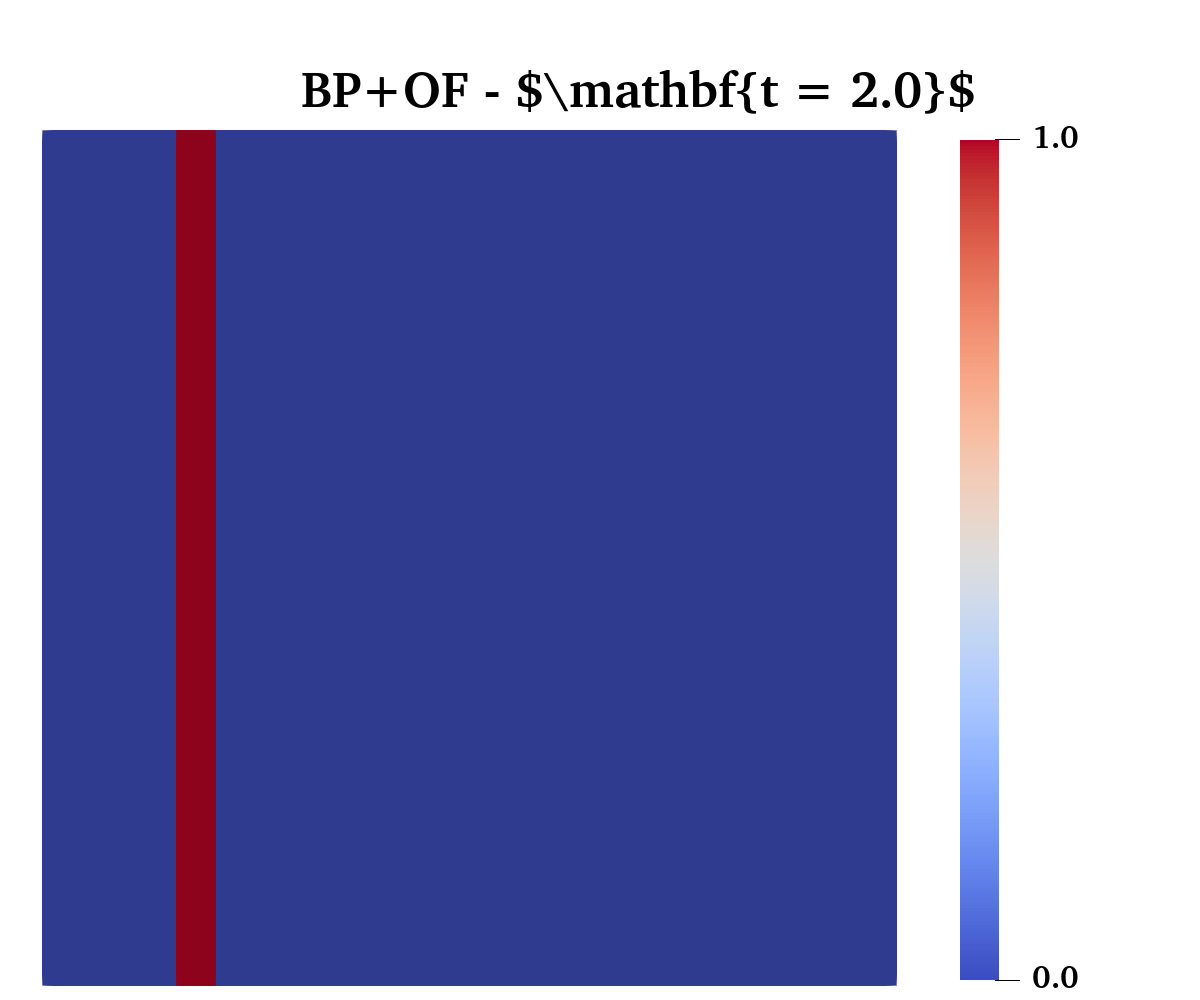}
    \caption{$t=0.5$}            
    \end{subfigure}
    % \vspace{0.2cm}
    % \begin{subfigure}{0.32\textwidth}
    % \centering
    % \includegraphics[trim = 0 0 100 120,clip,width=\textwidth]{indicator_299.png}
    % \caption{$t=0.75$}        
    % \end{subfigure}
    %\hfill
    \begin{subfigure}{0.3\textwidth}
    \centering
    \includegraphics[trim = 0 0 100 120,clip,width=\textwidth]{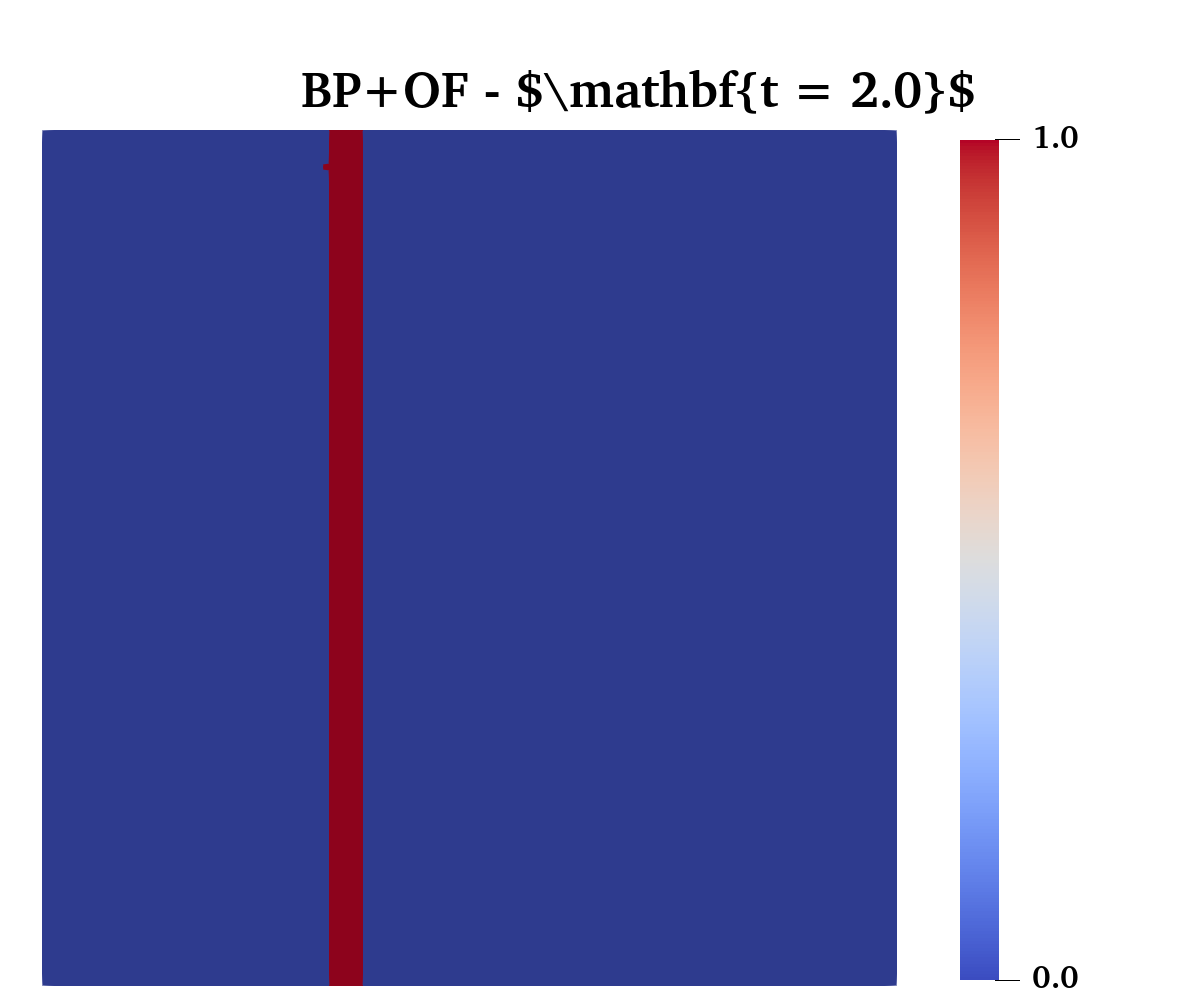}
    \caption{$t=1.0$}        
    \end{subfigure}
    \caption{(Example~2-1) The troubled-cell indicator at time $t=0.25$, $0.5$, and $1$. The red colored regions indicate the cells where the limiter is triggered.}
    \label{ex2-indicator}
\end{figure}

We zoom in on the solution profiles over the interval $x\in[0,0.05]$ to highlight the differences between the two solutions. As shown in Figure~\ref{ex2-indicator}, the troubled-cell indicators shift rightward over time, in close agreement with the propagation of the saturation front. Moreover, from the magnified view of the solution profiles, we observe that the regions not identified as troubled cells exhibit a less diffusive behavior.

\begin{figure}[ht!]
\centering
    \begin{subfigure}{0.3\textwidth}
    \centering
    \includegraphics[trim = 0 10 0 20, clip,width=\textwidth]{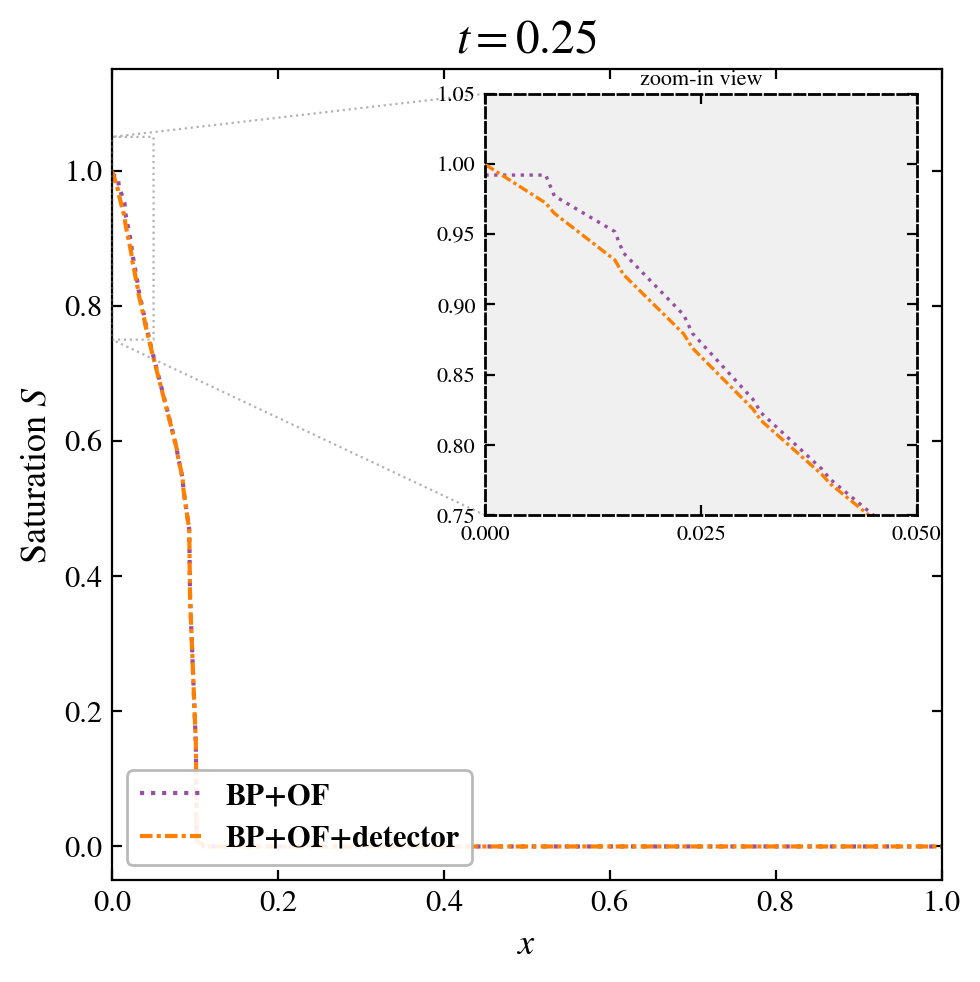}
    \caption{$t=0.25$}
    \end{subfigure}
    \begin{subfigure}{0.3\textwidth}
    \centering
    \includegraphics[trim = 0 10 0 20, clip,width=\textwidth]{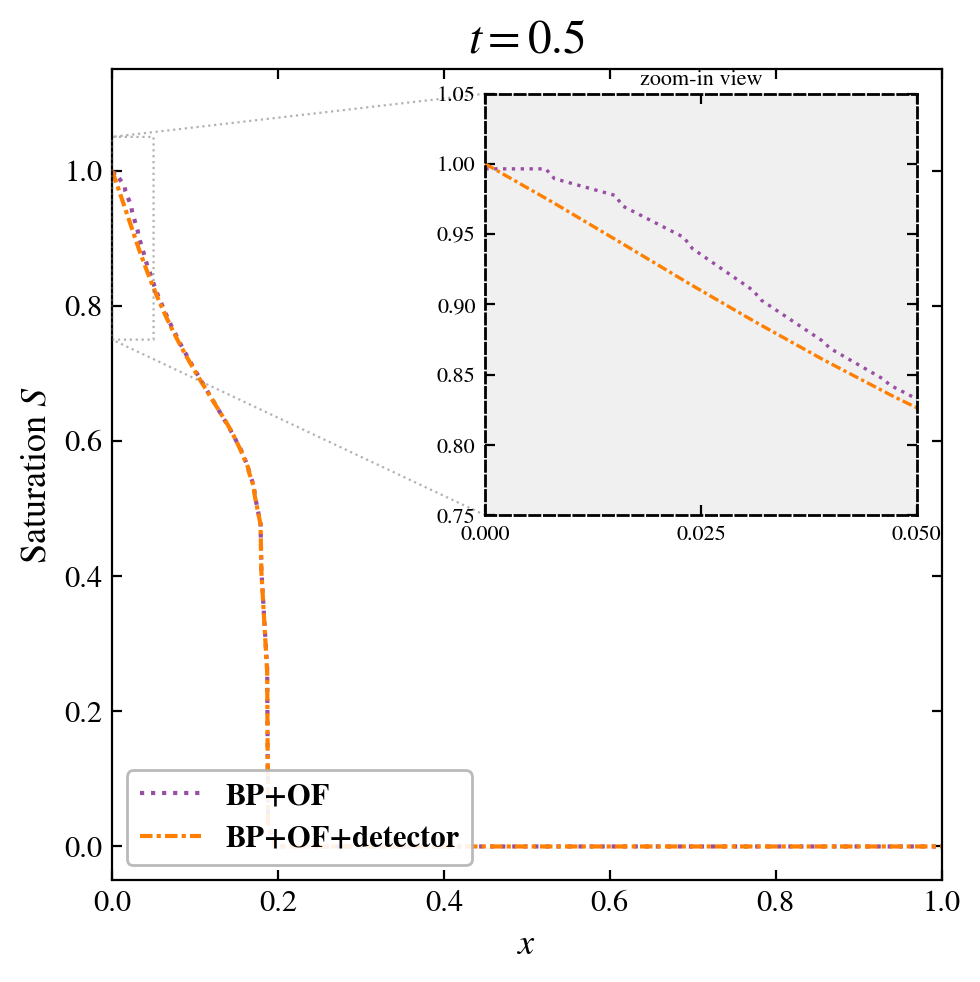}
        \caption{$t=0.5$}
    \end{subfigure}
    % \vspace{0.2cm}
    % \begin{subfigure}{0.35\linewidth}
    % \centering
    % \includegraphics[trim = 0 10 0 20, clip,width=\linewidth]{02_saturation_ts299.png}
    %     \caption{$t=0.75$}
    % \end{subfigure}
    \begin{subfigure}{0.3\textwidth}
    \centering
    \includegraphics[trim = 0 10 0 20, clip,width=\textwidth]{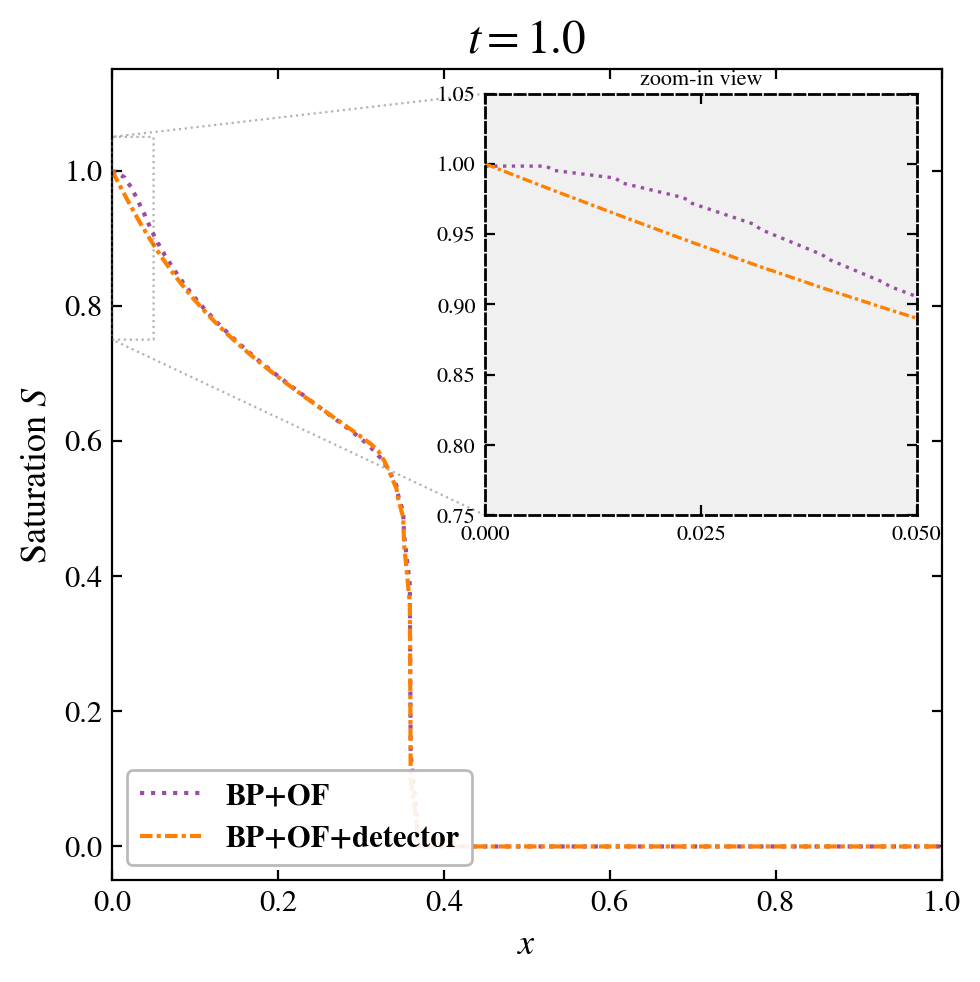}
        \caption{$t=1.0$}
    \end{subfigure}
    \caption{(Example~2-1) $\SOF,\SOFD$ on $\{0\leq x \leq 1\}\times \{y = 0.5\}$ at time $t=0.25$, $0.5$, and $1$. }
    \label{ex2-02}
\end{figure}

We now present the numerical results for a non-matching mesh, using the same physical and numerical parameters as in the uniform-mesh experiment. The non-matching mesh is constructed with local refinement over the regions $[0,1]\times[0.5,0.75]$ and $[0.25,0.5]\times[0,1]$. 

First, we compare the numerical solutions obtained from three different cases: ~i) without applying any limiter, ~ii) applying only the BP limiter, and ~iii) applying both the BP and OF limiters. The solution profiles along $\{0\leq x\leq 1\}\times\{y=0.6\}$ at times $t=0.25$, $0.5$, and $1$ are presented in Figure~\ref{nonmatching-ex2-01}. The results show that the front locations obtained by the limiters match those of the unlimited solution. When only the BP limiter is applied, oscillations are still visible near the sharp saturation front. By additionally applying the OF limiter, these oscillations are effectively controlled, resulting in a more stable and diffusive numerical solution. 

\begin{figure}[ht!]
\centering
    \begin{subfigure}{0.3\linewidth}
    \centering
    \includegraphics[trim = 0 10 0 20, clip,width=\linewidth]{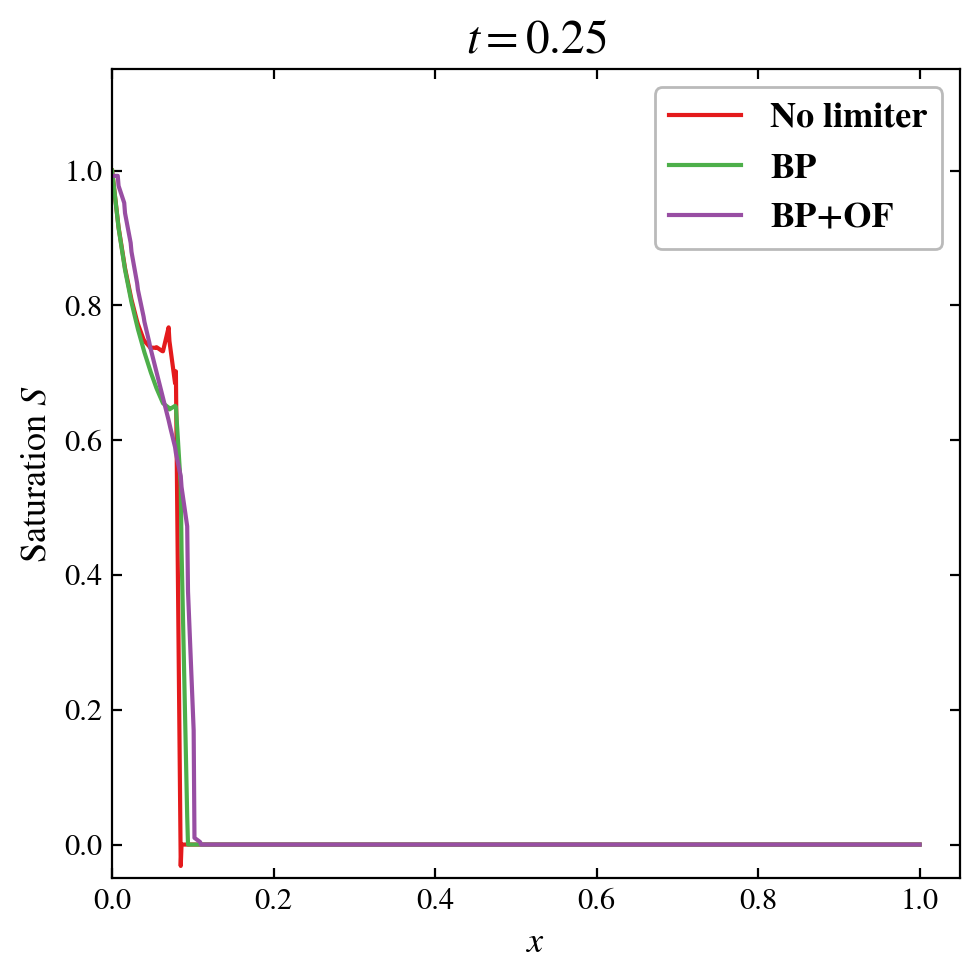}
        \caption{$t=0.25$}
    \end{subfigure}
    \begin{subfigure}{0.3\linewidth}
    \centering
    \includegraphics[trim = 0 10 0 20, clip,width=\linewidth]{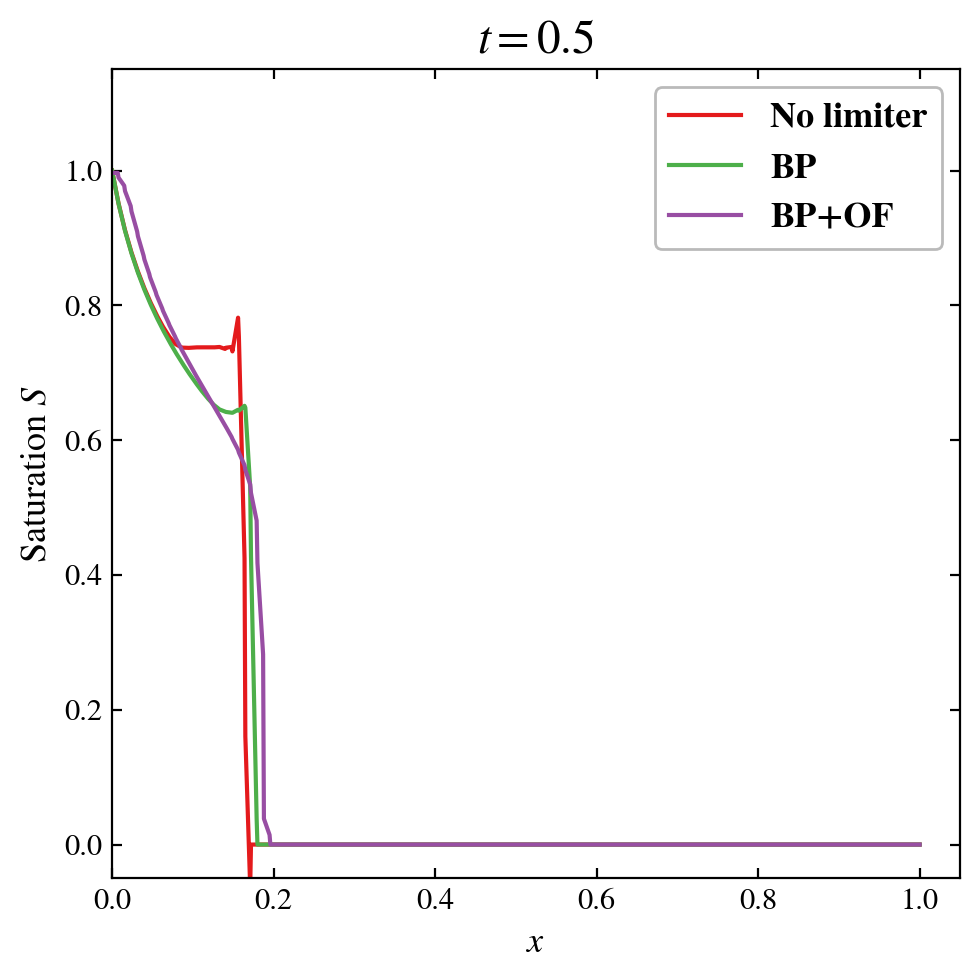}
        \caption{$t=0.5$}
    \end{subfigure}
    % \vspace{0.2cm}
    % \begin{subfigure}{0.32\linewidth}
    % \centering
    % \includegraphics[trim = 0 10 0 20, clip,width=\linewidth]{nonmatching-01_saturation_ts299.png}
    %     \caption{$t=0.75$}
    % \end{subfigure}
    \begin{subfigure}{0.3\linewidth}
    \centering
    \includegraphics[trim = 0 10 0 20, clip,width=\linewidth]{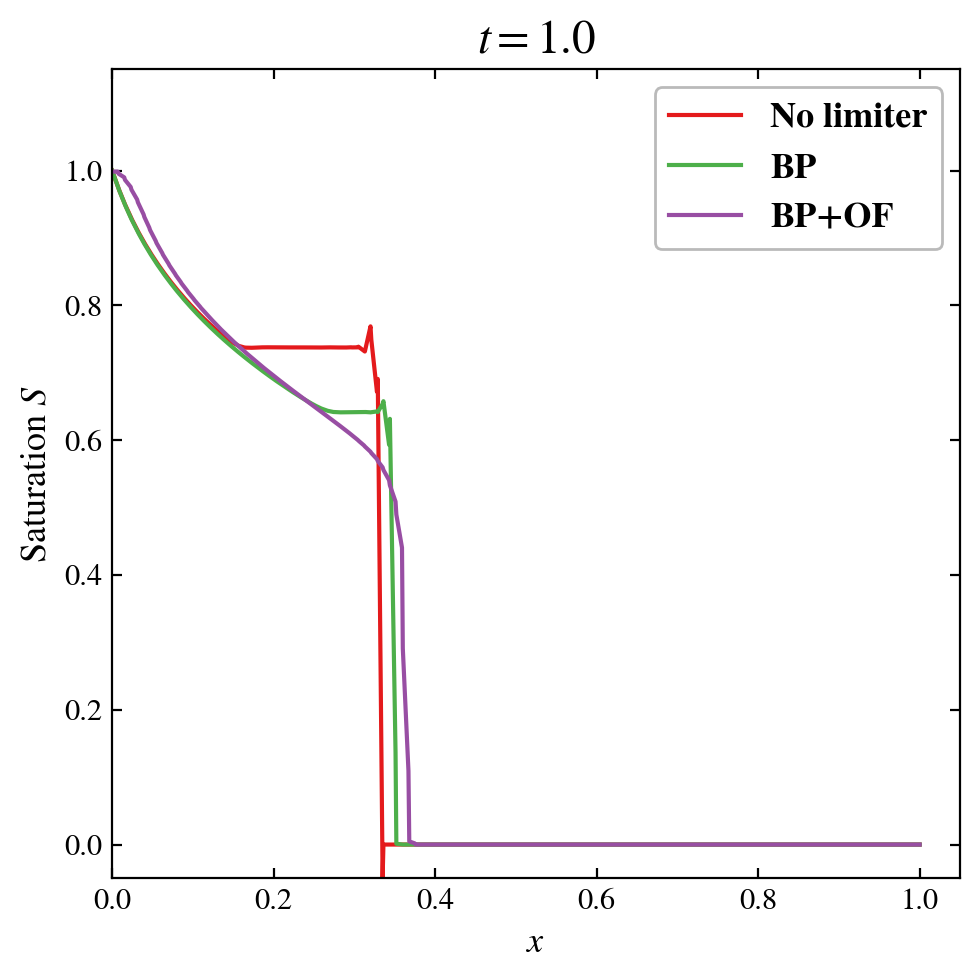}
        \caption{$t=1.0$}
    \end{subfigure}
    \caption{(Example~2-1, Non-matching mesh) $S$, $\SBP$, $\SOF$ on $\{0\leq x \leq 1\}\times \{y = 0.6\}$ at time $t=0.25$, $0.5$, and $1$.}
    \label{nonmatching-ex2-01}
\end{figure}

Next, we compare the numerical solutions obtained by ~i) the OF and BP limiters applied over all cells, and ~ii) the OF and BP limiters applied only to the detected troubled cells on the non-matching mesh. The same definition of the smooth indicator as in the uniform-mesh experiment is employed. A cell is classified as a troubled cell whenever
$\delta_{i}\geq\SI{1e-5}{}$.

We present the distribution of the trouble-cell indicator on the non-matching mesh in Figure~\ref{ex2-nonmatching-indicator}. The cells highlighted in red mark the regions where the limiter is selectively activated.
In the early phase of the simulation, the detector identifies a large number of troubled cells in both the coarse and fine mesh regions. As the solution evolves, the number of troubled cells gradually decreases, particularly within the refined region, indicating that the solution becomes smoother and the anti-oscillation limiting is required only locally near the remaining sharp front. We then compare solutions ~i) and ~ii) in Figure~\ref{nonmatching-ex2-02}.

\begin{figure}[ht!]
\centering
    \begin{subfigure}{0.3\linewidth}
    \centering
    \includegraphics[trim = 0 0 100 120,clip,width=\linewidth]{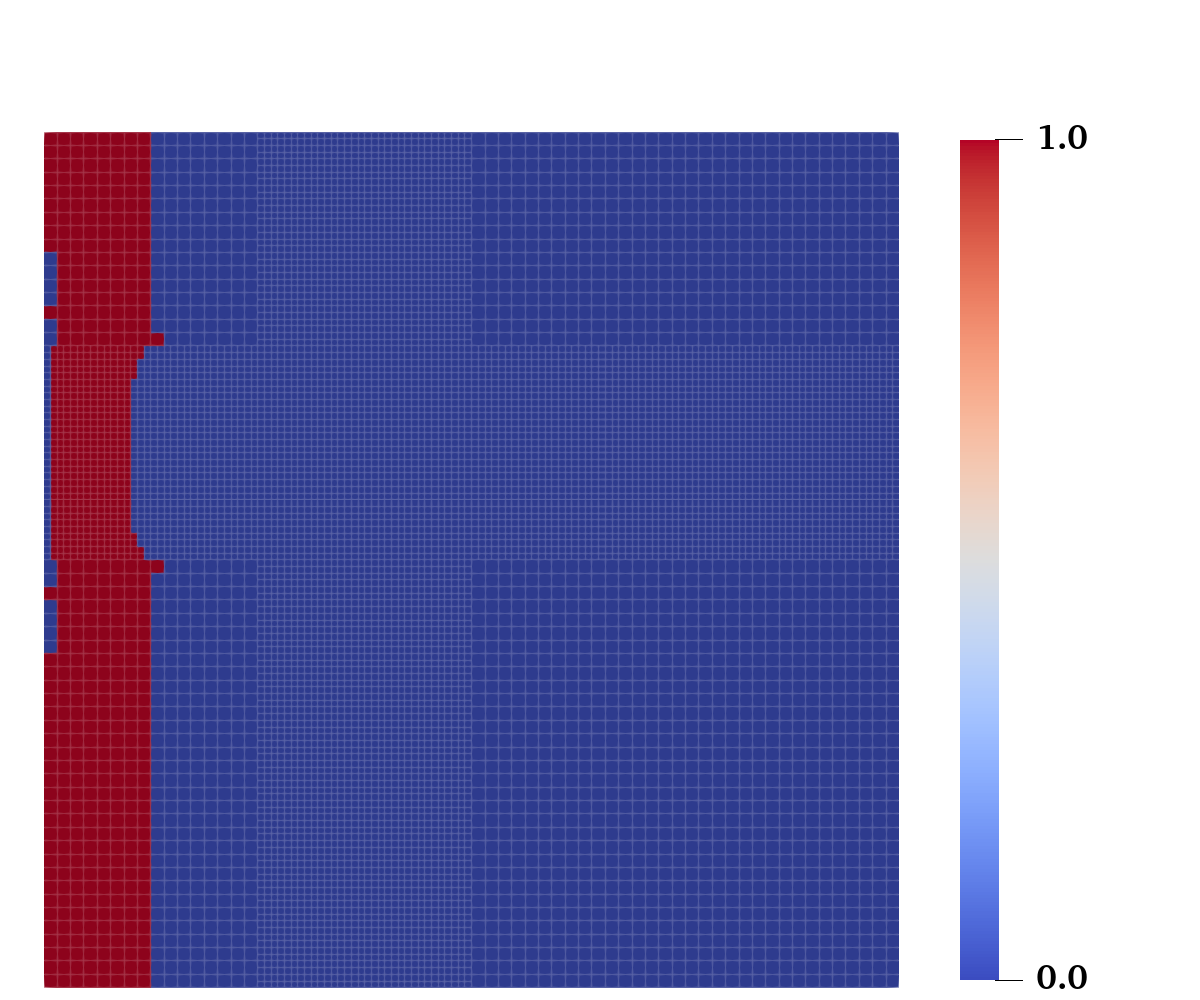}
        \caption{$t=0.25$}
    \end{subfigure}
    \begin{subfigure}{0.3\linewidth}
    \centering
    \includegraphics[trim = 0 0 100 120,clip,width=\linewidth]{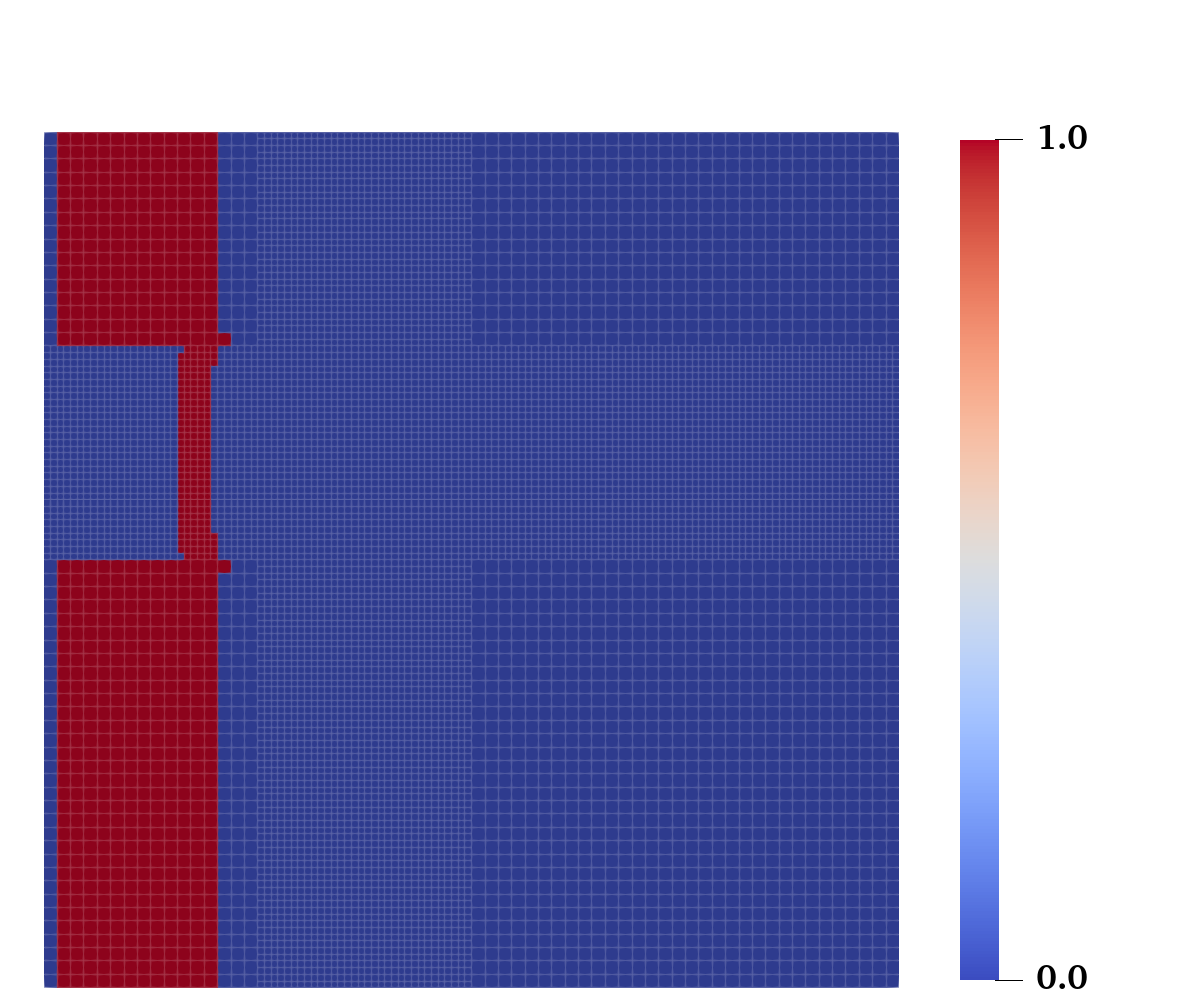}
        \caption{$t=0.5$}
    \end{subfigure}
    % \vspace{0.2cm}
    % \begin{subfigure}{0.32\linewidth}
    % \centering
    % \includegraphics[trim = 0 0 100 120,clip,width=\linewidth]{nonmatching-indicator-wireframe_ts299.png}
    %     \caption{$t=0.75$}
    % \end{subfigure}
    \begin{subfigure}{0.3\linewidth}
    \centering
    \includegraphics[trim = 0 0 100 120,clip,width=\linewidth]{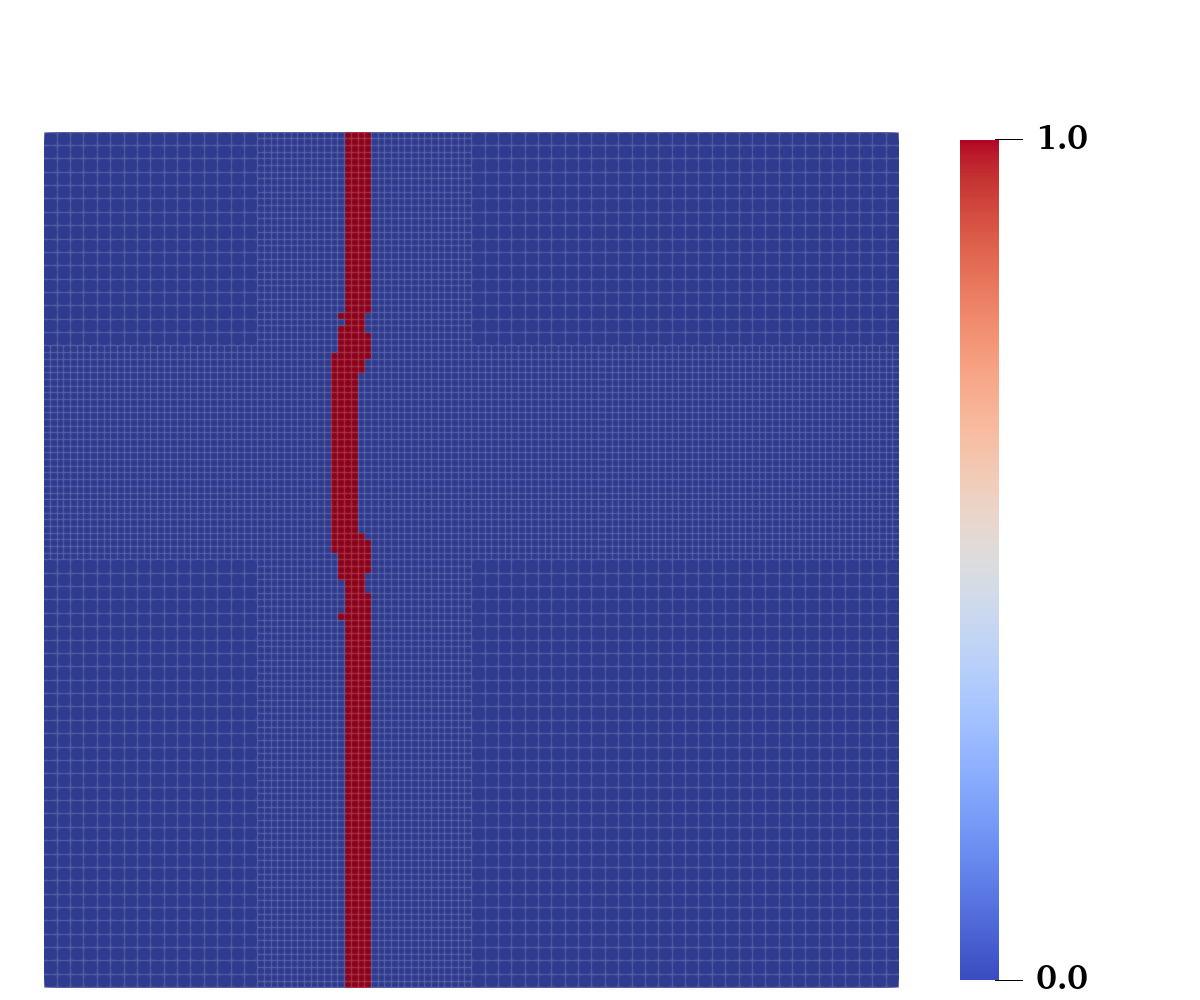}
        \caption{$t=1.0$}
    \end{subfigure}
    \caption{(Example~2-1, Non-matching mesh) Evolution of the troubled-cell indicator at $t = 0.25$, $0.5$, and $1$.}
    \label{ex2-nonmatching-indicator}
\end{figure}

We zoom in on the solution profiles over the interval $x\in[0,0.05]$ to emphasize the differences between ~i) and ~ii). As illustrated in Figure~\ref{ex2-nonmatching-indicator}, the troubled-cell indicator shifts toward the right-hand side as time evolves, especially in the refined region, following the propagation of the sharp saturation front. Furthermore, the magnified solution profiles in
Figure~\ref{nonmatching-ex2-02} demonstrate that the regions not classified
as troubled retain sharper features and exhibit less diffusive behavior,
confirming that the limiter is applied more selectively.

\begin{figure}[ht!]
\centering
    \begin{subfigure}{0.3\linewidth}
    \centering
    \includegraphics[trim = 0 10 0 20, clip,width=\linewidth]{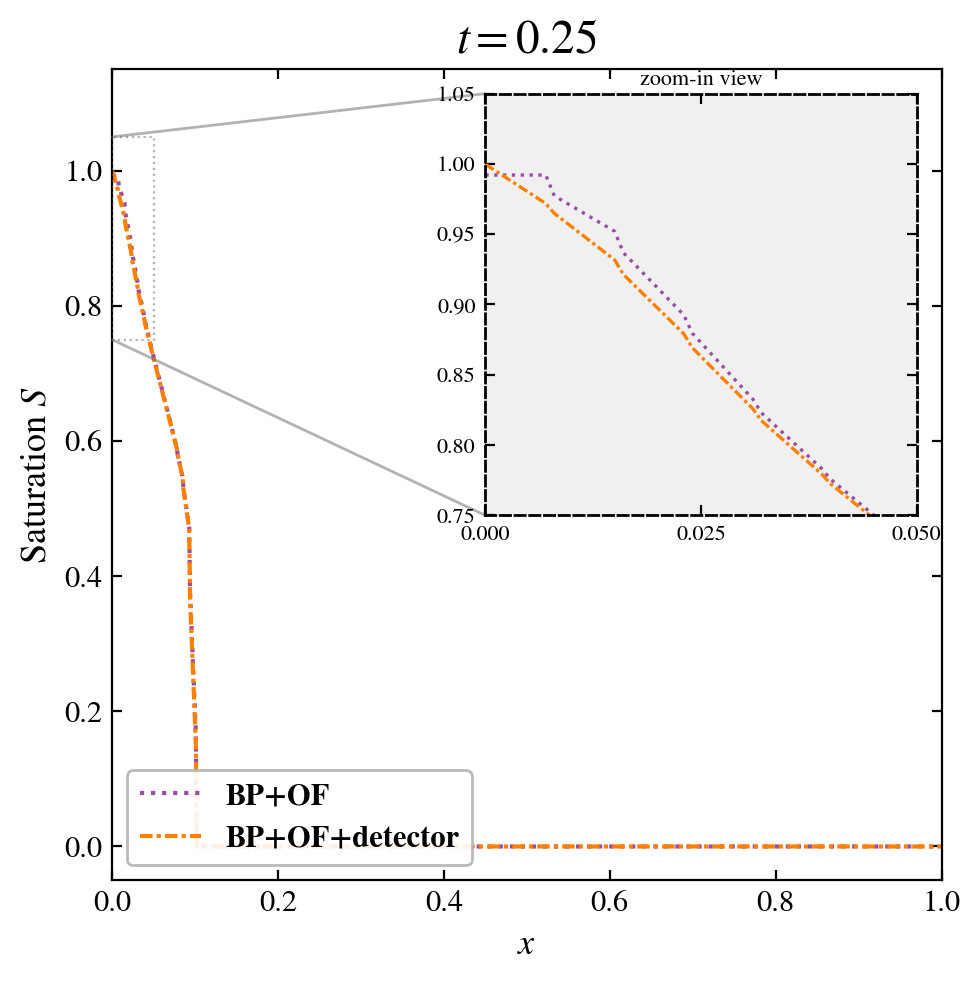}
        \caption{$t=0.25$}
    \end{subfigure}
    \begin{subfigure}{0.3\linewidth}
    \centering
    \includegraphics[trim = 0 10 0 20, clip,width=\linewidth]{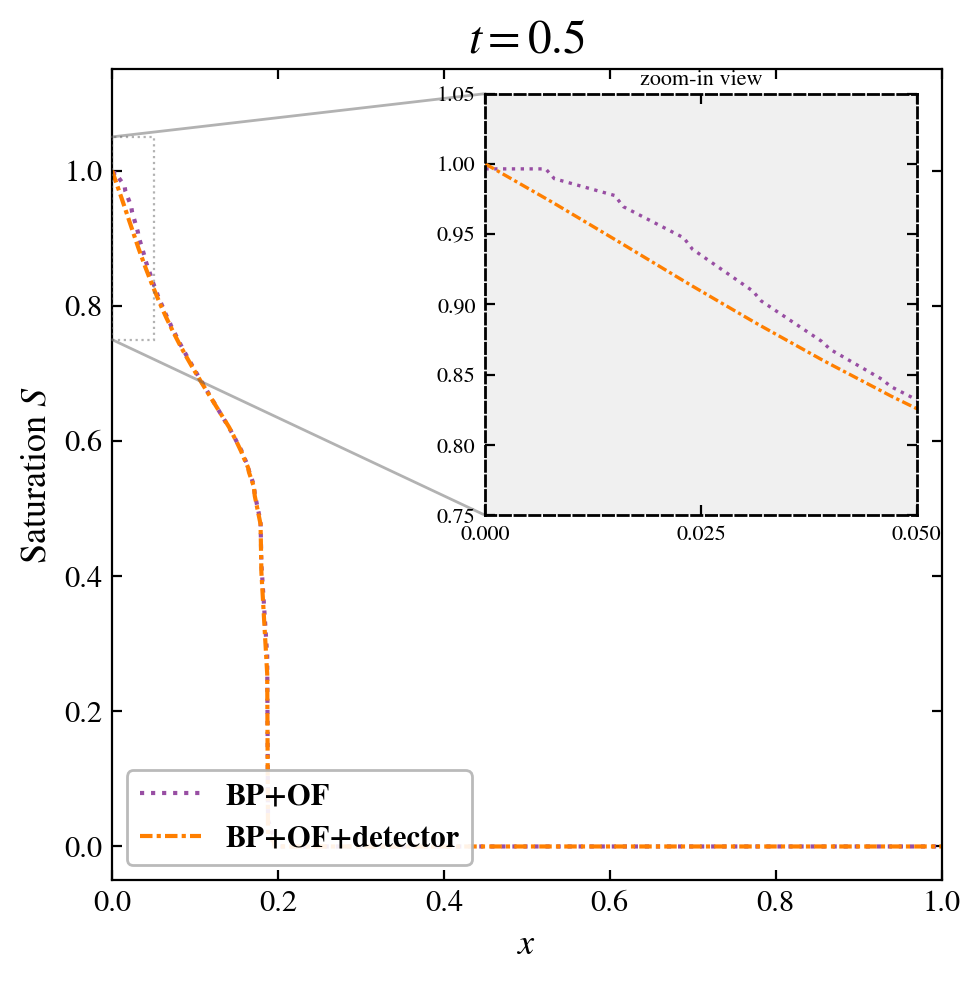}
        \caption{$t=0.5$}
    \end{subfigure}
    % \vspace{0.2cm}    
    % \begin{subfigure}{0.35\linewidth}
    % \centering
    % \includegraphics[trim = 0 10 0 20, clip,width=\linewidth]{nonmatching-02-saturation_ts299.png}
    %     \caption{$t=0.75$}
    % \end{subfigure}
    \begin{subfigure}{0.3\linewidth}
    \centering
    \includegraphics[trim = 0 10 0 20, clip,width=\linewidth]{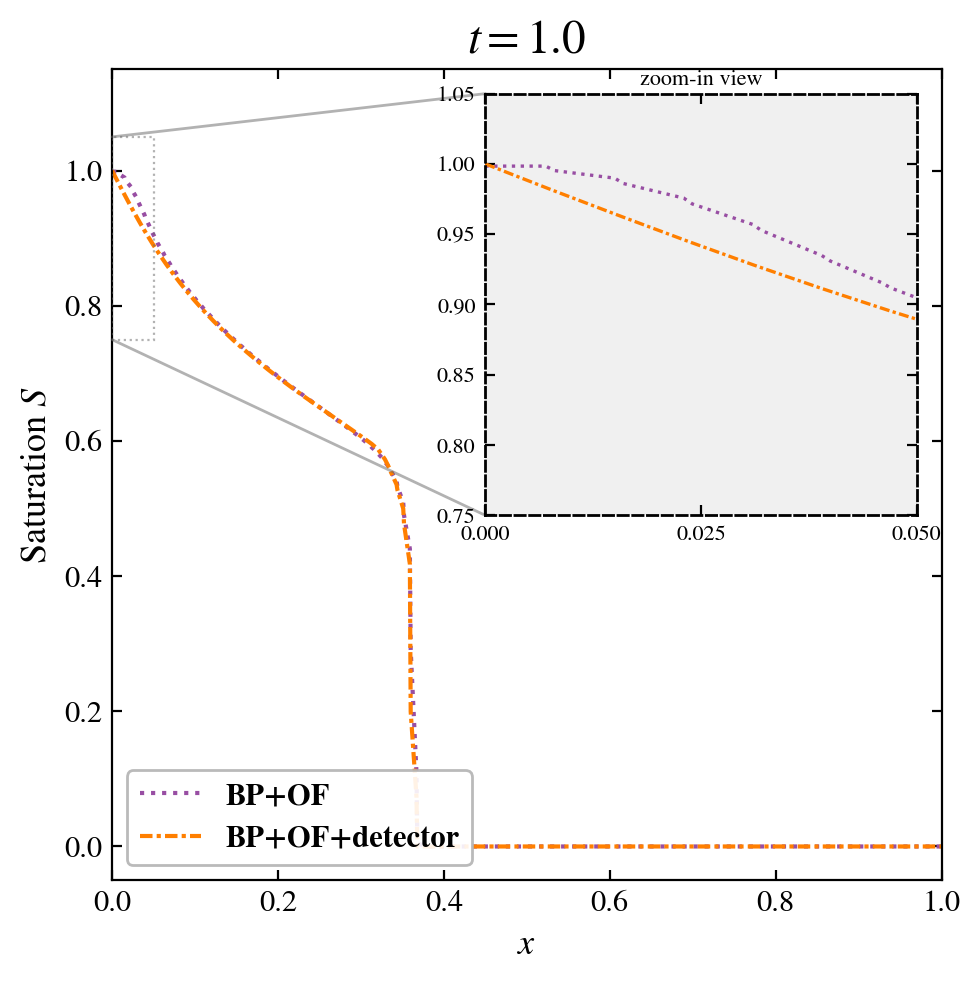}
        \caption{$t=1.0$}
    \end{subfigure}
    \caption{(Example~2-1, Non-matching mesh) $\SOF,\SOFD$ on $\{0\leq x \leq 1\}\times \{y = 0.6\}$ at time $t=0.25$, $0.5$, and $1$. }
    \label{nonmatching-ex2-02}
\end{figure}

\subsubsection{Example~2-2: Non-homogeneous domain}\label{ex2-2-simulation}
In this test, we set permeability $\vK = K_D(\vx)\vI$ where 
\begin{align*}
    K_D(\vx) = 
    \begin{cases}
        \SI{1e-3}{},\ &\ \mathrm{if}\ \vx\in[0.4,0.6]^2,\\
        \SI{1e-1}{},\ &\ \mathrm{else}.
    \end{cases}
\end{align*}
Other physical parameters are defined as $\mu_w = 1$, $\mu_n = 2$, $\rho_w = \rho_n = 1$, $c_w^F=\SI{1e-1}{}$, and $\phi=1$. Relative permeabilities follow the same definition as in \eqref{eq:relative_permeability}. For any $e\in \cE_h$, we denote $K_{+} = K_D|_{E^+}$ and $K_{-}=K_D|_{E^-}$. Then we set the penalty term on each face $\sigma = 20K_e$ where $K_e$ is defined as the harmonic mean of $K_+$ and $K_-$, i.e., $K_e\coloneqq\frac{2K_{+}K_{-}}{K_{+}+K_{-}}$. The tolerance for the splitting method is $\varepsilon = \SI{1e-12}{}$.
We consider a uniform mesh with $h = 1/128$ and the final time is $\T=2$, where the time step is chosen to be $\dt = 0.005$. The initial and boundary conditions are depicted in Figure~\ref{ex2-2-heterogeneous}. 
\par
The necessity of the bound-preserving anti-oscillation postprocessing is demonstrated in Figure~\ref{ex3-01}. Figure~\ref{ex3-01}(a) compares the saturation without limiters and with the BP+OF limiters at $t=1.29$. Without limiting, oscillations and out-of-bounds values can be observed, and the simulation becomes unstable and crashes after this time. In contrast, $\SOF$ remains within the prescribed bounds and effectively suppresses the oscillations. Figure~\ref{ex3-01}(b) shows the stable limited solution $\SOF$ at the final time $\T=2$.

\begin{figure}[ht!]
\centering
\newlength{\singlefigurewidth}
\setlength{\singlefigurewidth}{0.29\linewidth}
    % \begin{subfigure}{0.3\linewidth}
    % \centering
    %     \includegraphics[trim = 0 0 100 120,clip,width=\linewidth]{noAB_saturation_ts199.png}
    %     \caption{Saturation with no limiters ; $t = 1$}
    % \end{subfigure}
    % \begin{subfigure}{0.3\linewidth}
    % \centering
    %     \includegraphics[trim = 0 0 100 120,clip, width=\linewidth]{noindicator_saturation_ts199.png}
    %     \caption{Saturation with BP+OF ; $t = 1$}
    % \end{subfigure}\\
    % \vspace{0.2cm}
    \begin{subfigure}{0.6\linewidth}
    \centering
        \includegraphics[trim = 0 0 100 120,clip,width=\singlefigurewidth]{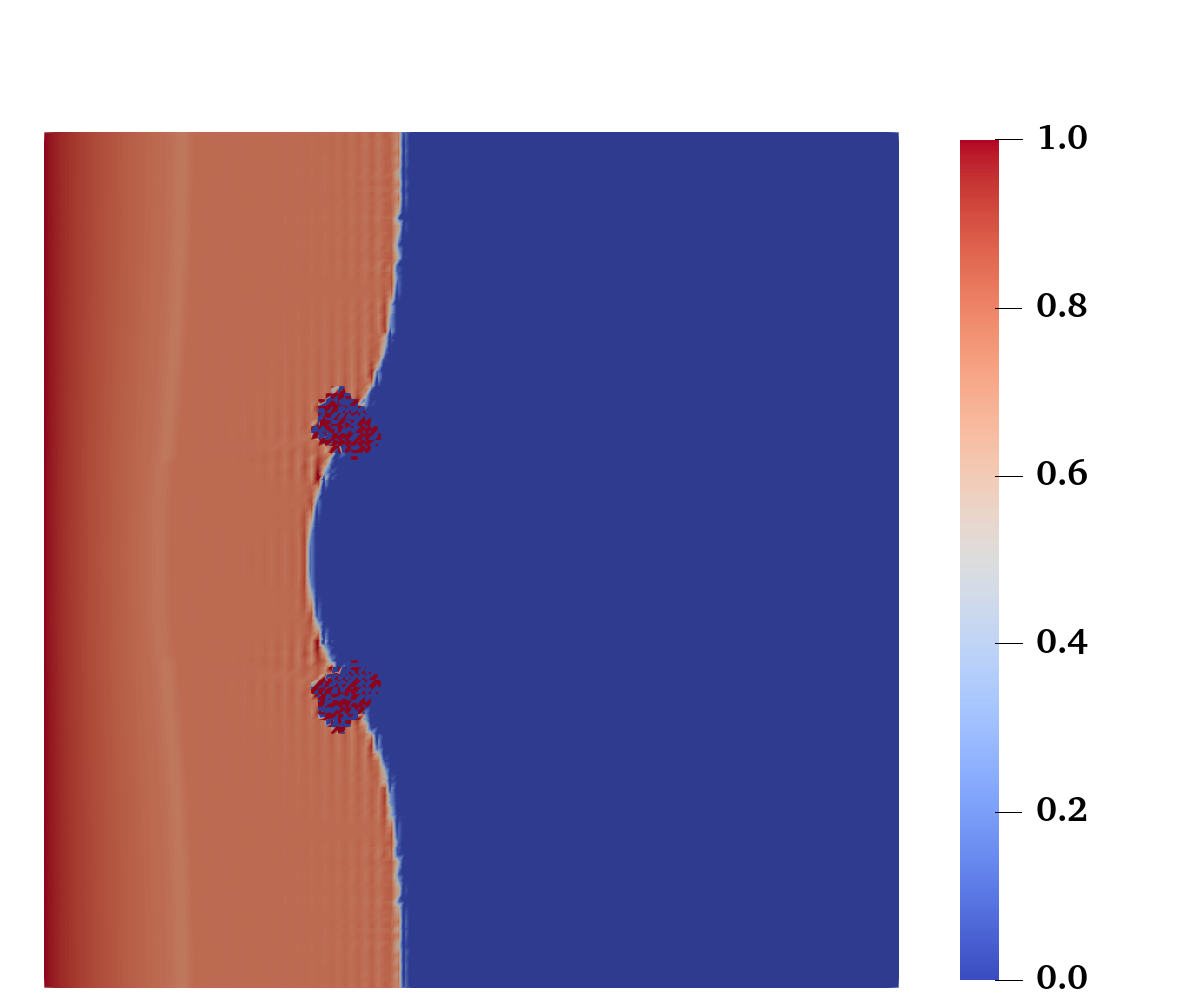}
        \hfill
        \includegraphics[trim = 0 0 100 120,clip,width=\singlefigurewidth]{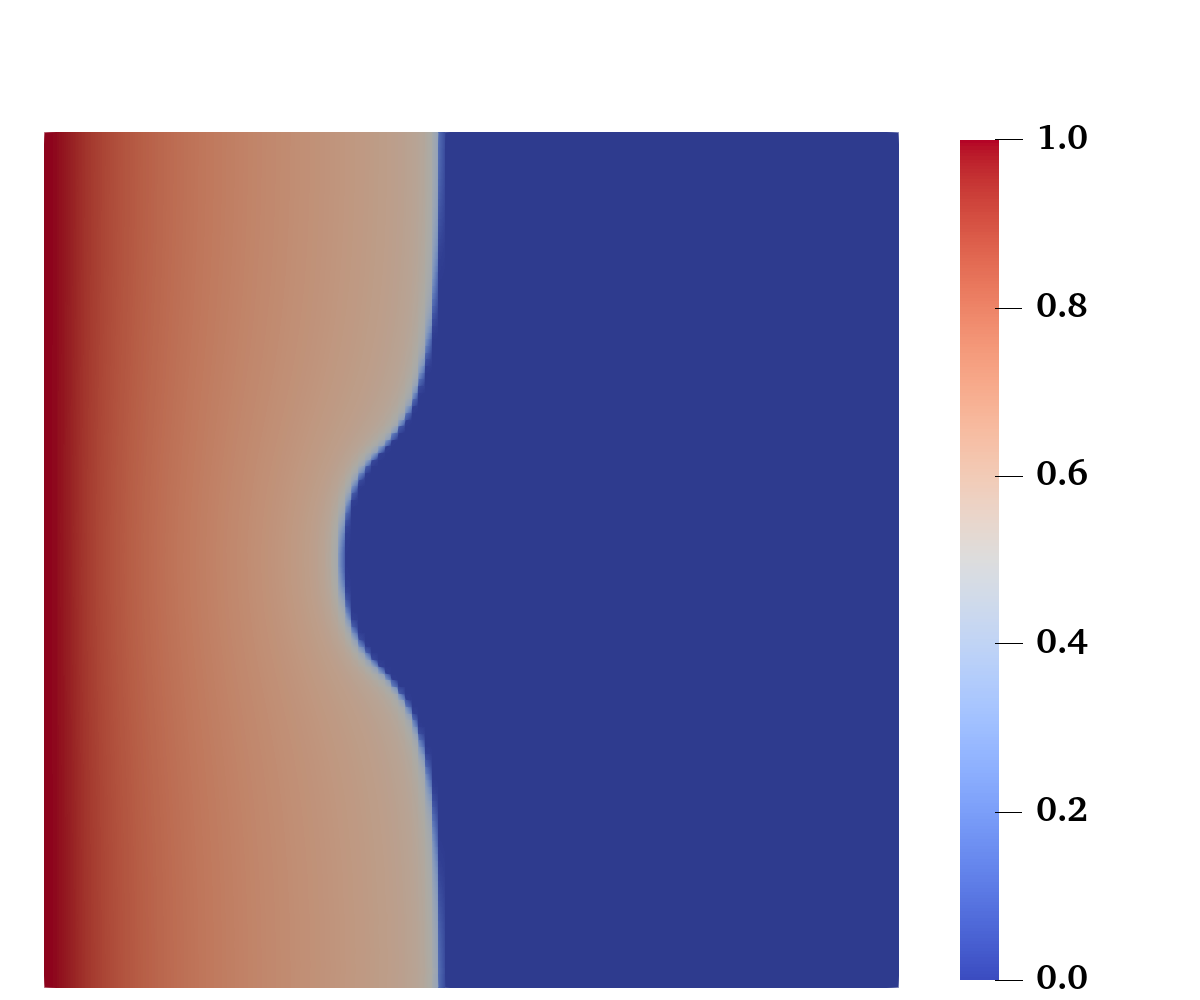}
        \caption{Left: $S$, Right:$\SOF$ ; $t = 1.29$}
    \end{subfigure}
    % \begin{subfigure}{0.3\linewidth}
    % \centering
    %     \includegraphics[trim = 0 0 100 120,clip,width=\linewidth]{noindicator_saturation_ts250.png}
    %     \caption{$\SOF$ ; $t = 1.29$}
    % \end{subfigure}
    % \hfill
    \raisebox{0.5cm}{\rule{0.6pt}{4cm}}
    % \hfill
    \begin{subfigure}{\singlefigurewidth}
    \centering
        \includegraphics[trim = 0 0 100 120,clip,width=\linewidth]{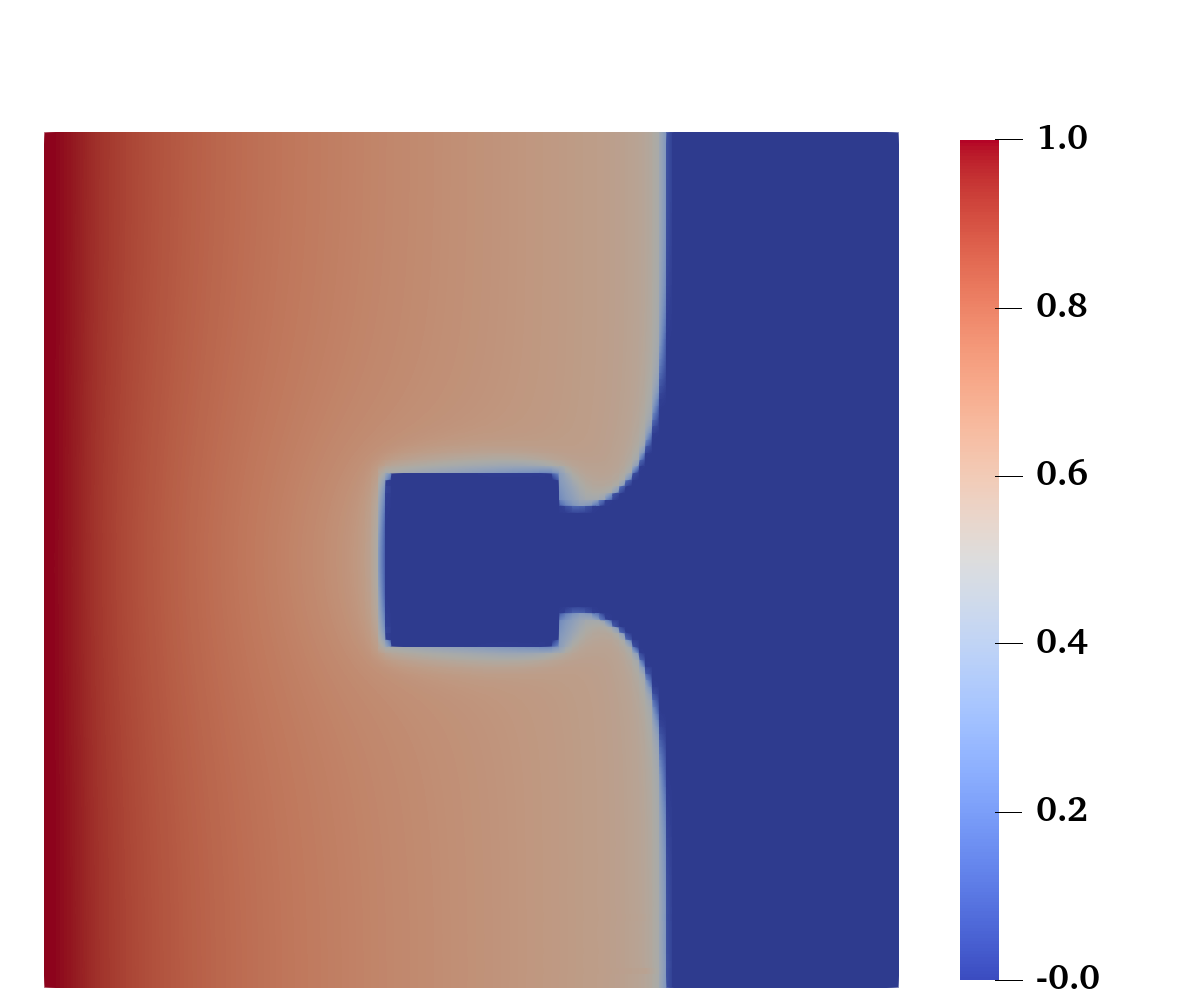}
        \caption{$\SOF$ ; $t = 2$}
    \end{subfigure}
    \caption{(Example~2-2) Left: an illustration of the necessity for the bound-preserving anti-oscillation postprocessing. The simulation without any limiter breaks before reaching the final time. Right: the saturation $\SOF$ at $\T =2$.}
    \label{ex3-01}
\end{figure}

Next, we illustrate the result with the troubled-cell detector. The smoothness indicator is defined as in \eqref{eq:smooth_indicator_x_direction}. We compare the two cases of indicators: ~i) mark cells when $\delta_{i} \geq \SI{1e-3}{}$ ~ii) mark cells when $\delta_{i} \geq \SI{1e-5}{}$. We expect that case ~ii) mark more cells than ~i), resulting in applying the minmod limiter to larger region. 
The result is displayed in Figure~\ref{ex3-02}, where subfigures (a) and (b) show the cell indicators for cases ~i) and ~ii), respectively. As expected, a larger number of cells are marked in case ~ii). The subfigures (c) and (d) compare the corresponding solution profiles at time $\T = 2$. Oscillatory behavior is observed in (c), whereas no visible oscillations are observed in (d).

\begin{figure}[ht!]
\centering
    \begin{subfigure}{0.3\linewidth}
    \centering
        \includegraphics[trim = 0 0 100 120,clip,width=\linewidth]{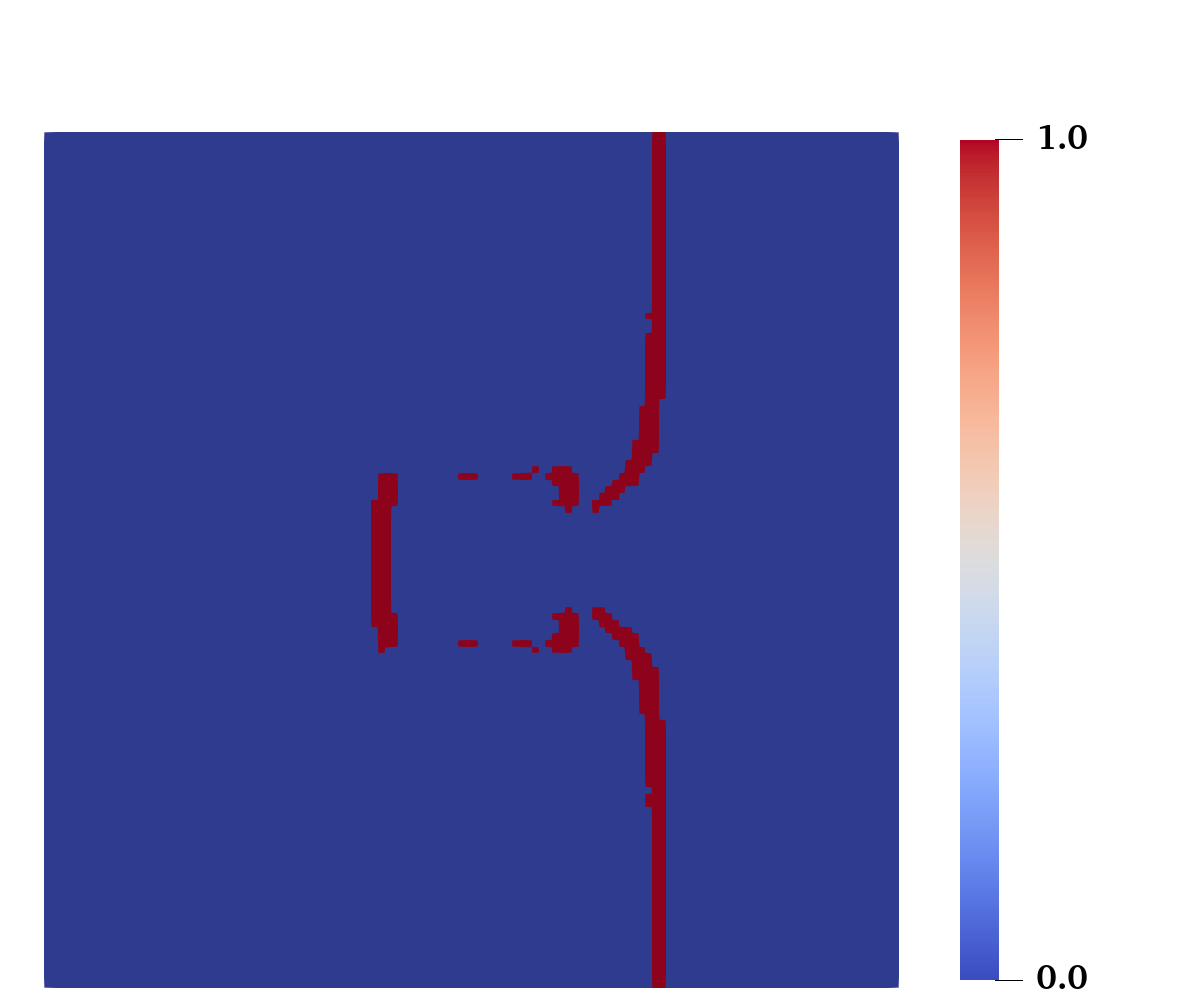}
        \caption{$\delta_{i}\geq\SI{1e-3}{}$ ; $\T = 2$}
    \end{subfigure}
    \begin{subfigure}{0.3\linewidth}
    \centering
        \includegraphics[trim = 0 0 100 120,clip, width=\linewidth]{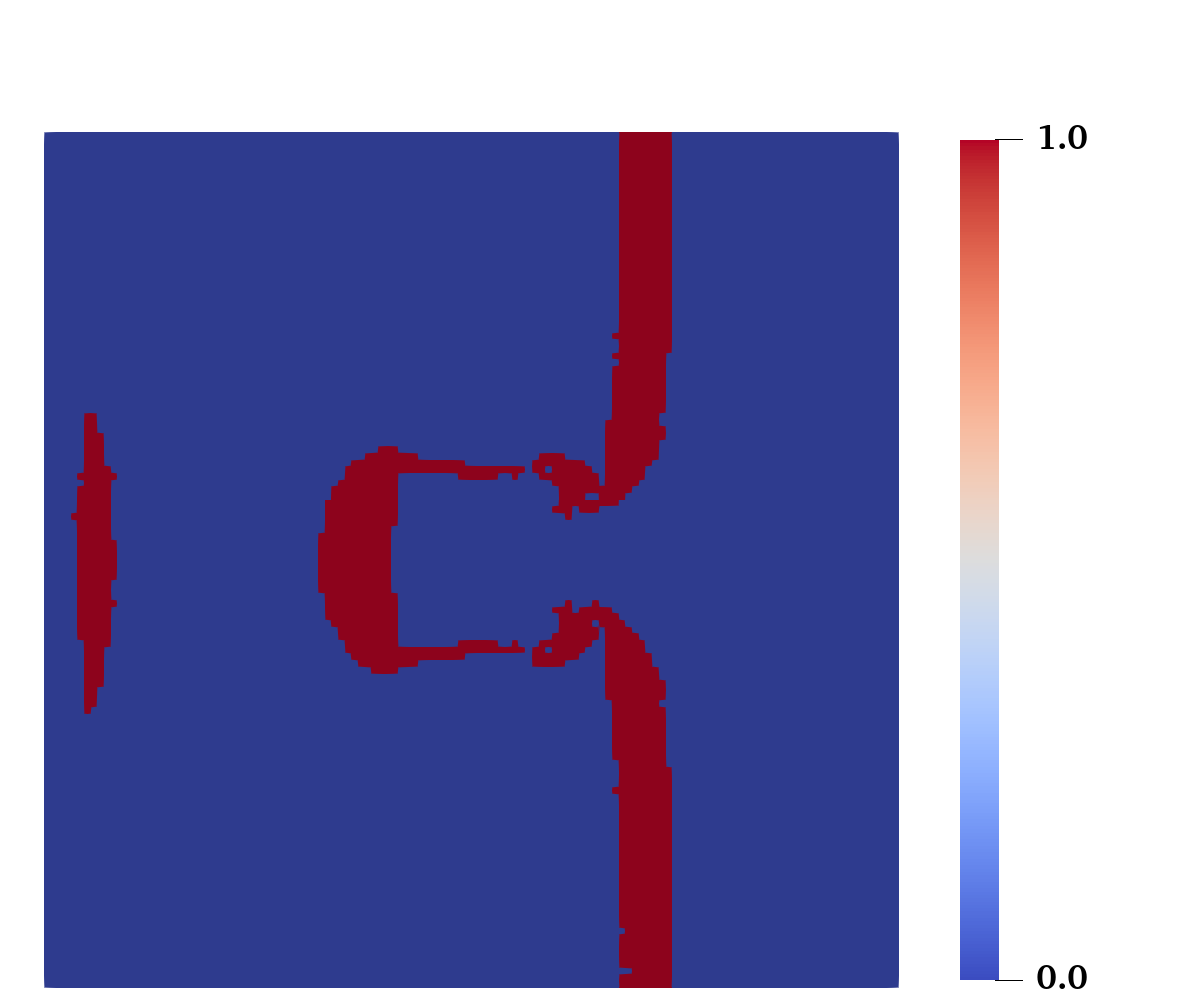}
        \caption{$\delta_{i}\geq\SI{1e-5}{}$ ; $\T = 2$}
    \end{subfigure}\\
    \vspace{0.2cm}
    \begin{subfigure}{0.3\linewidth}
    \centering
        \includegraphics[trim = 0 0 100 120,clip,width=\linewidth]{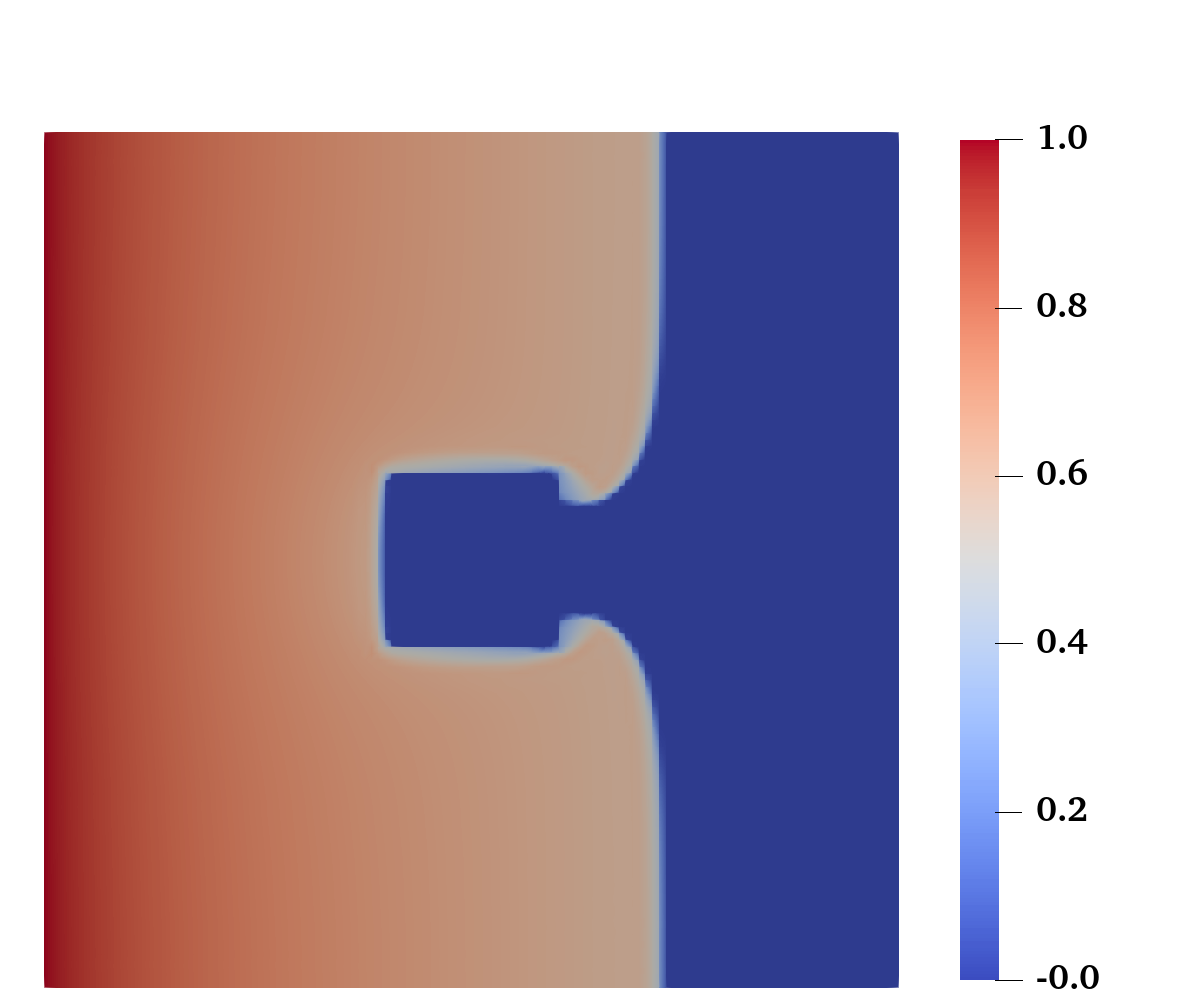}
        \caption{$\SOFD$ with $\delta_{i}\geq\SI{1e-3}{}$; $\T = 2$}
    \end{subfigure}
    \begin{subfigure}{0.3\linewidth}
    \centering
        \includegraphics[trim = 0 0 100 120,clip,width=\linewidth]{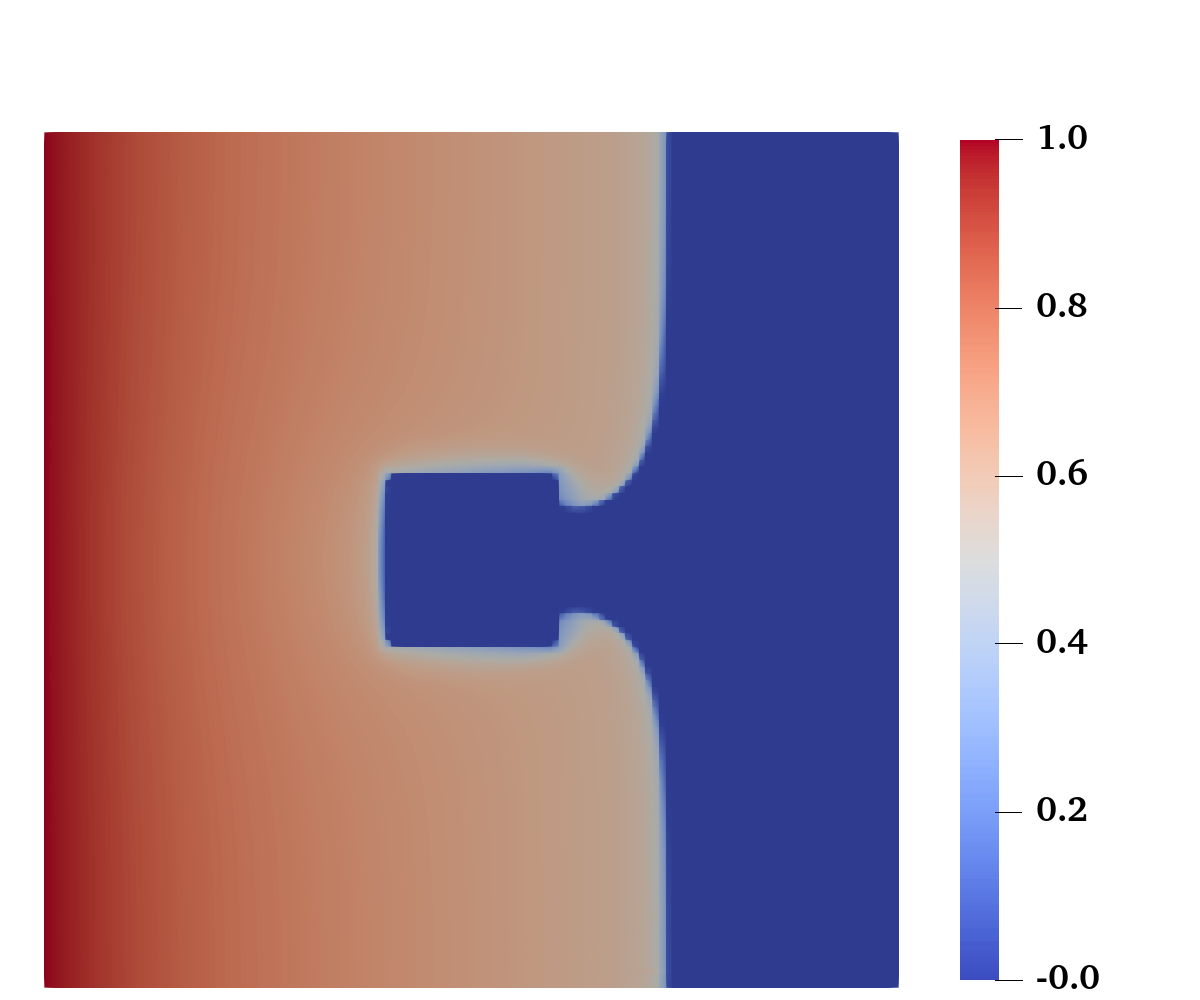}
        \caption{$\SOFD$ with $\delta_{i} \geq \SI{1e-5}{}$; $\T = 2$}
    \end{subfigure}
    \caption{(Example~2-2) Top: cell indicators with case i) $\delta_{i} \geq \SI{1e-3}{}$ and ii) $\delta_{i}\geq\SI{1e-5}{}$. Bottom: associated $\SOFD$ at $\T = 2$.}
    \label{ex3-02}
\end{figure}

Now, we present the numerical results for a non-matching mesh. We introduce a non-matching mesh with local refinement around the square region $[0.4,0.6]^2$, where the permeability is lower than the rest of the domain. Note that $h_{\max}=2h_{\min}$. 

Figure~\ref{ex3-nonmatching-01} demonstrates the effect of the BP+OF limiters for the non-matching mesh case. Figure~\ref{ex3-nonmatching-01}(a) compares the solution $S$ and $\SOF$ at time $t=0.1$. While $S$ exhibits significant oscillations, $\SOF$ shows that BP+OF limiters effectively suppress these oscillations and keep the saturation bounded. $\SOF$ at the final time $\T=2$ is presented in Figure~\ref{ex3-nonmatching-01}(b).

\begin{figure}[ht!]
\centering
% \newlength{\singlefigurewidth}
\setlength{\singlefigurewidth}{0.29\linewidth}
    % \begin{subfigure}{0.3\linewidth}
    % \centering
    %     \includegraphics[trim = 0 0 100 120,clip,width=\linewidth]{nonmatching_noAB_saturation_ts199.png}
    %     \caption{Saturation with no limiters ; $t = 1$}
    % \end{subfigure}
    % \begin{subfigure}{0.3\linewidth}
    % \centering
    %     \includegraphics[trim = 0 0 100 120,clip, width=\linewidth]{nonmatching_noindicator_saturation_ts199.png}
    %     \caption{Saturation with BP+OF ; $t = 1$}
    % \end{subfigure}\\
    % \vspace{0.2cm}
    \begin{subfigure}{0.6\linewidth}
    \centering
        \includegraphics[trim = 0 0 100 120,clip,width=\singlefigurewidth]{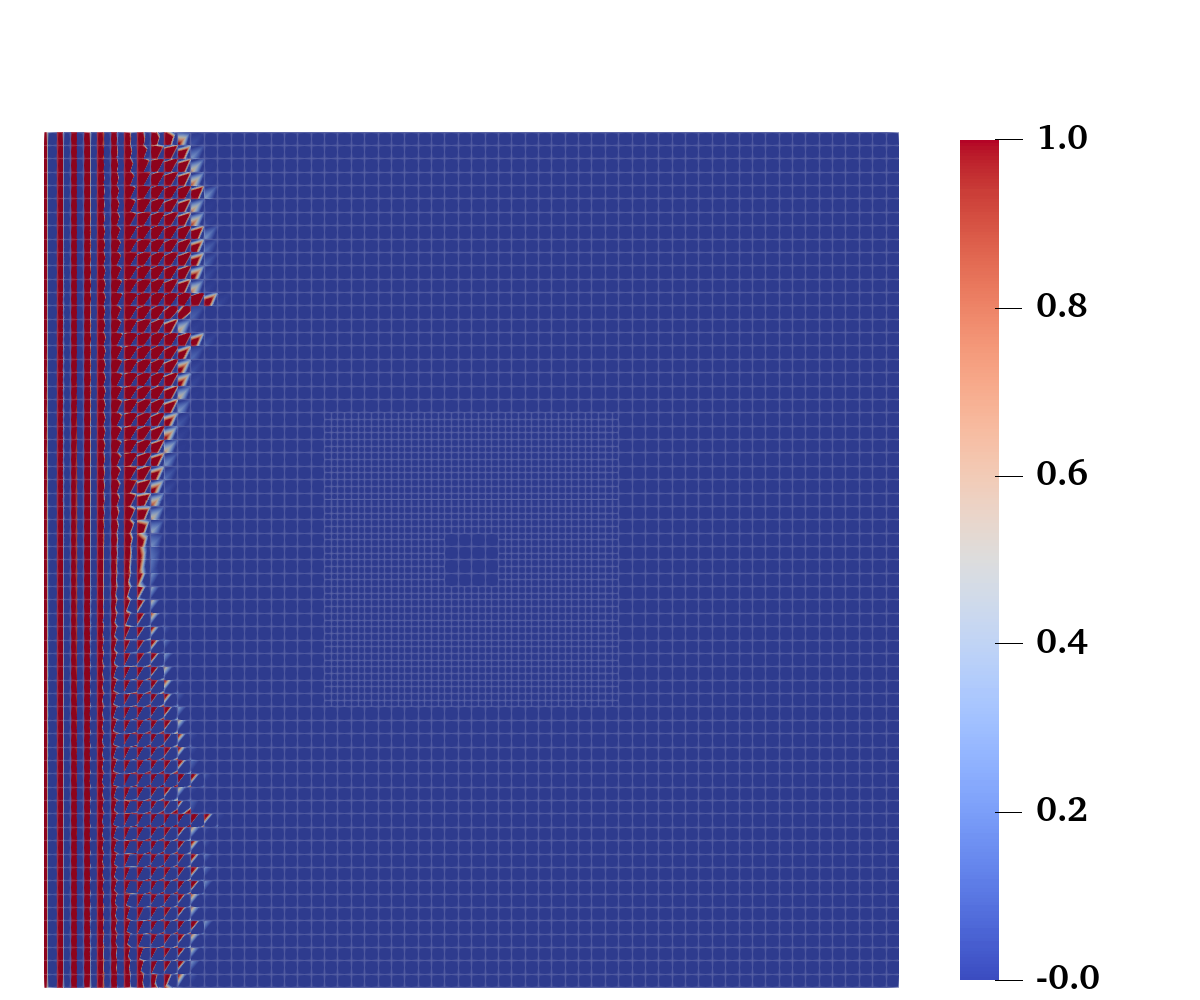}
        \includegraphics[trim = 0 0 100 120,clip,width=\singlefigurewidth]{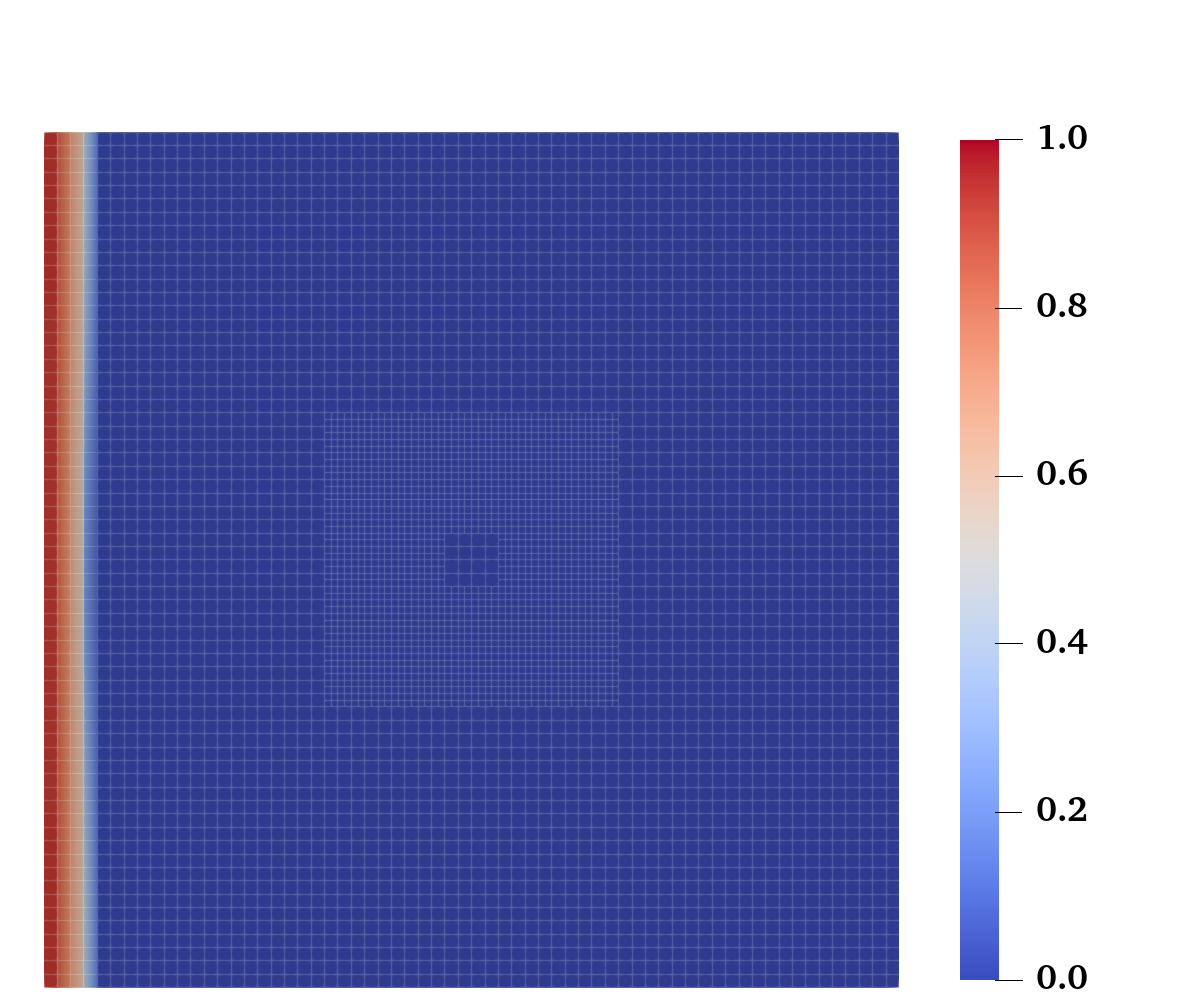}
        \caption{Left: $S$, Right:$\SOF$ ; $t = 0.1$}
        % \caption{$S$ ; $t = 0.1$}
    \end{subfigure}
    \raisebox{0.5cm}{\rule{0.6pt}{4cm}}
    \begin{subfigure}{\singlefigurewidth}
    \centering
        \includegraphics[trim = 0 0 100 120,clip,width=\linewidth]{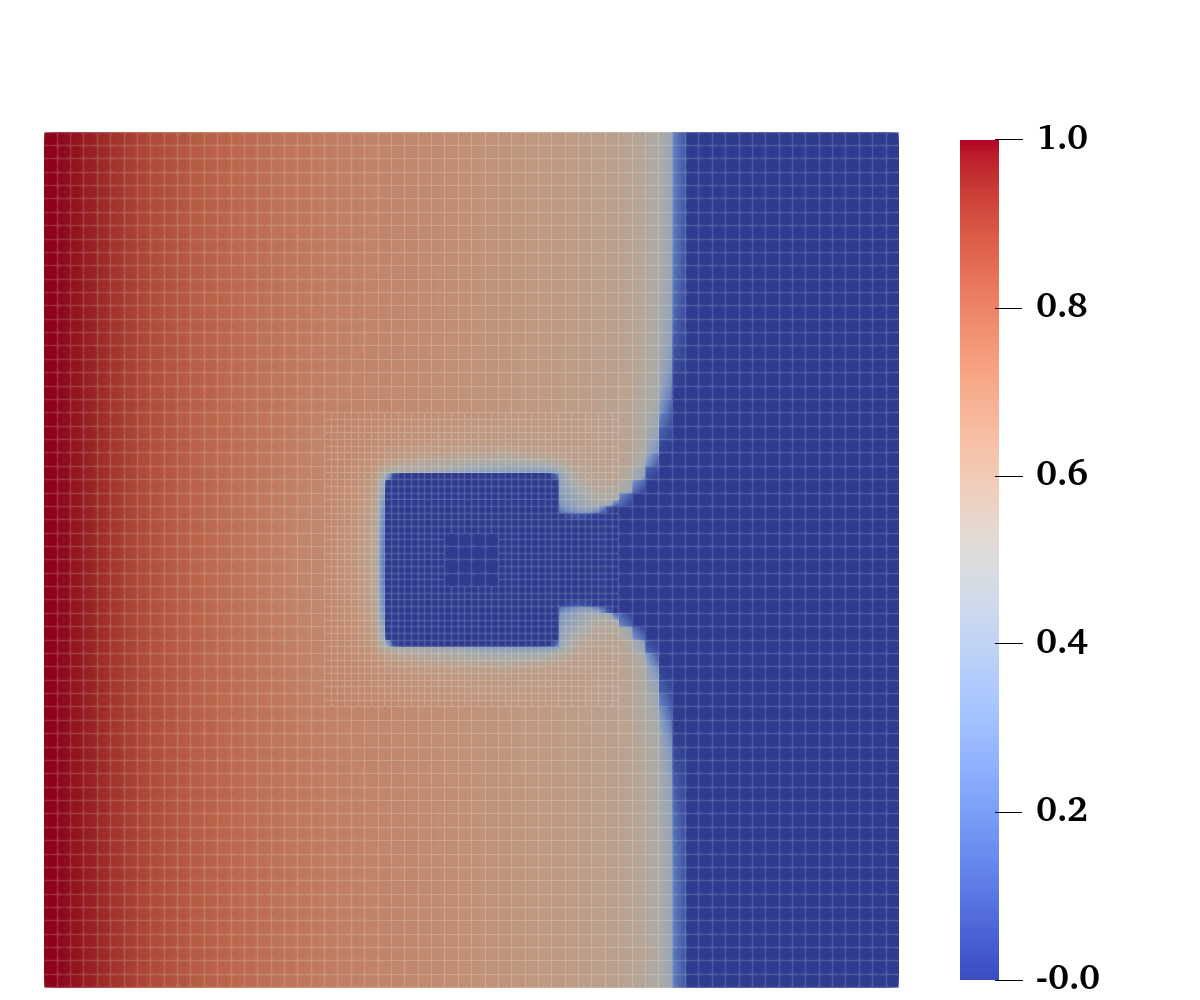}
        \caption{$\SOF$ ; $t = 2$}
    \end{subfigure}
    \caption{(Example~2-2, Non-matching mesh) $S$ and $\SOF$ at time $t = 0.1$ and $\SOF$ at time $t = 2$.}
    \label{ex3-nonmatching-01}
\end{figure}

Finally, we examine the troubled-cell detector for the non-matching mesh case. The smoothness indicator is defined as in \eqref{eq:smooth_indicator_x_direction}, and two thresholds are compared:~(i) $\delta_{i}\geq \SI{1e-3}{}$ and~(ii) $\delta_{i}\geq \SI{1e-5}{}$. The lower threshold in case~(ii) marks more cells as troubled cells, which means that the OF limiter is applied to a larger portion of the domain. Figure~\ref{ex3-nonmatching-02} (a)-(b) shows the corresponding cell indicators for each threshold, where a greater number of cells are detected in case~(ii). This demonstrates that a stricter troubled cell detection threshold enhances oscillation control for the non-matching mesh problem as well.

\begin{figure}[ht!]
\centering
    \begin{subfigure}{0.3\linewidth}
    \centering
        \includegraphics[trim = 0 0 100 120,clip,width=\linewidth]{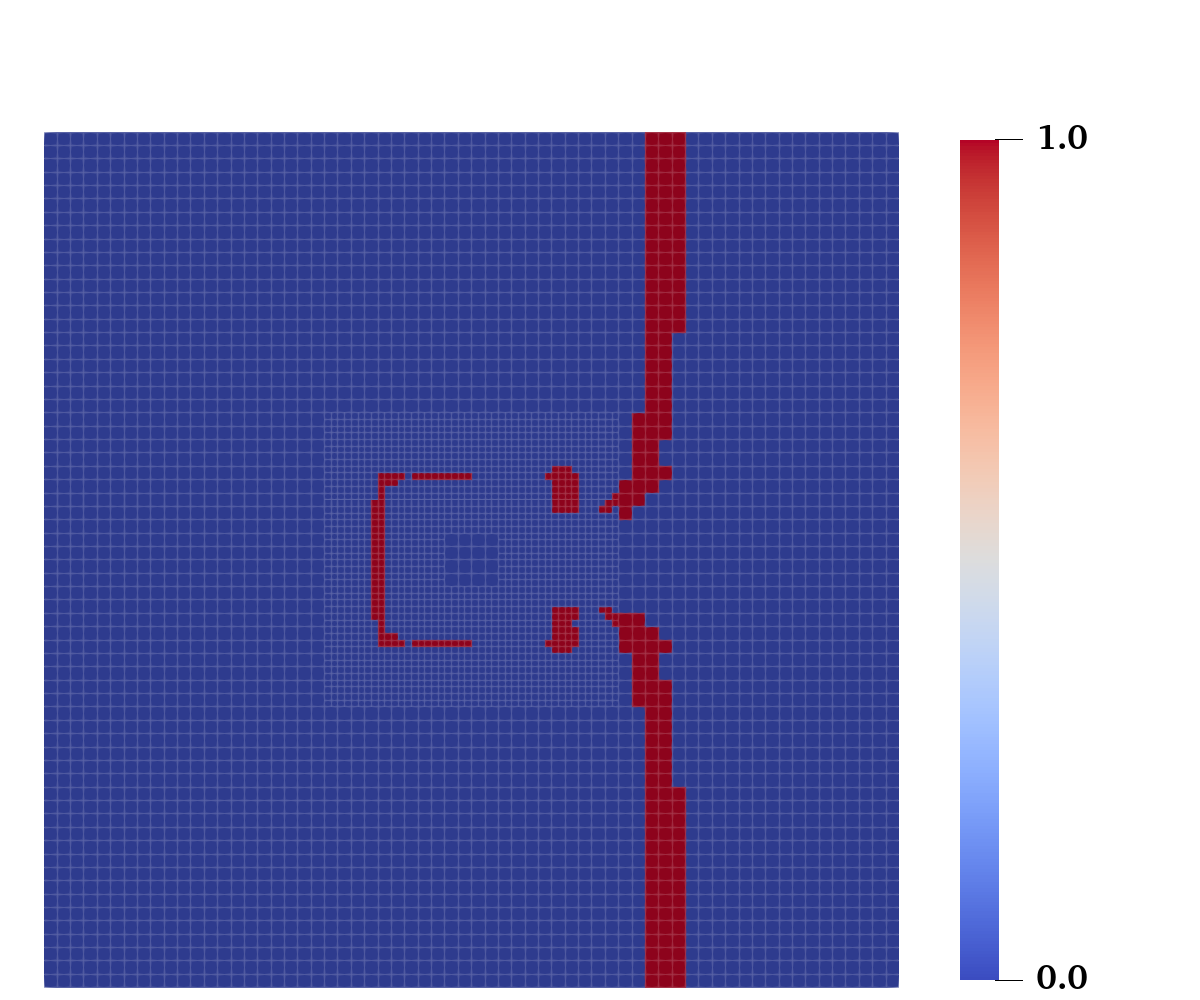}
        \caption{$\delta_{i}\geq\SI{1e-3}{}$ ; $t = 2$}
    \end{subfigure}
    \begin{subfigure}{0.3\linewidth}
    \centering
        \includegraphics[trim = 0 0 100 120,clip, width=\linewidth]{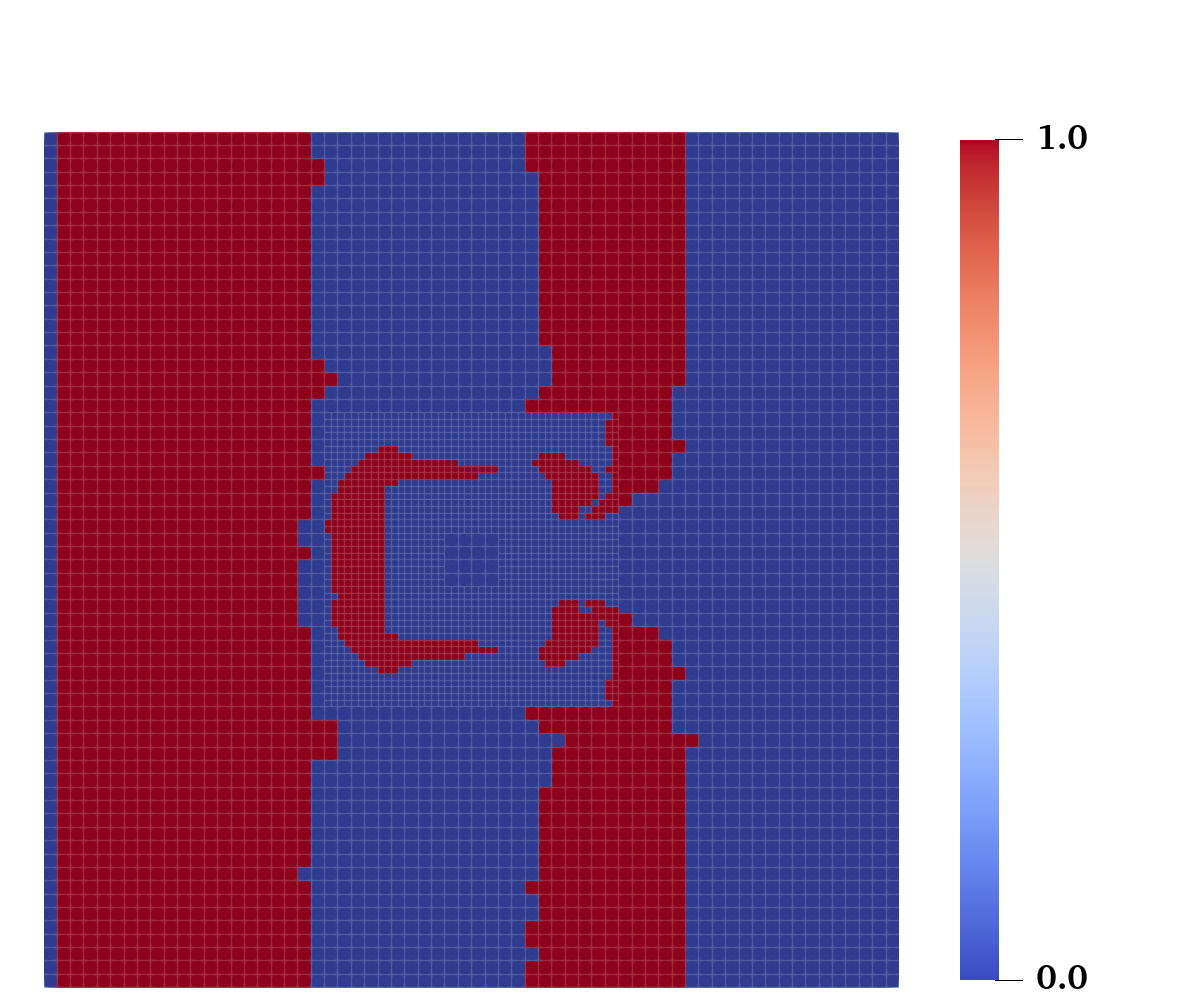}
        \caption{$\delta_{i}\geq\SI{1e-5}{}$ ; $t = 2$}
    \end{subfigure}\\
    \vspace{0.2cm}
    \begin{subfigure}{0.3\linewidth}
    \centering
        \includegraphics[trim = 0 0 100 120,clip,width=\linewidth]{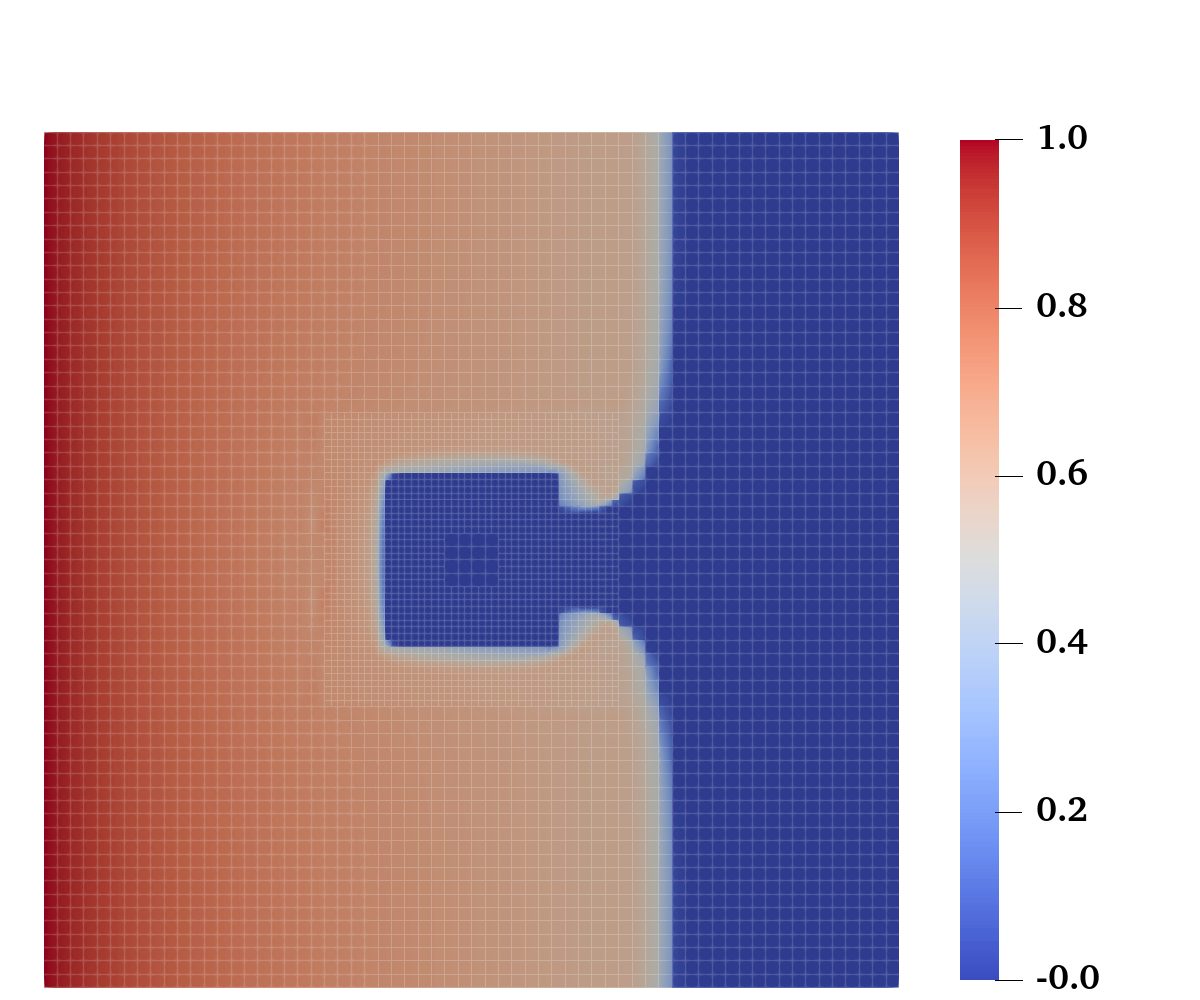}
        \caption{$\SOFD$ with $\delta_{i}\geq\SI{1e-3}{}$; $t = 2$}
    \end{subfigure}
    \begin{subfigure}{0.3\linewidth}
    \centering
        \includegraphics[trim = 0 0 100 120,clip,width=\linewidth]{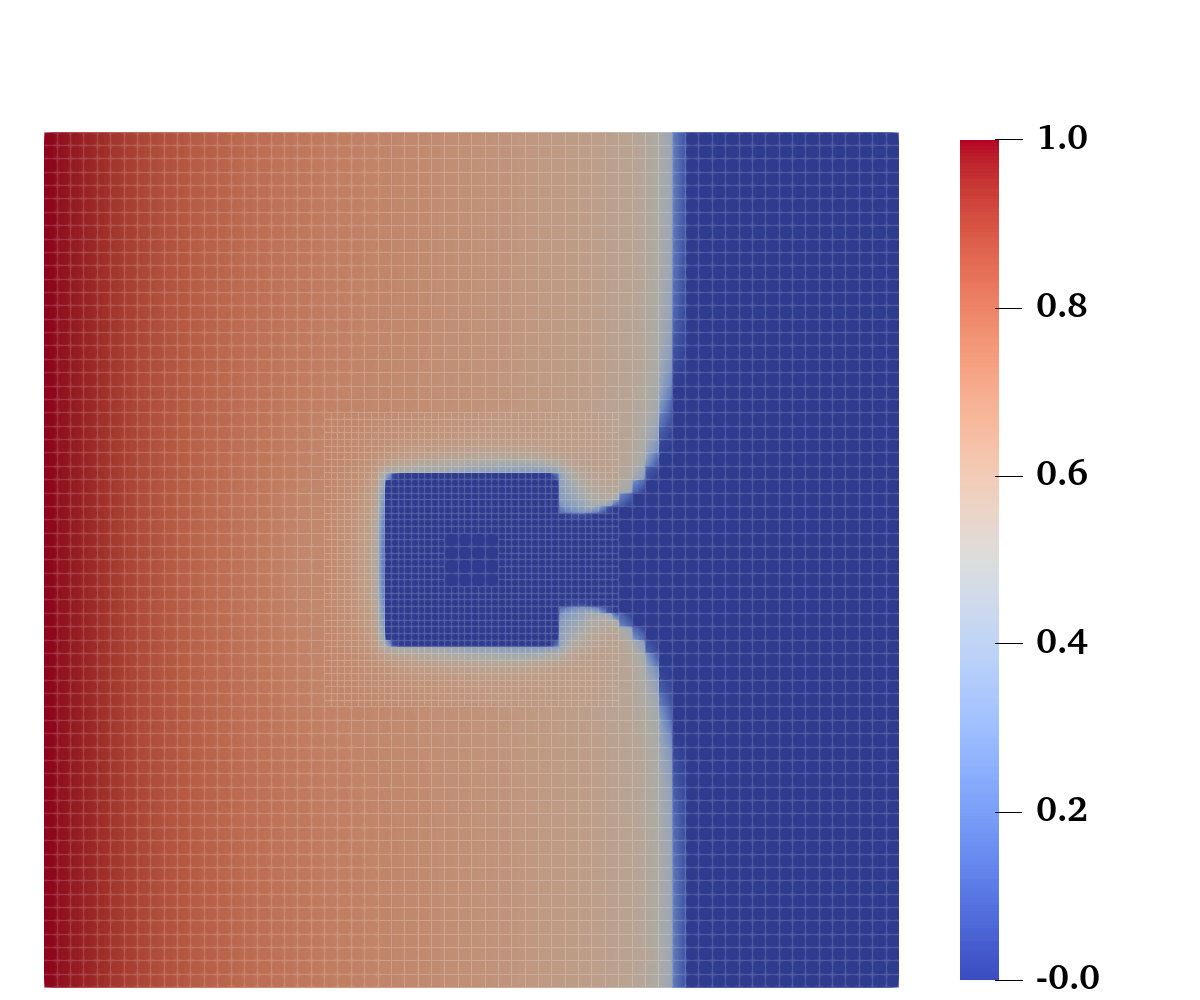}
        \caption{$\SOFD$ with $\delta_{i} \geq \SI{1e-5}{}$; $t = 2$}
    \end{subfigure}
    \caption{(Example~2-2) Top: cell indicators with case i) $\delta_{i} \geq \SI{1e-3}{}$ and ii) $\delta_{i}\geq\SI{1e-5}{}$. Bottom: associated $\SOFD$ at $\T = 2.$}
    \label{ex3-nonmatching-02}
\end{figure}

\subsection{Example~3: Flow in fractured domain}
In this example, we consider two-phase flow in a fractured porous medium defined on the domain $\Omega = [0,1]^2$. Flow in fractured porous media has attracted considerable interest in a variety of applications~\cite{yang2026heat,hu2025ensemble,lee2021modeling}. A network of fracture segments is distributed throughout the domain, as illustrated in Figure~\ref{fractures_domain}.
A network of fracture segments is distributed throughout the domain, as illustrated in Figure~\ref{fractures_domain}. 
The initial conditions are defined as $s_w(\vx,0)=p_w(\vx,0)=0$. The inflow boundary condition is prescribed as $s_w(\vx,\cdot)=1$ on $\{x=0\}\times[0,1]$, and Dirichlet boundary conditions for pressure are imposed as $p_w(\vx,\cdot)=1$ on $\{x=0\}\times[0,1]$ and $p_w(\vx,\cdot)=0$ on $\{x=1\}\times[0,1]$. No-flow conditions $\vu_N = 0$ are prescribed on the top and bottom boundaries. 
The time step is chosen to be $\dt=0.01$ where the final time is $\T=25$. The permeability field corresponding to the fractured medium is shown in Figure~\ref{ex4-permeability}. The fracture segments are modeled as high-permeability with permeability $K_D=10^{-1}$, whereas the permeability of the surrounding matrix is $K_D=10^{-3}$. 
In addition, we employ local mesh refinement around fracture boundaries to capture fracture geometry and sharp permeability transitions accurately. The resulting non-matching mesh is constructed such that $h_{\max}=2h_{\min}$ where $h_{\min}=1/128$. 

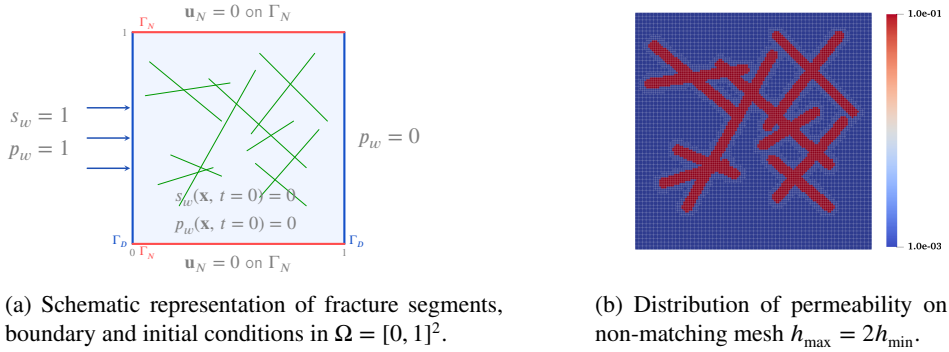
\begin{figure}[ht!]
\centering

    \begin{subfigure}{0.4\linewidth}
        \scalebox{0.4}{\begin{tikzpicture}[
    line cap=round,
    line join=round,
    font=\normalsize,
]
 
\def\L{7} 
 
\fill[domainbg] (0,0) rectangle (\L,\L);
 
\def\bx{2.8}  \def\by{2.8}  
\def\bw{1.4}  \def\bh{1.4}   
% \fill[blockgreenlight] (\bx,\by) rectangle (\bx+\bw,\by+\bh);
% \draw[blockgreen, line width=1.6pt] (\bx,\by) rectangle (\bx+\bw,\by+\bh);
 \draw[green!60!black, line width=0.8pt] (1.26,2.94) -- (2.94,2.24);
\draw[green!60!black, line width=0.8pt] (1.54,1.26) -- (4.06,5.74);
\draw[green!60!black, line width=0.8pt] (2.52,5.46) -- (5.74,2.52);
\draw[green!60!black, line width=0.8pt] (3.78,3.08) -- (5.32,4.06);
\draw[green!60!black, line width=0.8pt] (4.20,1.54) -- (6.16,3.78);

\draw[green!60!black, line width=0.8pt] (0.56,6.02) -- (2.94,4.06);
\draw[green!60!black, line width=0.8pt] (0.42,4.90) -- (3.22,5.32);
\draw[green!60!black, line width=0.8pt] (0.84,1.96) -- (2.66,2.52);
\draw[green!60!black, line width=0.8pt] (3.92,3.64) -- (6.02,6.30);
\draw[green!60!black, line width=0.8pt] (4.20,6.30) -- (6.44,4.06);
\draw[green!60!black, line width=0.8pt] (4.06,2.66) -- (5.74,1.26);
\draw[labelblue, line width=2pt] (0,0) -- (0,\L);
\draw[labelblue, line width=2pt] (\L,0) -- (\L,\L);
\draw[red!70, line width=2pt] (0,0) -- (\L,0);
\draw[red!70, line width=2pt] (0,\L) -- (\L,\L);

\node[below] at (0,    -0.05) {$0$};
\node[below] at (\L,   -0.05) {$1$};
\node[left]  at (-0.05, \L)   {$1$};
 
% \node[font=\LARGE] at (1.4, \L-1.0)
%       {$\mathbf{K} = 10^{-1}\mathbf{I}$};
 
% \node[font=\LARGE, align=center] at (\bx+\bw/2, \by+\bh/2)
%       {$\mathbf{K} = 10^{-6}\mathbf{I}$};
 
\foreach \y in {2.5, 3.5, 4.5}{
  \draw[arrowcol, line width=1.4pt, -stealth]
        (-1.5,\y) -- (-0.08,\y);
}

\node[ font=\huge, anchor=east] at (-1.95, 4.2)
      {$s_w = 1$};
\node[ font=\huge, anchor=east] at (-1.95, 3.1)
      {$p_w = 1$};

\node[labelblue, font=\large, anchor=north east, align=right]
      at (-0.05, 0.4) {$\Gamma_{D}$};

\node[red!70, font=\large, anchor=north, align=right] at (0.5, 7.6) {$\Gamma_N$};
 
\node[font=\huge, anchor=west] at (\L+0.35, \L/2)
      {$p_w = 0$};
\node[labelblue, font=\large, anchor=north west, align=left]
      at (\L+0.05, 0.4) {$\Gamma_{D}$};

\node[red!70, font=\large, anchor=south, align=right] at (0.5, -0.6) {$\Gamma_N$};
 
\node[font=\LARGE] at (\L/2, -0.65)
      {$\mathbf{u}_N = 0 \text{ on } \Gamma_{N}$};
 
\node[font=\LARGE] at (\L/2, \L+0.65)
      {$\mathbf{u}_N = 0 \text{ on } \Gamma_{N}$};
 
\node[font=\LARGE, align=center] at (\L/2, 1.1) {%
  \begin{tabular}{l}
    $s_w(\mathbf{x},\,t=0) = 0$\\[4pt]
    $p_w(\mathbf{x},\,t=0) = 0$
  \end{tabular}
};

\end{tikzpicture}}
    \caption{Schematic representation of fracture segments, boundary and initial conditions in $\Omega = [0,1]^2$.}
    \label{fractures_domain}
    \end{subfigure}
    \hspace{1cm}
    \begin{subfigure}{0.3\linewidth}
    \centering
    \raisebox{0.3cm}{\includegraphics[trim = 0 0 50 120, clip, width=0.85\linewidth]{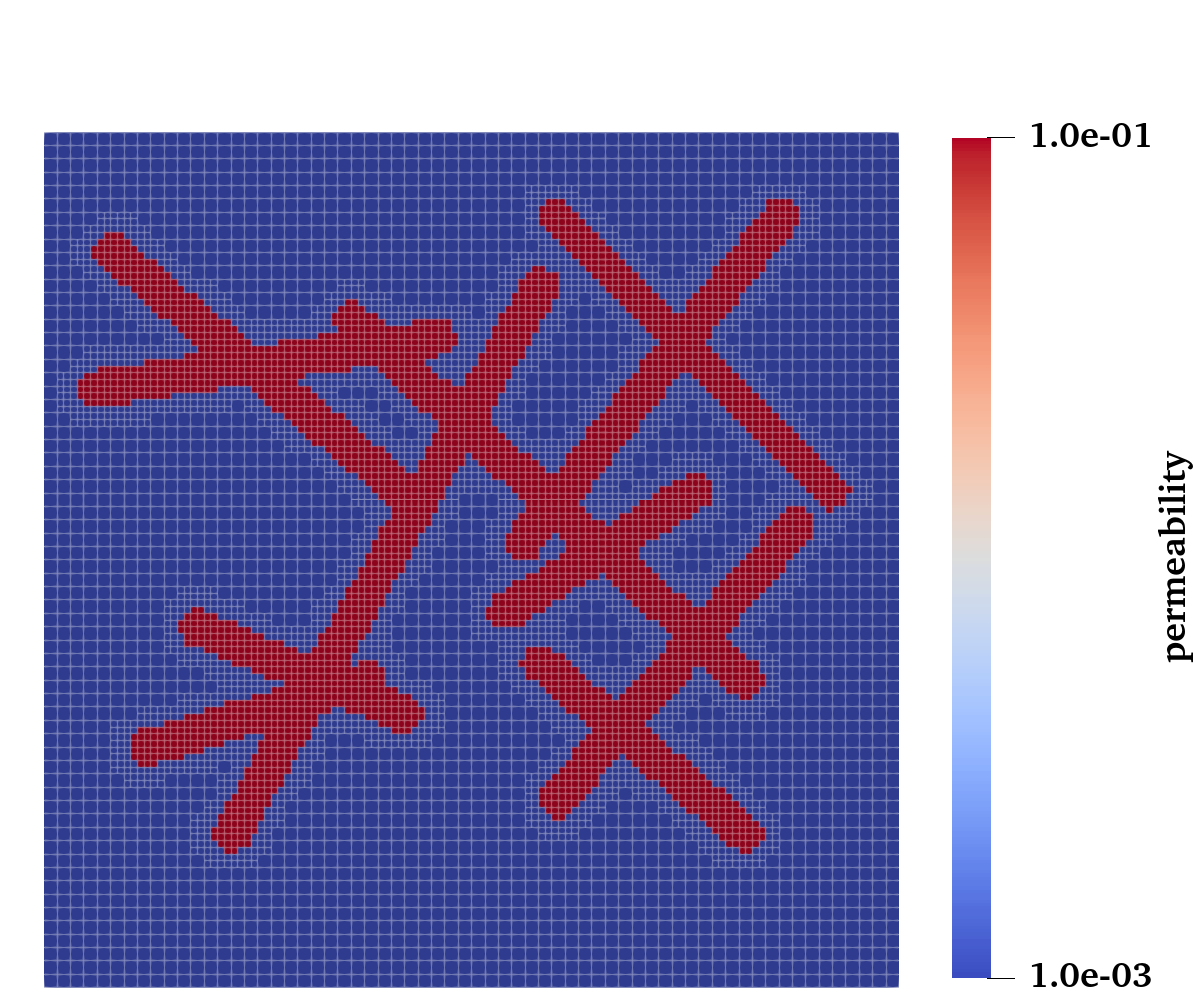}}
    \caption{Distribution of permeability on a non-matching mesh $h_{\max}=2h_{\min}$.}
    \label{ex4-permeability}
    \end{subfigure}
    \caption{(Example~3) (a) Fracture configuration with boundary conditions; (b) resulting permeability field on a non-matching mesh.}
    \label{ex4-setup}
\end{figure}

The evolution of the saturation with the BP+OF limiter applied is shown in Figure~\ref{ex4-bpof}. The saturation front initially propagates through the highly permeable fracture and gradually spreads into the surrounding porous matrix. As time propagates, the saturation profiles remain smooth and not oscillatory throughout the simulation.

We present the result obtained with the BP+OF limiter combined with a troubled-cell detector, using the smoothness indicator $\delta_{i}\geq\SI{1e-5}{}$ in Figure~\ref{ex4-detector}. The resulting saturation remains stable and free from oscillations. The solution is very similar to that obtained when the limiter is applied everywhere. However, the solutions with the detector are less diffusive within the fractures and near the fracture-matrix interfaces, allowing sharper saturation fronts to be preserved. This indicates that limiting only the troubled cells reduces unnecessary numerical dissipation while maintaining stability and the physically accurate solution.
\begin{figure}[ht!]
\centering
    \begin{subfigure}{0.3\linewidth}
    \centering
        \includegraphics[trim = 0 0 100 120,clip,width=\linewidth]{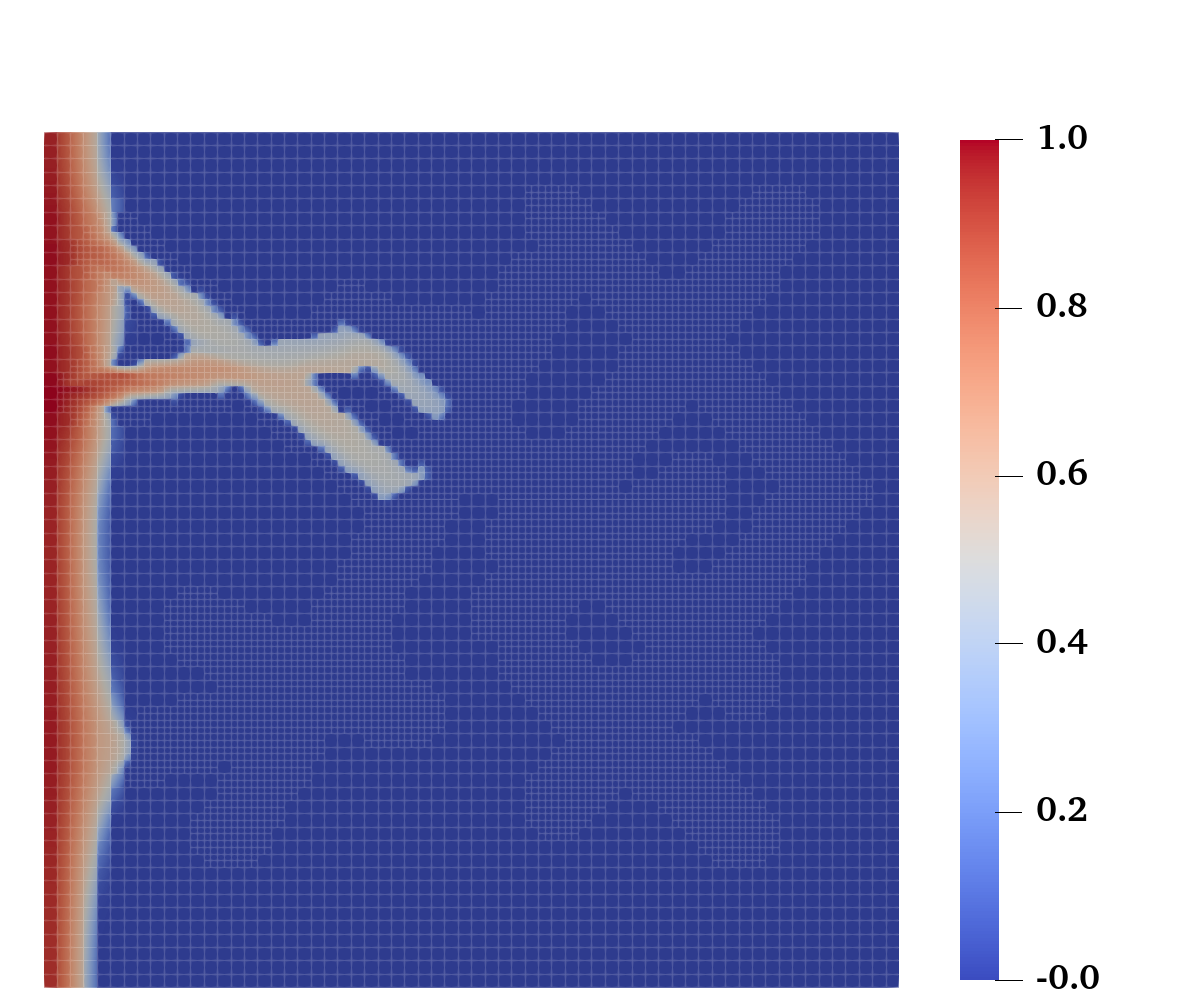}
        \caption{$\SOF$; $t=5$}
    \end{subfigure}
    \hfill
    \begin{subfigure}{0.3\linewidth}
    \centering
        \includegraphics[trim = 0 0 100 120,clip, width=\linewidth]{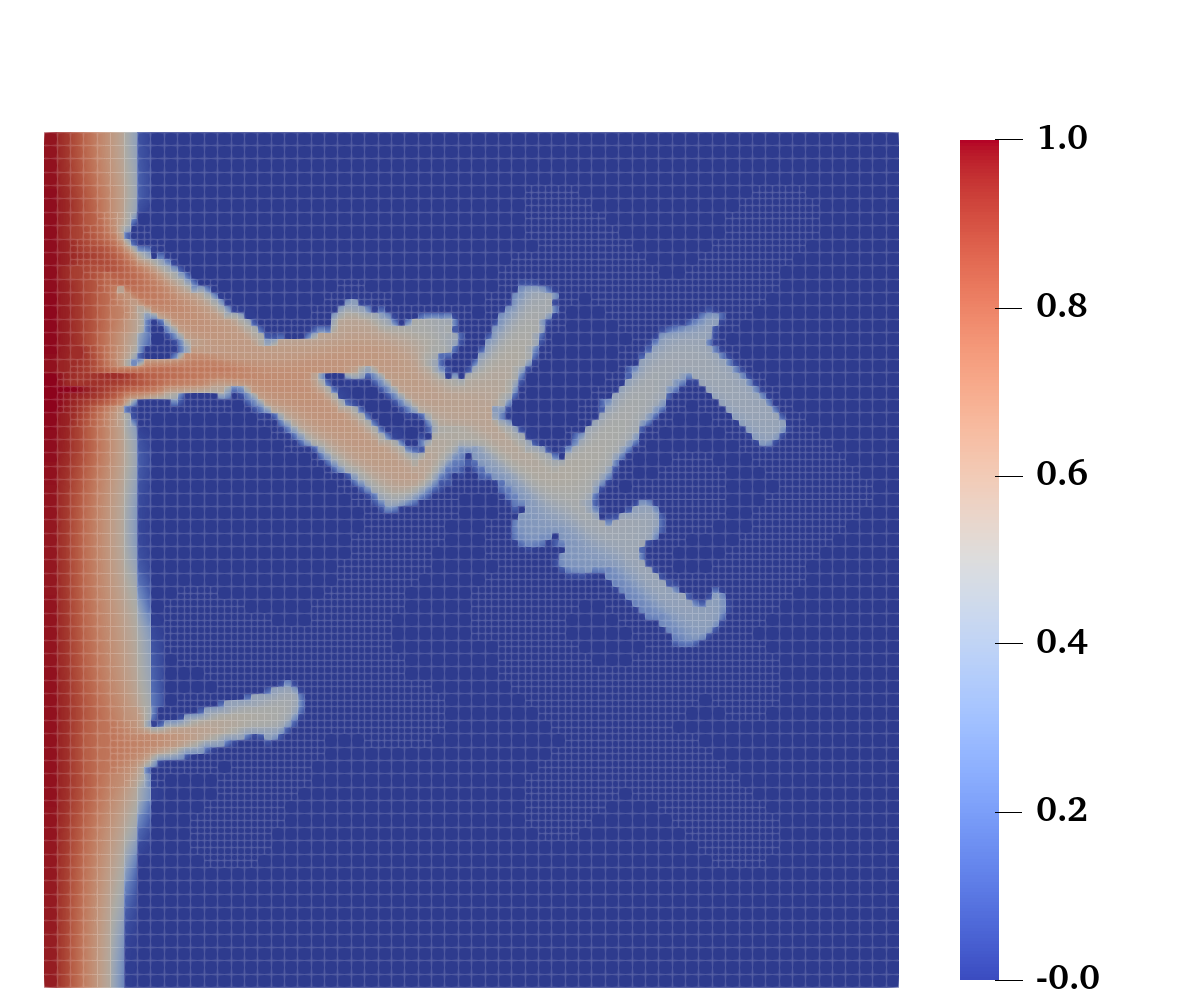}
        \caption{$\SOF$; $t=10$}
    \end{subfigure}
    % \vspace{0.5cm}
    \hfill
    \begin{subfigure}{0.3\linewidth}
    \centering
        \includegraphics[trim = 0 0 100 120,clip,width=\linewidth]{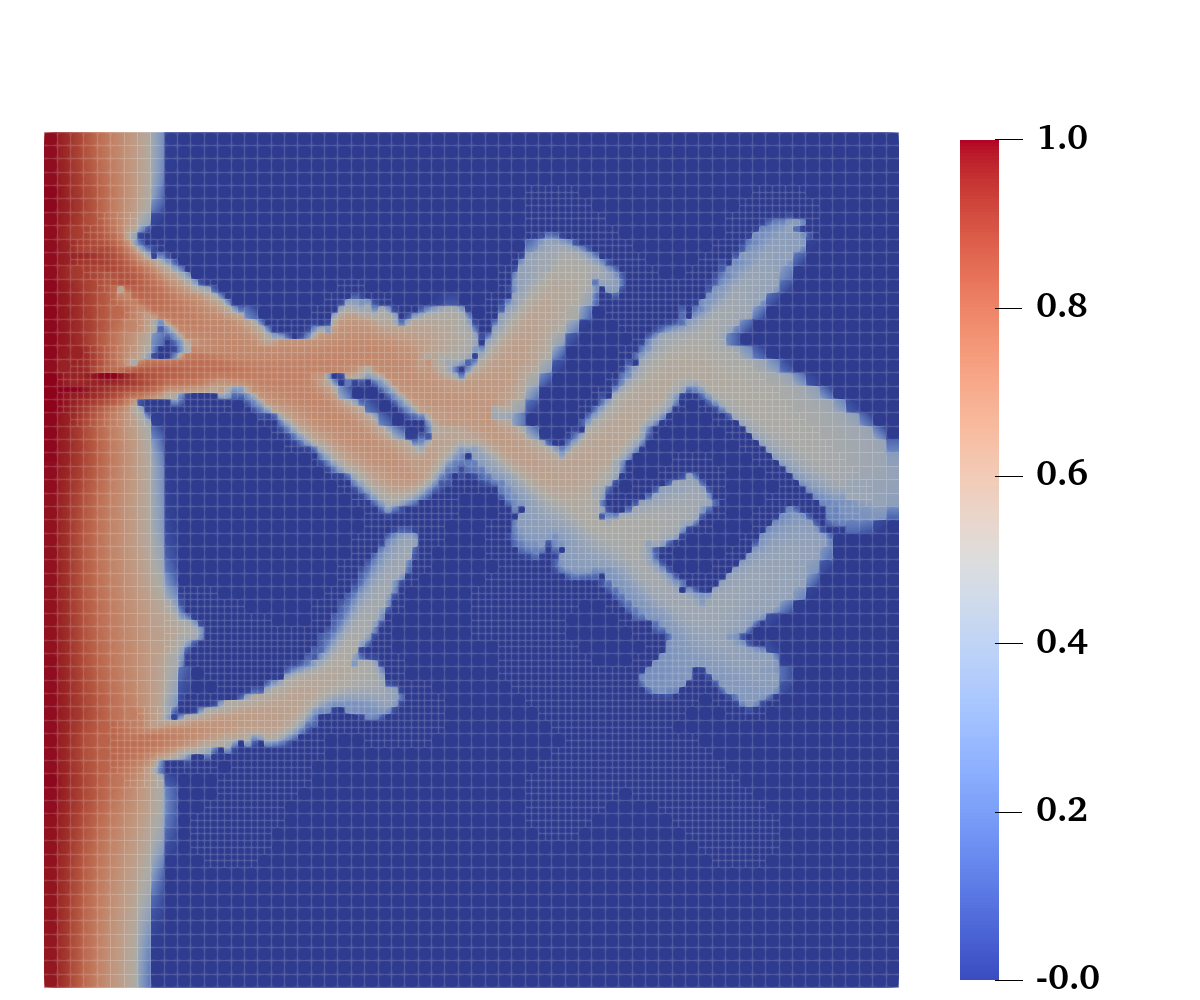}
        \caption{$\SOF$; $t=15$}
    \end{subfigure}
    
    \begin{subfigure}{0.3\linewidth}
    \centering
        \includegraphics[trim = 0 0 100 120,clip,width=\linewidth]{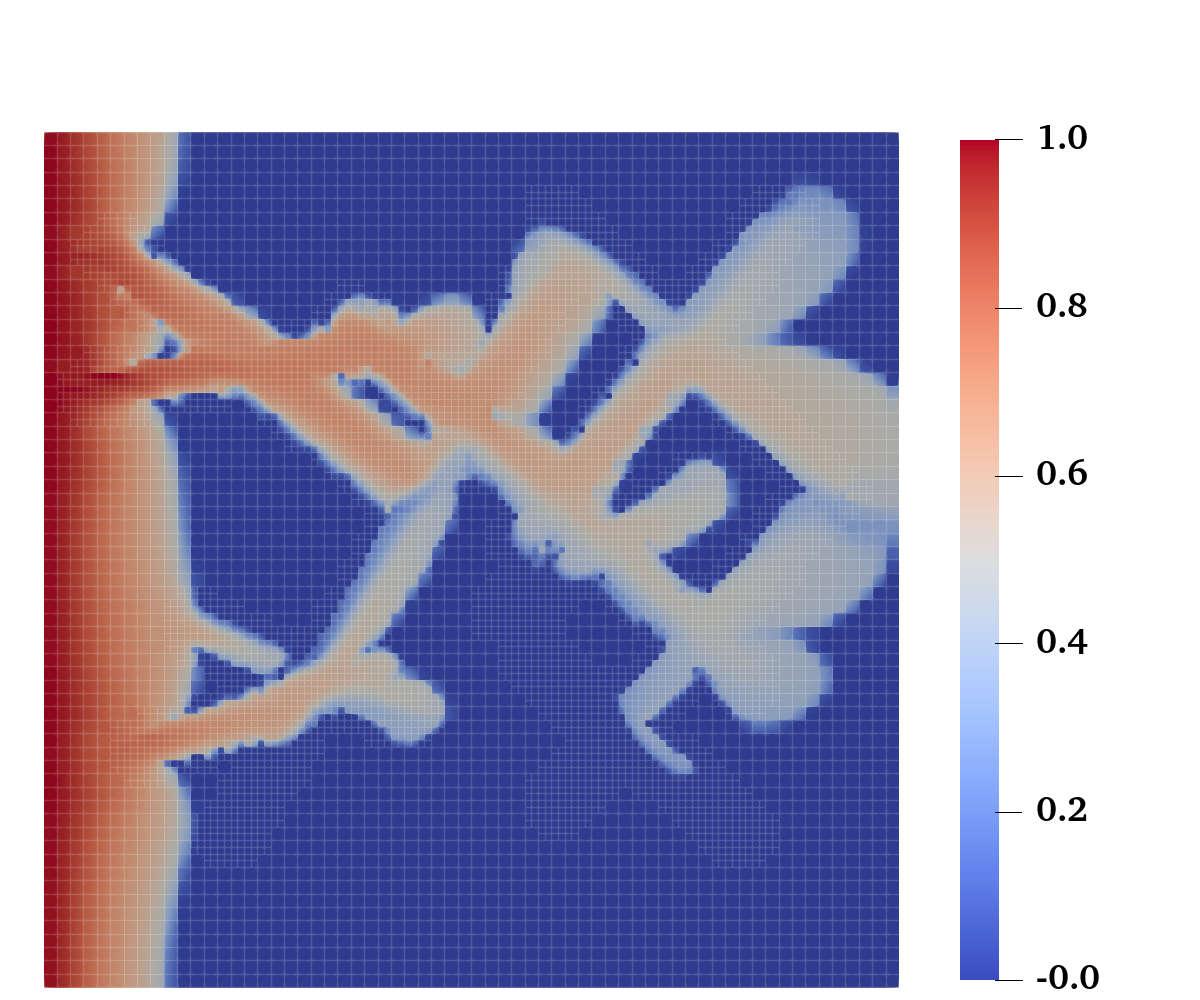}
        \caption{$\SOF$; $t=20$}
    \end{subfigure}
    % \hfill
    \hspace{1.5cm}
    \begin{subfigure}{0.3\linewidth}
    \centering
        \includegraphics[trim = 0 0 100 120,clip,width=\linewidth]{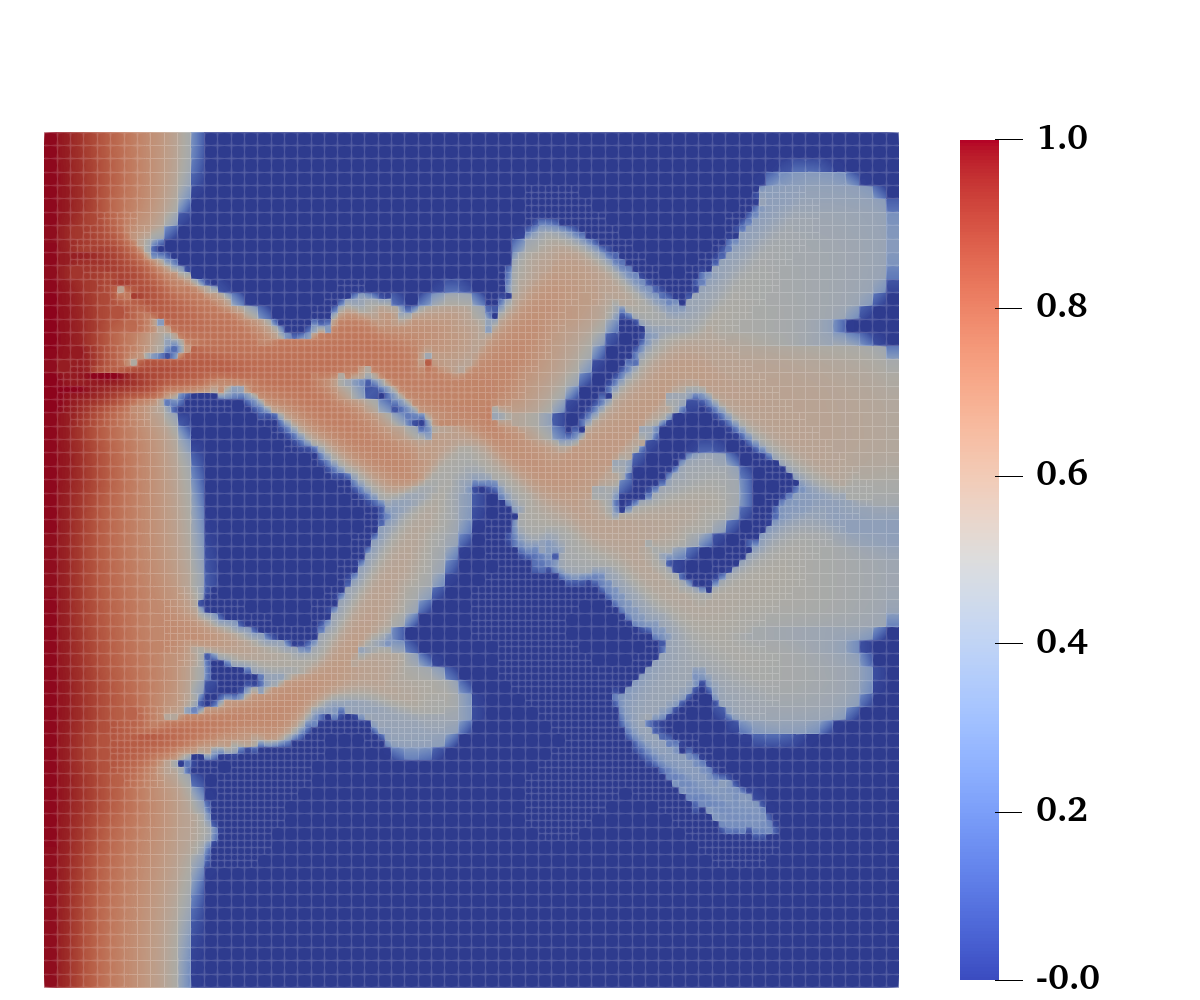}
        \caption{$\SOF$; $t=25$}
    \end{subfigure}
    % \hfill
    % \begin{subfigure}{0.3\linewidth}
    % \centering
    %     \includegraphics[trim = 0 0 100 120,clip,width=\linewidth]{ex4_bpof_2999.png}
    %     \caption{$\SOF$; $t=30$}
    % \end{subfigure}
    \caption{(Example~3) $\SOF$ at time $t = 5, 10, 15, 20$, and $t=25$.}
    \label{ex4-bpof}
\end{figure}

\begin{figure}[ht!]
\centering
    \begin{subfigure}{0.3\linewidth}
    \centering
        \includegraphics[trim = 0 0 100 120,clip,width=\linewidth]{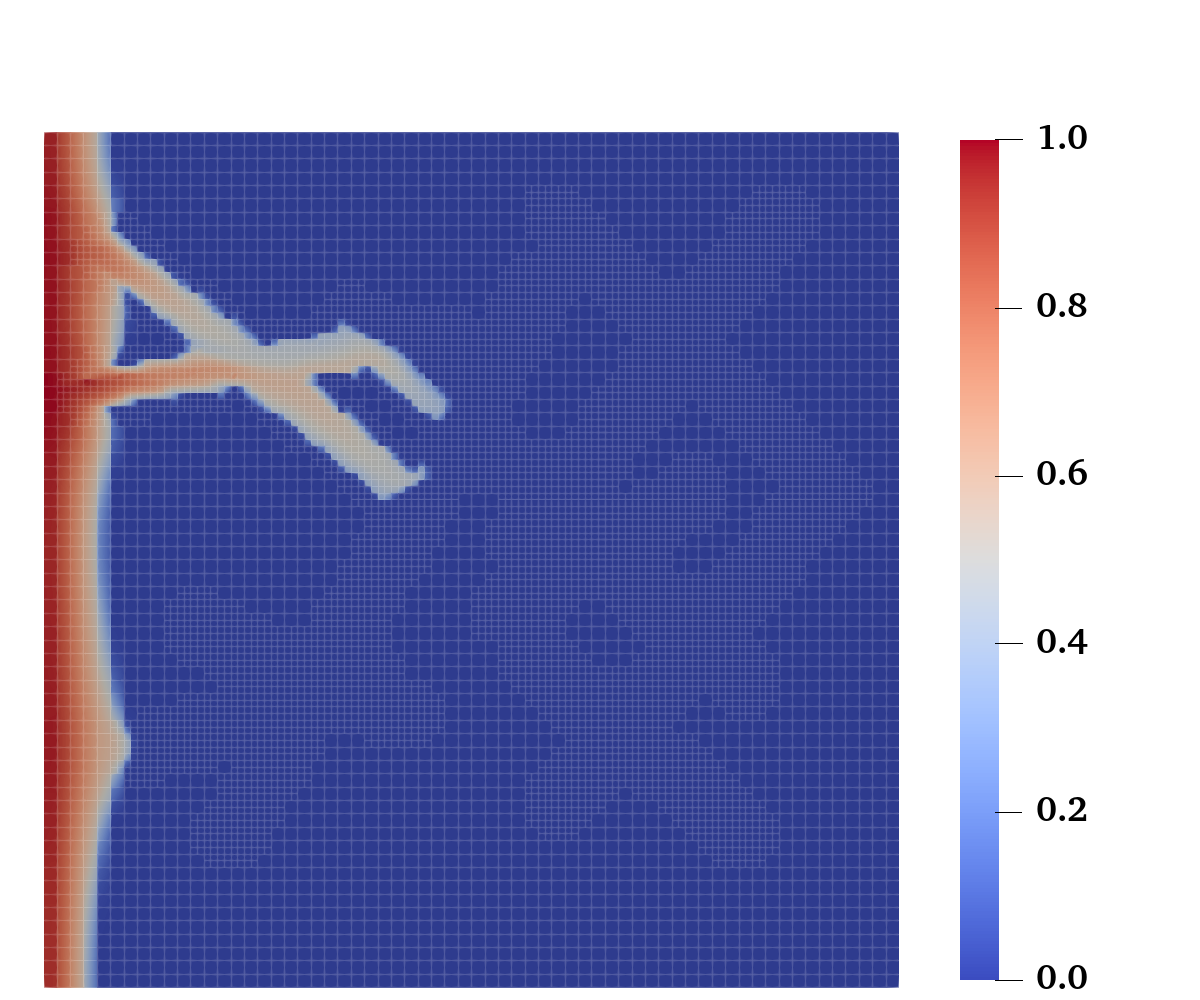}
        \caption{$\SOFD$ with $\delta_{i}\geq\SI{1e-5}{}$; $t = 5$}
    \end{subfigure}
    \hfill
    \begin{subfigure}{0.3\linewidth}
    \centering
        \includegraphics[trim = 0 0 100 120,clip, width=\linewidth]{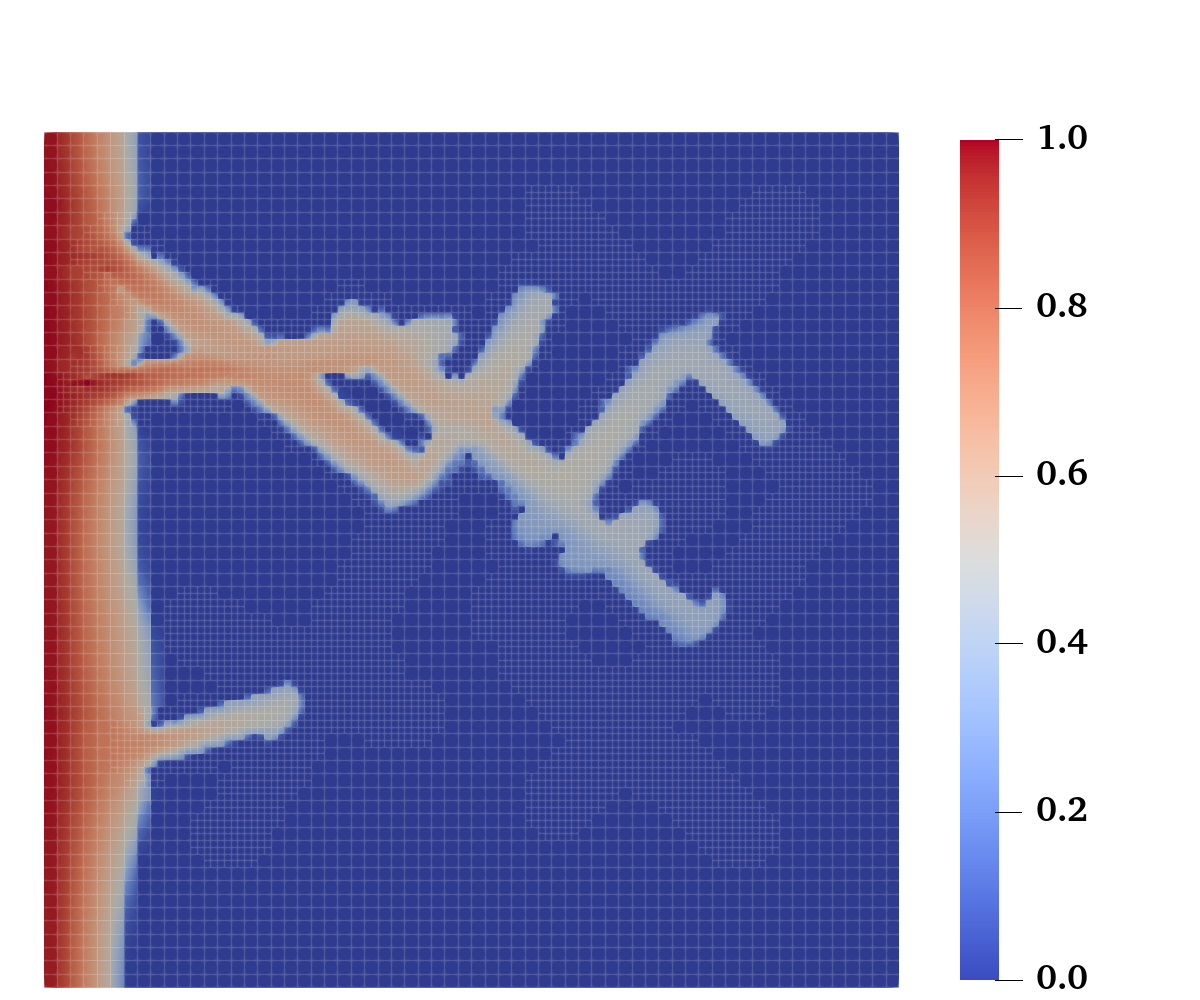}
        \caption{$\SOFD$ with $\delta_{i}\geq\SI{1e-5}{}$; $t = 10$}
    \end{subfigure}
    % \vspace{0.5cm}
    \hfill
    \begin{subfigure}{0.3\linewidth}
    \centering
        \includegraphics[trim = 0 0 100 120,clip,width=\linewidth]{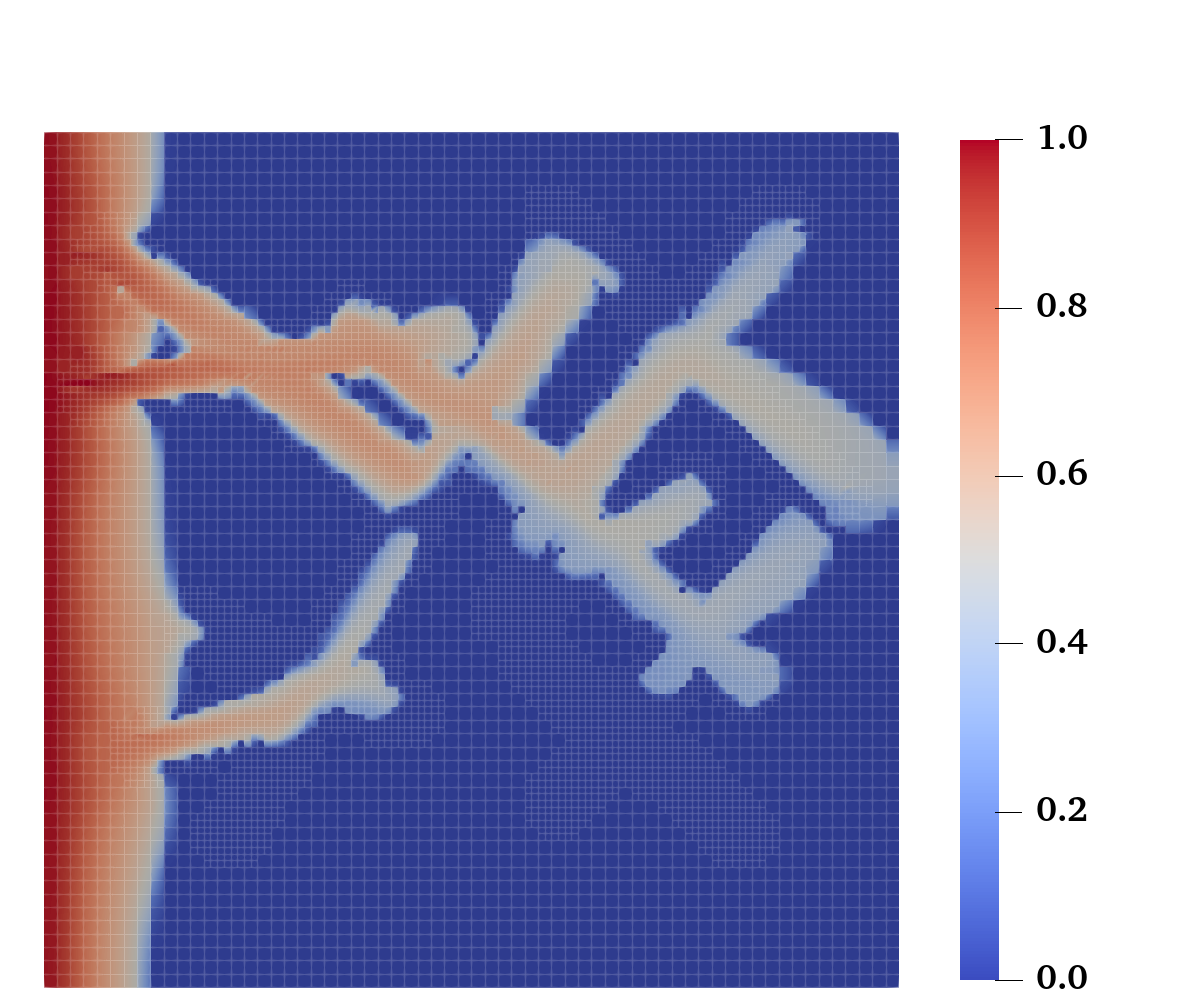}
        \caption{$\SOFD$ with $\delta_{i}\geq\SI{1e-5}{}$; $t = 15$}
    \end{subfigure}

    \begin{subfigure}{0.3\linewidth}
    \centering
        \includegraphics[trim = 0 0 100 120,clip,width=\linewidth]{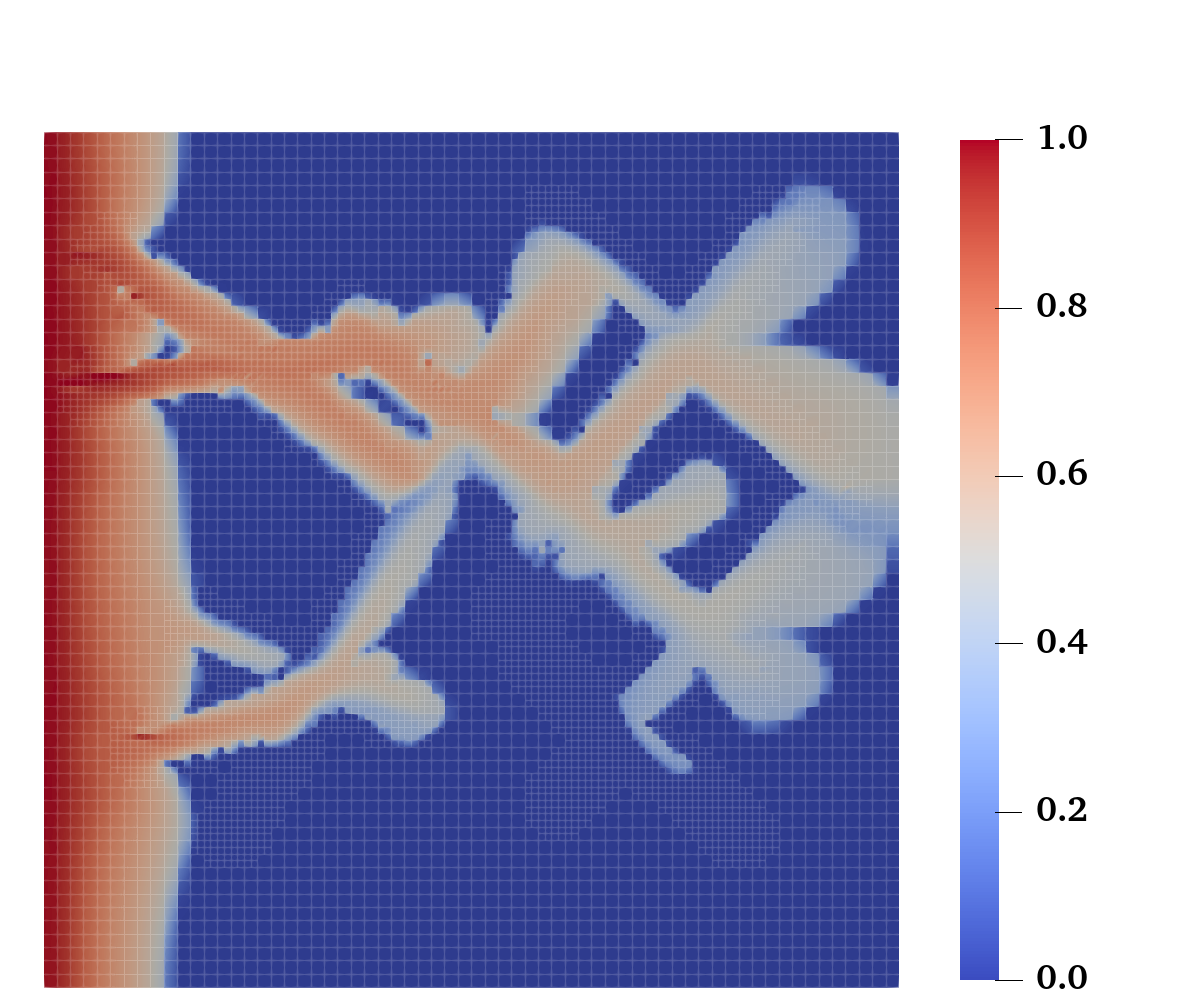}
        \caption{$\SOFD$ with $\delta_{i}\geq\SI{1e-5}{}$; $t = 20$}
    \end{subfigure}
    % \hfill
    \hspace{1.5cm}
    \begin{subfigure}{0.3\linewidth}
    \centering
        \includegraphics[trim = 0 0 100 120,clip,width=\linewidth]{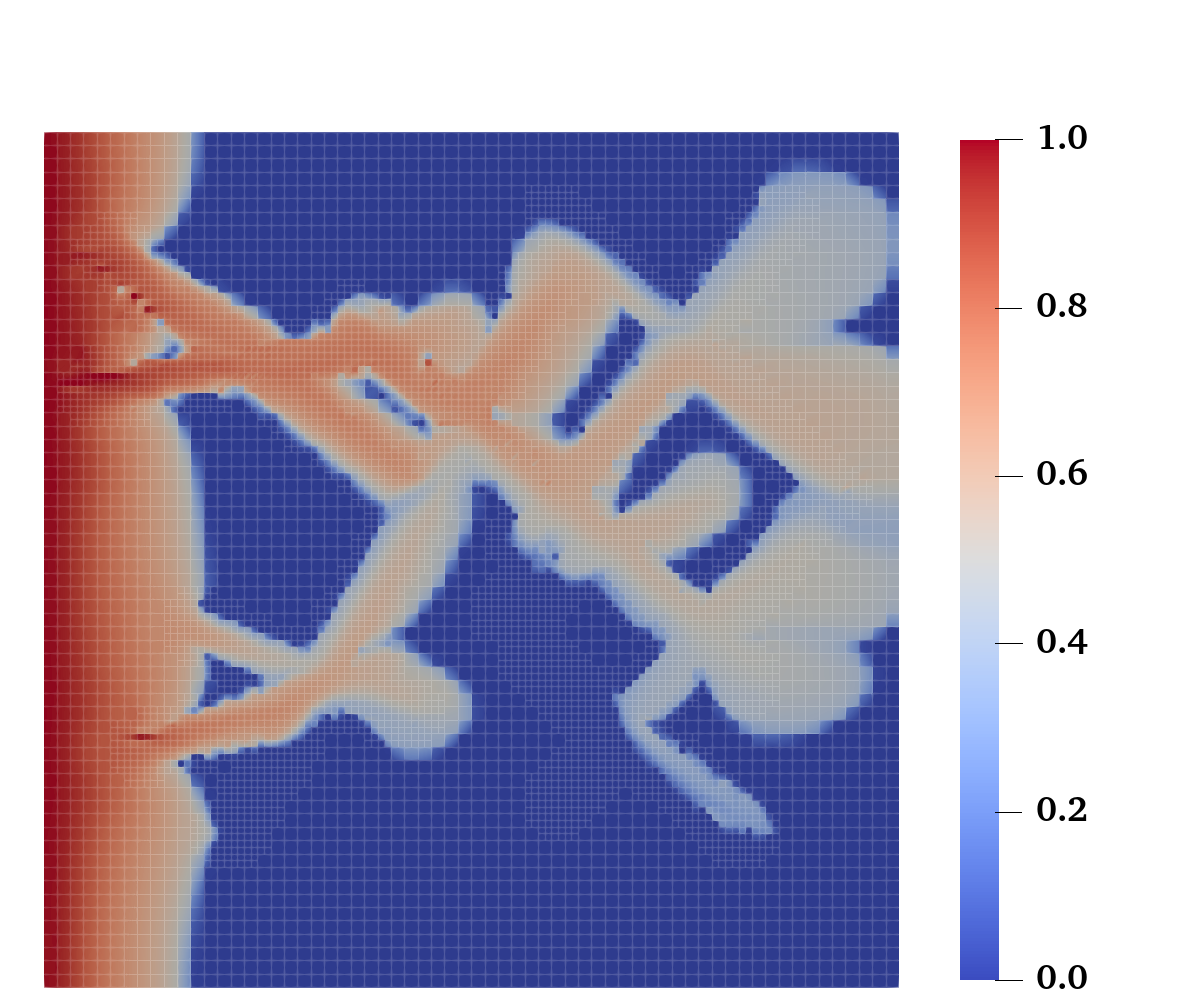}
        \caption{$\SOFD$ with $\delta_{i}\geq\SI{1e-5}{}$; $t = 25$}
    \end{subfigure}
    % \hfill
    % \begin{subfigure}{0.3\linewidth}
    % \centering
    %     \includegraphics[trim = 0 0 100 120,clip,width=\linewidth]{ex4-detector-2999.png}
    %     \caption{$\SOFD$ with $\delta_{lm}\geq\SI{1e-4}{}$; $t = 30$}
    % \end{subfigure}
    \caption{(Example~3) $\SOFD$ at time $t = 5, 10, 15, 20$, and $25$.}
    \label{ex4-detector}
\end{figure}

\section{Conclusion} \label{conclusion}
In this paper, we introduce a bound-preserving anti-oscillation postprocessing procedure for DG discretization of two-phase porous media flow equations with the IMPES method. 
To the best of our knowledge, this is the first application of the optimization-based cell average limiter to two-phase porous media flow, and the first extension of this framework to non-matching meshes. 
The optimization-based limiter modifies cell averages to satisfy the prescribed bounds and preserve global mass by minimizing the $L^2$ distance to the original cell averages. 
For non-matching meshes, we establish nonexpansiveness of the cell average limiter at the matrix-vector level. We employ the DY splitting method to efficiently compute the minimizer. 
To suppress spurious oscillations, a minmod limiter is applied, combined with a troubled-cell detector that avoids unnecessary numerical diffusion by limiting only troubled cells. The resulting framework is simple to implement, requiring only three core postprocessing operations, and is well-suited for large-scale simulations. 
Numerical experiments confirm that the proposed method preserves optimal convergence rates, enforces conservation and physical bounds, produces sharp saturation fronts, and performs robustly on both homogeneous and non-homogeneous porous media, including fractured domains.

\section*{Acknowledgments}
Sanghyun Lee was partially supported by the U.S. National Science Foundation under Grant 
DMS-2208402 and the U.S. Department of Energy, Office of Science, Energy Earthshots Initiatives under Award Number DE-SC-0024703. The work of Seungmin Lee is supported by the U.S. Department of Energy, Office of Science, Energy Earthshots Initiatives under Award Number DE-SC-0024703. Chen Liu is partially supported by the U.S. National Science Foundation under Grant DMS-2513106.

\bibliographystyle{cas-model2-names}
\bibliography{ref}

\end{document}